\documentclass[10pt]{amsart}
\usepackage{xcolor}
\usepackage{amsfonts}
\usepackage{amsmath, amsfonts, amsthm, amssymb, amscd}
\usepackage[mathcal]{euscript}
\usepackage{mathrsfs}
\usepackage{ulem}
\usepackage{mathabx,a4wide}
 
\newtheorem{theorem}{Theorem}[section]
\newtheorem{proposition}[theorem]{Proposition}
\newtheorem{lemma}[theorem]{Lemma}

\numberwithin{equation}{section}
\newcommand \newperp  {\underline{\partial}_\perp}
\newcommand \sbar {{\bar s}}
\newcommand \bel {\be \label}
\newcommand \Lcal {\mathcal L}
\newcommand \phib {\overline \phi}
\newcommand \xb {\overline {x}}
\newcommand \delb {\overline {\del}}
\newcommand \hb{\overline h}
\newcommand \Hb{\overline H}
\newcommand \Tb {\overline {T}}
\newcommand \minb{\overline{m}}
\newcommand \Cbar {\overline C}
\newcommand \Phib{\overline{\Phi}}
\newcommand \Psib{\overline{\Psi}}
\newcommand \Kcal {\mathcal K}
\newcommand \Hcal {\mathcal H}
\newcommand \Boxt {\widetilde {\Box}}
\newcommand \del \partial
\newcommand \delu {\underline{\del}}
\newcommand \Tu {\underline{T}}

\newcommand {\Gammau} {\underline M}
\newcommand {\Thetau}{\underline{\Theta}}
\newcommand {\thetau}{\underline{\theta}}

\newcommand {\minu}{\underline{m}}
\newcommand \RR{\mathbb{R}}
\newcommand {\vep}{\varepsilon}

\newcommand {\gb}{\overline{g}}
\newcommand {\Mb}{\overline {M}}
\newcommand {\Rb}{\overline{R}}
\newcommand {\nablab}{\overline{\nabla}}

\let\oldmarginpar\marginpar
\renewcommand\marginpar[1]{\-\oldmarginpar[\raggedleft\footnotesize #1]%
{\raggedright\footnotesize #1}}
\newcommand \eps {\epsilon}
\newcommand \ih {\widehat {\imath}}
\newcommand \ic {\widecheck{\imath}}
\newcommand \jh {\widehat {\jmath}}
\newcommand \jc {\widecheck{\jmath}}

\newcommand \kc {\widecheck{k}}
\newcommand \alphar {{{\alpha}}}
\newcommand \betar {{{\beta}}}
\newcommand \Hf {{\textbf H}}
\newcommand \be {\begin{equation}}
\newcommand \ee {\end{equation}}

\begin{document}

\title[The nonlinear stability of Minkowski space for massive fields]{The global nonlinear stability of Minkowski space
\\ 
for self-gravitating massive fields.
\\ 
\Large 
T\MakeLowercase{he} W\MakeLowercase{ave-}K\MakeLowercase{lein}-G\MakeLowercase{ordon} M\MakeLowercase{odel}}


\author[Philippe G. L{\smaller e}FLOCH and Yue MA]{Philippe G. L{\scriptsize e}Floch and Yue Ma}

\address{PLF: Laboratoire Jacques-Louis Lions, Centre National de la Recherche Scientifique,
Universit\'e Pierre et Marie Curie (Paris VI), 4 Place Jussieu, 75252 Paris, France. 
\newline 
E-mail:{\tt contact@philippelefloch.org} 
\newline 
\newline 
YM:  School of Mathematics and Statistics, Xi'an Jiaotong University, Xi'an, 710049 Shaanxi, P.R. China. Email: {\tt yuemath@xjtu.edu.cn}
}


\date{} 

\ 

\begin{abstract} The Hyperboloidal Foliation Method (introduced by the authors in 2014) is extended here and applied to the Einstein equations of general relativity. Specifically, we establish the nonlinear stability of Minkowski spacetime for self-gravitating massive scalar fields, while existing methods only apply to massless scalar  fields. First of all, by analyzing the structure of the Einstein equations in wave coordinates, we exhibit a nonlinear wave-Klein-Gordon model defined on a curved background, which is the focus of the present paper. For this model, we prove here the existence of global-in-time solutions to the Cauchy problem, when the initial data have sufficiently small Sobolev norms. A major difficulty comes from the fact that the class of conformal Killing fields of Minkowski space is significantly reduced in presence of a massive scalar field, since the scaling vector field is not conformal Killing for the Klein-Gordon operator. Our method relies on the foliation (of the interior of the light cone) of Minkowski spacetime by hyperboloidal hypersurfaces and uses Lorentz-invariant energy norms. We introduce a frame of vector fields adapted to the hyperboloidal foliation and we establish several key properties: Sobolev and Hardy-type inequalities on hyperboloids, as well as sup-norm estimates which correspond to the sharp time decay for the wave and the Klein-Gordon equations. These estimates allow us to control interaction terms associated with the curved geometry and the massive field, by distinguishing between two levels of regularity and energy growth and by a successive use of our key estimates in order to close a bootstrap argument. 
\end{abstract}
\maketitle
 
\footnotetext[1]{Published in: {\it Communications in Mathematical Physics} (2016).} 
\setcounter{tocdepth}{1}
\tableofcontents



\section{Introduction}

\subsection{The global existence problem}

In this paper and its companion \cite{PLF-YM-two}, we study the global-in-time existence problem for small amplitude solutions to nonlinear wave equations, with a two-fold objective:

\begin{itemize}

\item First, we provide a significant extension of the Hyperboloidal Foliation Method, recently proposed by the authors \cite{PLF-MY-book}. This method is based on a foliation of the interior of the future  light cone by hyperboloidal hypersurfaces and on Sobolev and Hardy inequalities adapted to this foliation.  
This method takes its root in work by Klainerman \cite{Klainerman85} and, later on, H\"ormander \cite{Hormander} concerning the standard Klein-Gordon equation. In comparison to our earlier theory in \cite{PLF-MY-book}, we are now able to encompass a much broader class of coupled wave-Klein-Gordon systems.

\item Our second objective is to apply this method to the Einstein equations of general relativity and arrive at a new approach for proving the nonlinear stability of Minkowski spacetime. Our method covers self-gravitating {\sl massive} scalar fields (as will be presented in full details in \cite{PLF-YM-two}), while earlier works were restricted to vacuum spacetimes or to spacetimes with massless scalar fields;
cf.~Christodoulou and Klainerman \cite{CK}, and Lindblad and Rodnianski \cite{LR1,LR2}, as well as Bieri and Zipser \cite{BieriZipser}.

\end{itemize}

\noindent The problem of the global dynamics of self-gravitating massive fields had remained open until now. The presence of a mass term poses a major challenge in order to establish a global existence theory for the Einstein equations (and construct future geodesically complete spacetimes). Namely, the class of conformal Killing fields of Minkowski spacetime
is reduced in presence of a massive scalar field, since the so-called scaling vector field is no longer conformal Killing and, therefore, cannot be used in implementing Klainerman's vector field method \cite{Klainerman80,Klainerman85}. 

In suitably chosen coordinates, the Einstein equations take the form of a coupled system of nonlinear wave-Klein-Gordon equations. More precisely, as in \cite{LR2}, we introduce wave coordinates, also called harmonic or De Donder gauge \cite{Asanov}, which allows one to exhibit the (quasi-null, see below) structure of the Einstein equations. 
The Hyperboloidal Foliation Method \cite{PLF-MY-book} was introduced precisely to handle such systems.
Yet, due to the presence of metric-related terms in the system under consideration, an important generalization is required before we can tackle the Einstein equations. Proposing such a generalization is our main purpose in the present paper.

By imposing asymptotically flat initial data on a spacelike hypersurface with sufficiently small ADM mass, one can first solve the Cauchy problem for the Einstein equations within a neighborhood  of this hypersurface 
(see  \cite{PLF-MY-book} for a sketch of the argument\footnote{The time of existence of the solution can be made arbitrarily large for compactly supported initial data with sufficiently small norm, so that this neighborhood does contain a hyperboloidal hypersurface.}) 
and, next, formulate the Cauchy problem when the initial data are posed on a hyperboloidal hyperspace or, alternatively, on a hyperboloid for the flat Minkowski metric after introducing suitable coordinates. In fact, the hyperboloidal Cauchy problem is, both, geometrically and physically natural.
 More precisely, let us consider Minkowski spacetime in standard Cartesian coordinates $(t, x^1, x^2, x^3)$ and observe that points on a hypersurface of constant time $t$ cannot be connected by a timelike curve, while points on a hyperboloid can be connected by such curves. Hence, hyperboloidal initial data can be ``physically prepared", while data on standard flat hypersurfaces cannot. An alternative standpoint would be to pose the Cauchy problem on a light cone, but while it is physically appealing and the Cauchy problem on a light cone has not been proven to be convenient for global analysis.

We emphasize that hyperboloidal foliations were used by Friedrich \cite{Friedrich81,Friedrich83} in order to establish the stability of Minkowski space in the future of a hyperboloidal hypersurface.  
Hyperboloidal foliations have also been found to be very efficient in numerical computations \cite{Frauendiener,MoncriefRinne,Rinne,Zenginoglu}. 

As was demonstrated in \cite{PLF-MY-book} for a rather general class of nonlinear wave equations, analyzing the global existence problem is quite natural in the hyperboloidal foliation of Minkowski spacetime and, importantly, lead to {\sl uniform bounds} on the energy of the solutions. Before proceeding with further details, let us summarize the main features of the method we propose:

\begin{itemize}

\item {\bf Lorentz invariance.} We rely on the foliation of Minkowski space by hyperboloids (defined as the level sets of constant Lorentzian distance from some origin), so that the fundamental energy of the wave-Klein-Gordon equations 
remains invariant under Lorentz transformations of Minkowski spacetime. Observe that in our construction, all the hyperboloids are asymptotic to the same limiting cone and approach the same sphere at infinity. (In particular, no energy can escape through null infinity.) 

\item {\bf Smaller set of Killing fields.} We avoid using the scaling vector field $S := r\del_r + t\del_t$, which is the key in order to handle Klein-Gordon equations and cover the Einstein-matter system when the evolution equation for the matter is {\sl not conformally invariant}.

\item {\bf Sharp rate of time decay.}
In order to control source-terms related to the curved geometry,
we establish sharp pointwise bounds for solutions to wave equations and Klein-Gordon equations with source-terms.

\end{itemize}

In the rest of this introduction, we explain how to derive, from the Einstein equations, a {\sl model problem} which will be our main focus in the present paper.


\subsection{Einstein equations for massive scalar fields}

We thus consider the Einstein equations for an unknown spacetime $(M,g)$:
\be
\label{eq main geo 0}
R_{\alpha\beta} - {R \over 2} g_{\alpha\beta} = 8 \pi \, T_{\alpha\beta},
\ee
where $R_{\alpha\beta}$ denotes\footnote{Throughout, Greek indices $\alpha,\beta,\gamma$ take values $0,1,2,3$ and Einstein convention is used.}
the Ricci curvature tensor  
and $R= g^{\alpha\beta} R_{\alpha\beta}$ denotes the scalar curvature.
The matter is taken to be a massive scalar field with potential $V=V(\phi)$ and stress-energy tensor
\be
\label{eq tensor T}
T_{\alpha\beta} := \nabla_\alpha \phi \nabla_\beta \phi - \Big( {1 \over 2} \nabla_\gamma \phi \nabla^\gamma \phi + V(\phi) \Big) g_{\alpha\beta}
\ee
and, specifically,  
\be
V(\phi) := \frac{c^2}{2}\phi^2,
\ee
where $c^2>0$ represents the mass of the scalar field.
By applying $\nabla^\alpha$ to \eqref{eq tensor T} and using the Bianchi identity 
$$
\nabla^\alpha \big( R_{\alpha\beta} - (R/2) g_{\alpha\beta} \big) =0, 
$$
we easily check that the Einstein--scalar field system implies 
\begin{subequations}
\label{eq main geo}
\be
\label{eq main geo a}
R_{\alpha\beta} = 8\pi\big(\nabla_\alpha\phi\nabla_\beta\phi + V(\phi) \, g_{\alpha\beta}\big),
\ee
\bel{eq main geo b}
\Box_g\phi = V'(\phi) = c^2 \phi.
\ee
\end{subequations}

The Cauchy problem for the Einstein-scalar field equations is formulated as follows \cite{CB}. An initial data set consists of a Riemannian three-manifold $(\Mb,\gb)$, a symmetric two-tensor $K$ defined on $\Mb$, and
two scalar fields $(\phib_0, \phib_1)$ defined on $\Mb$.
We then seek for a $(3+1)$-dimensional Lorentzian manifold $(M,g)$ satisfying the following properties:
\begin{itemize}

\item There exists an embedding $i: \Mb\to M$ such that the induced metric $i^*(g)$ coincides with $\gb$, while  the second fundamental form of $i(\Mb) \subset M$ coincides with the prescribed two-tensor $K$. 

\item The restriction of $\phi$ and $\Lcal_\nu \phi$ to $i(\Mb)$ coincides with the data $\phi_0$ and $\phi_1$ respectively, where $\nu$ denotes the (future-oriented) unit normal to $i(\Mb) \subset M$.

\item Moreover, the manifold $(M,g)$ satisfies the Einstein equations \eqref{eq main geo}.
\end{itemize}

More precisely, one seeks for a {\sl globally hyperbolic development} of the given initial data, that is, a Lorentzian manifold such that every time-like geodesic extends toward the past direction in order to meet the initial hypersurface $\Mb$. Furthermore, a notion of {\sl maximal development} was defined by Choquet-Bruhat and Geroch \cite{ChoquetBruhatGeroch,CB} and such a development was shown to exist for a large class of matter models. The maximal development need not be future geodesically complete, and a main challenge in the field of mathematical general relativity is the construction of classes of future geodesically complete spacetimes. 

Furthermore, it should be emphasized that, in order to fulfill the equations \eqref{eq main geo}, the initial data set $(\Mb, \gb, K)$ cannot be arbitrary and must satisfy Einstein's constraint equations:
\bel{eq constraint}
\aligned
&  \Rb + K_{ij} \, K^{ij} - (K_i^i)^2   = 8\pi T_{00},
\\
& \nablab^i K_{ij} - \nablab_j K_l^l  = 8 \pi T_{0j},
\endaligned
\ee
where $\Rb$ is the scalar curvature of the metric $\gb$ and $\nablab$ denotes its Levi-Civita connection, and the terms $T_{00}$ and $T_{0i}$ are determined from the data $\phib_0$ and $\phib_1$.

Minkowski spacetime provides one with a trivial solution to the Einstein equations, which satisfies the Cauchy problem associated with the initial data $(\Mb,\gb,K,\phib_0,\phib_1)$ when $\Mb = \RR^3$ is endowed with  the standard Euclidian metric and $K \equiv 0$, while the matter terms vanish identically $\phib_0=\phib_1 \equiv 0$.  The question we address in the present paper is whether this solution is dynamically stable under small perturbations of the initial data. More precisely, given an initial data set $(\Mb, \gb, K, \phib_0, \phib_1)$ such that $\Mb$ is diffeomorphic to $\RR^3$, $\gb$ is close to the flat metric and $K, \phib_0, \phib_1$ are sufficiently small,
does the associated solution $(M,g)$ to the Einstein-massive scalar field system remain
close to the flat Minkowski spacetime $\RR^{1+3}$ ? 

Clearly, this nonlinear stability problem is of fundamental importance in physics. It is expected that Minkowski spacetime is the ground state state of the theory with the lowest possible energy. As far as massless scalar fields are concerned,  the nonlinear stability of Minkowski spacetime was indeed established in Christodoulou and Klainerman's pioneering work \cite{CK}. In the  present work (including \cite{PLF-YM-two}), we solve this question for {\sl massive} scalar fields.


\subsection{Einstein-scalar field equations in wave coordinates}\label{subsec intro 3} 

Our first task is to express the field equations \eqref{eq main geo} in a well-chosen coordinate system and then derive our wave-Klein-Gordon model problem. We follow \cite{CB,LR1} and work in wave coordinates satisfying, by definition, 
\be
\Box_g x^\alpha = 0. 
\ee 
We postpone to \cite{PLF-YM-two} the details of the derivation and directly write the formulation of the Einstein-massive scalar field equations in wave coordinates:
\begin{subequations}\label{eq main wavecoordinate}
\be
\Boxt_g g_{\alpha\beta}  = Q_{\alpha\beta}(g;\del g,\del g) + P_{\alpha\beta}(g;\del g,\del g)
- 16\pi\big(\del_\alpha\phi\del_\beta\phi + V(\phi)g_{\alpha\beta}\big),
\ee
\be
\Boxt_g \phi - V'(\phi) = 0,
\ee
\end{subequations}
where $\Boxt_g := g^{\alpha\beta} \del_{\alpha} \del_{\beta}$ is referred to as the (reduced) wave operator in curved space. In \eqref{eq main wavecoordinate}, we distinguish between several types of nonlinearity:

\begin{itemize}

\item {\bf Null terms.} 
The quadratic terms 
\be
\aligned
Q_{\alpha\beta}
:= \, &
  g^{\lambda\lambda'}g^{\delta\delta'}\del_{\delta}g_{\alpha\lambda'}\del_{\delta'}g_{\beta\lambda}
-g^{\lambda\lambda'}g^{\delta\delta'}\big
(\del_{\delta}g_{\alpha\lambda'}\del_{\lambda}g_{\beta\delta'} - \del_{\delta}g_{\beta\delta'}\del_{\lambda}g_{\alpha\lambda'}\big)
\\
&+g^{\lambda\lambda'}g^{\delta\delta'}
\big(\del_\alpha g_{\lambda'\delta'}\del_{\delta}g_{\lambda\beta} - \del_\alpha g_{\lambda\beta}\del_{\delta}g_{\lambda'\delta'}\big)
+\frac{1}{2}g^{\lambda\lambda'}g^{\delta\delta'}
\big(\del_\alpha g_{\lambda\beta}\del_{\lambda'}g_{\delta\delta'} - \del_\alpha g_{\delta\delta'}\del_{\lambda'}g_{\lambda\beta}\big)
\\
&+g^{\lambda\lambda'}g^{\delta\delta'}
\big(\del_\beta g_{\lambda'\delta'}\del_{\delta}g_{\lambda\alpha} - \del_\beta g_{\lambda\alpha}\del_{\delta}g_{\lambda'\delta'}\big)+\frac{1}{2}g^{\lambda\lambda'}g^{\delta\delta'}
\big(\del_\beta g_{\lambda\alpha}\del_{\lambda'}g_{\delta\delta'} - \del_\beta g_{\delta\delta'}\del_{\lambda'}g_{\lambda\alpha}\big)
\endaligned
\ee
are standard null forms with cubic corrections. Their treatment in a global existence proof
is a now classical matter and, in particular, are already dealt with by standard methods.  

\item  {\bf Quasi-null terms.} 
The quadratic terms 
\be
P_{\alpha\beta} 
: = - \frac{1}{2}g^{\lambda\lambda'}g^{\delta\delta'}\del_\alpha g_{\delta\lambda'}\del_\beta g_{\lambda\delta'}
+\frac{1}{4}g^{\delta\delta'}g^{\lambda\lambda'}\del_\beta g_{\delta\delta'}\del_\alpha g_{\lambda\lambda'}
\ee
are referred to as ``weak null" terms in \cite{LR1}, but we prefer to propose the new terminology 
{\sl ``quasi-null terms''.} As first noted in \cite{LR1}, quasi-null terms are found to be analogous to standard null terms, 
{\sl provided} the tensorial structure of the Einstein equations and the wave coordinate condition are carefully taken into account.

\item {\bf Curved metric terms.}
Setting now 
\be
h^{\alpha\beta} := g^{\alpha\beta} - m^{\alpha\beta}, 
\qquad 
h_{\alpha\beta} := m_{\alpha\beta} - g_{\alpha\beta}
\ee
and considering the term $\Boxt_g g_{\alpha\beta}$, we see that we must also treat the quasi-linear terms
$$
h^{\alpha'\beta'}\del_{\alpha'}\del_{\beta'}h_{\alpha\beta},
\qquad 
h^{\alpha'\beta'}\del_{\alpha'}\del_{\beta'}\phi.
$$
We will deal with these metric-related terms by the following two approaches:

\begin{itemize}

\item First, thanks to the wave coordinate condition, we can assume that $h^{\alpha\beta}$ behaves essentially like a null quadratic form and consider, therefore, that $h^{\alpha\beta}$ is null. More precisely, the first term
$$
h^{\alpha'\beta'}\del_{\alpha'}\del_{\beta'}h_{\alpha\beta},
\qquad h^{\alpha\beta}\text{ being a null form}
$$
can be treated by our arguments in \cite{PLF-MY-book}. 

\item The second quasi-linear term $h^{\alpha'\beta'}\del_{\alpha'}\del_{\beta'}\phi$ (without necessarily imposing the null condition) requires our new technique which is presented in this paper and is based on sharp sup-norm bounds for solutions to
wave equations and Klein-Gordon equations. 
\end{itemize}
\end{itemize}

Our aim is presenting first in a simplifed form several arguments that will be required to cope with the full system of Einstein equations in  \cite{PLF-YM-two}. In order to derive here a model problem, we proceed by removing (from \eqref{eq main wavecoordinate}): 
\begin{itemize}

\item the null terms $Q_{\alpha\beta}$
(which are handled in \cite{PLF-MY-book}), 

\item the quasi-null terms $P_{\alpha\beta}$ (postponed to \cite{PLF-YM-two}, where the structure of the Einstein equations and the wave coordinate condition will be discussed), and

\item the quasi-linear terms $h^{\alpha'\beta'}\del_{\alpha'}\del_{\beta'}h_{\alpha\beta}$
(to be treated by the wave coordinate condition and, in turn, the method already presented in \cite{PLF-MY-book}). 

\end{itemize} 
These formal simplifications, therefore, lead us to the model\footnote{Our convention for the wave operator is $\Box := -\del_t\del_t + \sum_a\del_a\del_a$.}
$$
\aligned
&\Box h_{\alpha\beta} = \del_\alpha\phi\del_\beta\phi + m_{\alpha\beta}V(\phi),
\\
&\Box \phi = H^{\alpha\beta}(h)\del_\alpha\del_\beta\phi + V'(\phi),
\endaligned
$$
with unknowns $h_{\alpha\beta}, \phi$ defined on Minkowski space,
where $H^{\alpha\beta}(h)$ can be assumed to depend linearly on $h_{\alpha\beta}$. We are primarily interested in the potential $V(\phi) = \frac{c^2}{2}\phi^2$ and, therefore after changing the notation, we arrive at the following system of two coupled equations: 
$$
\aligned
&-\Box u = P^{\alpha\beta}\del_\alpha v\del_\beta v + Rv^2,
\\
&-\Box v + H^{\alpha\beta} u\del_\alpha\del_\beta v + c^2 v = 0, 
\endaligned
$$
where $u, v$ are two scalar unknowns and $P^{\alpha\beta}, H^{\alpha\beta}, R, c$ are given constants (and only the 
obvious positivity condition $c^2>0$ is relevant).


\subsection{Analysis on the model problem}

As illustrated by the derivation above, in order to deal with the Einstein-massive scalar field equations, 
we must weaken a key assumption made in \cite{PLF-MY-book} and, as we will see, cope with wave equations posed on a curved space for which the Minkowski metric need not represent the underlying geometry in a sufficiently accurate manner. Namely, we recall that, in the notation of \cite[Section~1]{PLF-MY-book}, interaction terms like $u_{\ih}\del\del v_{\jc}$ involving components $u_{\ih}$ of wave equations and component $v_{\jc}$ of Klein-Gordon equations were not included in our theory. The same restriction was also assumed in a pioneering work by Katayama \cite{Katayama12a,Katayama12b} 
on wave-Klein-Gordon equations. In the present paper, we overcome this challenging difficulty and extend our earlier analysis (of the system (1.2.1) in \cite{PLF-MY-book} by now removing the condition (1.2.4e) therein). 

To this end, in the present paper, we derive and take advantage of two pointwise estimates:

\begin{itemize}

\item {\bf A sharp sup-norm estimate for solutions to the wave equation} in Minkowski space with source-term, as stated in Theorem~\ref{Linfini wave}, below. Suitable decay is assumed on the source-term, as is relevant for our analysis, and the proof is based on the solution formula available for the wave equation in flat space.

\item {\bf A sharp sup-norm estimate for solutions to the Klein-Gordon equation} in curved space in $(3+1)$-dimensions (as stated in Theorem~\ref{Linfini KG}, below). Our estimate is motivated by a pioneering work by Klainerman \cite{Klainerman85} on the global existence problem for small amplitude solutions to nonlinear Klein-Gordon equations in four spacetime dimensions.

\end{itemize}
\noindent  Note that an estimate as above could also be derived in $(2+1)$-dimensions with different rates \cite{Ma}. Global existence results for nonlinear Klein-Gordon equations were also established by Shatah in the pioneering work \cite{Shatah}. 
Klein-Gordon systems have received a lot of attention in the literature and we can, for instance, refer to \cite{Bachelot88,Bachelot94,Delort01,Delort04,Hormander,Klainerman85} and the references therein.

For clarity in the presentation, we do not treat the most general class of systems, but based on our formal derivation from the Einstein equations, we now study the following {\bf wave-Klein-Gordon model}:
\bel{eq main}
\aligned
&-\Box u = P^{\alpha\beta}\del_\alpha v\del_\beta v + Rv^2,
\\
&-\Box v + u \, H^{\alpha\beta} \del_\alpha\del_\beta v + c^2 v = 0,
\endaligned
\ee
with unknowns $u, v$ 
posed on Minkowski space $\RR^{3+1}$ and prescribed initial data\footnote{For convenience in the following proof and without loss of generality, we prescribe data at time $t=2$.} $u_0, u_1, v_0, v_1$ posed on the spacelike hypersurface $t=2$: 
\bel{initialdata}
\aligned
& u|_{t=2} = u_0, \quad &&\del_t u|_{t=2} = u_1,
\\
& v|_{t=2} = v_0, \qquad &&\del_t v|_{t=2} = v_1. 
\endaligned
\ee 
Here, $P^{\alpha\beta}, R, H^{\alpha\beta}, c$ are given constants, and the initial data are sufficiently smooth functions that are compactly
supported in the unit ball $\{ (x_1)^2 + (x_2)^2 + (x_3)^2 < 1\}$ with $x=(x_1, x_2, x_3) \in \RR^3$. 

We emphasize that, according to our analysis in Section~\ref{subsec intro 3}, \eqref{eq main} includes the essential difficulty arising in the Einstein-massive field system. Note in passing that, in \eqref{eq main}, 
there is no such term like $R u^2$ which would imply finite time blow-up (as first pointed out by John \cite{John}). 
Our main result in the present paper is as follows.

\begin{theorem}[Global existence theory for the wave-Klein-Gordon model]
\label{thm main}
Consider the nonlinear wave-Klein-Gordon system \eqref{eq main} for some given parameter values 
$P^{\alpha\beta}, R, H^{\alpha\beta}$ and $c>0$. 
Given any integer $N \geq 8$, there exists a positive constant $\vep_0=\vep_0(N)>0$ such that if the initial data satisfy
\bel{eq main initial}
\aligned
&\| (u_0, v_0) \|_{H^{N+1}(\RR^3)} + \| (u_1, v_1) \|_{H^N(\RR^3)} < \vep_0,
\endaligned
\ee
then the Cauchy problem  \eqref{eq main}-\eqref{initialdata} admits a global-in-time solution.
\end{theorem}

As done in \cite{PLF-YM-two}, the Cauchy problem can be reformulated with initial data prescribed on a hyperboloid and the smallness condition \eqref{eq main initial} leads to a similar smallness condition for the hyperboloidal initial data. As already pointed out in \cite{PLF-MY-book}, the presence of the quasi-linear term $u \, H^{\alpha\beta}\del_\alpha\del_\beta v$ may possibly change the asymptotic behavior of solutions for large times. In fact, our proof will only show that the lower-order energy of the wave component remains globally bounded for all times, while
the high-order energy of the wave component $u$ and the lower-order energy of the Klein-Gordon component $v$ could in principle grow at the rate $t^\delta$ for some (small) $\delta>0$. On the other hand, the higher-order energy associated with the Klein-Gordon component $v$ 
may significantly increase at the rate $t^{\delta+1/2}$ for some (small) $\delta>0$.

The proof of Theorem~\ref{thm main} will occupy the rest of this paper which we outline as follows:

\begin{itemize}

\item Proceeding with a bootstrap argument, we assume that, within some hyperbolic time interval, the hyperboloidal energy of suitable derivatives of the unknowns (up to a certain order) satisfy a set of bounds.

\item Our assumptions use {\sl two levels of regularity} and distinguish between the behavior of lower-order and higher-order energy norms, the low-order derivatives enjoying a much better control in time. Recall that, in \cite{PLF-MY-book}, we could already prove that the lower-order energy of the wave component is uniformly bounded in time,  
but 
the growth rate for  the high-order Klein Gordon energy was solely $t^{\delta}$.  

\item By Sobolev inequality (on hyperboloids), we can turn these $L^2$ type inequalities to a set of sup-norm estimates, which we refer to as {\sl basic decay estimates.} These decay estimates are not sharp enough in order to close our bootstrap argument.

\item Relying on these basic decay estimates, we establish {\sl refined decay estimates} by relying on two technical sup-norm estimates established below for wave equations and Klein-Gordon equations.

\item Equipped with these refined decay estimates, we are able to improve our initial assumptions and close the bootstrap argument.

\end{itemize}

Before we proceed with the details of the proof (which is rather long), the reader may find it useful to read through the following {\sl heuristic arguments} which rely on notations (only briefly explained here) to be rigorously introduced only later (in the course of the following three sections). Our proof proceeds with a bootstrap argument and considers the largest time interval $[2,s^*]$ (in the `hyperbolic time' $s$ defined as $s^2 = t^2 - r^2$) within which the following energy bounds hold: 
$$
\aligned
&E_m(s,\del^IL^J u)^{1/2}\leq C_1\vep s^{k\delta}, &&|J|=k, \quad |I|+|J|\leq N, \qquad
&&& \text{wave / high-order,}
\\
&E_m(s,\del^IL^J u)^{1/2}\leq C_1\vep, \quad && \hskip1.6cm  |I|+|J|\leq N-4, &&& \text{wave / low-order,}
\\
&E_{m,c^2}(s,\del^IL^J v)^{1/2}\leq C_1\vep s^{1/2+k\delta}, &&|J|=k, \quad |I|+|J|\leq N, &&& \text{Klein-Gordon / high-order,}
\\
&E_{m,c^2}(s,\del^IL^J v)^{1/2}\leq C_1\vep s^{k\delta}, &&|J|=k, \quad |I|+|J|\leq N-4 &&& \text{Klein-Gordon / low-order,}
\endaligned
$$
where $\eps, \delta, C_1$ are parameters. 
These bounds concern the energy of the wave component $u$ and the Klein-Gordon component $v$, and distinguish between low-order and high-order derivatives. We have denoted by $E_m$ the energy associated with the wave equation (for the flat metric $m$), while 
$\del^I$ are partial derivative operators and $L^J$ are combinations of Lorentz boosts (see below for details). 
The heart of our proof of Theorem~\ref{thm main} is proving that, by selecting a sufficiently large constant $C_1$ and sufficiently small $\vep, \delta>0$, the above energy bounds in fact imply the following improved energy bounds (obtained by replacing $C_1$ by $C_1/2$): 
$$
\aligned
&E_m(s,\del^IL^J u)^{1/2}\leq \frac{1}{2}C_1\vep s^{k\delta}, &&|J|=k, \quad |I|+|J|\leq N, \qquad
&&& \text{wave / high-order,}
\\
&E_m(s,\del^IL^J u)^{1/2}\leq \frac{1}{2}C_1\vep, \quad && \hskip1.6cm |I|+|J|\leq N-4, &&& \text{wave / low-order,}
\\
&E_{m,c^2}(s,\del^IL^J v)^{1/2}\leq \frac{1}{2}C_1\vep s^{1/2+k\delta}, &&|J|=k, \quad |I|+|J|\leq N, &&& \text{Klein-Gordon / high-order,}
\\
&E_{m,c^2}(s,\del^IL^J v)^{1/2}\leq \frac{1}{2}C_1\vep s^{k\delta}, &&|J|=k, \quad |I|+|J|\leq N-4 &&& \text{Klein-Gordon / low-order.}
\endaligned
$$
(Of course, it is then a standard matter to deduce from this property that, in fact, $s^*= +\infty$.)

To derive the improved energy bounds, we differentiate the equations \eqref{eq main} with $\del^IL^J$ with $|I|+|J|\leq N$: 
$$
\aligned
-\Box \del^I L^Ju 
& = \del^IL^J\left(P^{\alpha\beta}\del_{\alpha}v\del_{\beta}v\right) + \del^IL^J\left(Rv^2\right), 
\\
-\Box \del^IL^J v + u \, H^{\alpha\beta}\del^IL^J v+ c^2\del^IL^J v 
& = -[\del^IL^J, u \, H^{\alpha\beta}\del_{\alpha}\del_{\beta}]v. 
\endaligned
$$
For these differentiated equations, we perform energy estimates along the hyperboloidal foliation and we are led to seek for an integrable time decay for the following the three terms\footnote{See \eqref{add-form} for the notation.}:
\bel{eq:948}
\aligned
T_1^{I,J}(s):=& \big\|\del^IL^J\left(P^{\alpha\beta}\del_{\alpha}v\del_{\beta}v\right)\big\|_{L_f^2(\Hcal_s)},
\\
T_2^{I,J}(s) :=&  \big\|\del^IL^J\left(Rv^2\right)\big\|_{L_f^2(\Hcal_s)},
\\
T_3^{I,J}(s) :=& 
\big\|[\del^IL^J, u \, H^{\alpha\beta}\del_{\alpha}\del_{\beta}]v\big\|_{L_f^2(\Hcal_s)}.
\endaligned
\ee
For lower-order indices $|I|+|J|\leq N-4$, the terms $T_1^{I,J}(s)$ and $T_2^{I,J}(s)$ are easily controlled, since 
from the bootstrap assumption and the global Sobolev inequalities on hyperboloids we 
have (basic) decay estimates which lead to time-integrable bounds: 
\bel{eq:948A}
T_1^{I,J}(s) + T_2^{I,J}(s) \lesssim s^{-3/2+(k+2)\delta}, \qquad \text{ provided } |I|+|J|\leq N-4 \text{ with } |J| =k.
\ee
On the other hand, for higher-order derivatives these basic decay rates are not sufficient and we can not conclude directly. In addition, for the third term $T_3^{I,J}(s)$ (for arbitrary $|I|+|J|$), we also cannot conclude directly and we need  sharper pointwise decay. 

To overcome this challenge, we rely on our $L^{\infty}$--$L^{\infty}$ sharp decay estimates, established below in Proposition \ref{Linfini wave} (for the wave component) and Proposition \ref{Linfini KG} (for the Klein-Gordon component). These $L^{\infty}$--$L^{\infty}$ bounds allow us to improve the basic pointwise estimates, and we find (for all $|I|+|J|\leq N-4$):  
$$
\aligned
|L^Iu| & \lesssim C_1\vep \, t^{-1}s^{k\delta},
\\
|\del^IL^Jv| & \lesssim C_1\vep \, (s/t)^{2-7\delta}s^{-3/2+k\delta},
\\
|\del^IL^J \del_{\alpha}v| & \lesssim C_1\vep \, (s/t)^{1-7\delta}s^{-3/2+k\delta}.
\endaligned
$$
Returning to our bootstrap assumptions, we thus see that for all  $|I|+|J|\leq N$  
\bel{eq:948B}
\aligned
\left\|\del^IL^J\left(\del_{\alpha}v\del_{\beta}v\right)\right\|_{L_f^2(\Hcal_s)}
\simeq& \sum_{I_1+I_2=I\atop J_1+J_2=J}\|\del^{I_1}L^{J_1}\del_{\alpha}v \, \del^{I_2}L^{J_2}\del_{\beta}v\|_{L_f^2(\Hcal_s)}
\\
\lesssim&C_1\vep s^{-3/2} \|\del^{I_2}L^{J_2}\del_{\beta}v\|_{L_f^2(\Hcal_s)}\lesssim (C_1\vep)^2s^{-1+k\delta}
\endaligned
\ee
(by assuming, without loss of generality, $|I_1|+|J_1|\leq N-4$ in the above calculation). 
Similarly, we also obtain 
\bel{eq:948C}
\|\del^IL^J \left(v^2\right)\|_{L_f^2(\Hcal_s)}\lesssim (C_1\vep)^2s^{-1+k\delta}.
\ee
We thus succeed to uniformly control the terms $T_1^{I,J}(s)$ and $T_2^{I,J}(s)$ (for all relevant $I, J$), and this is already sufficient to conclude with the desired improved energy bounds for the {\sl wave component.} 

Dealing with the last term $T_3^{I,J}(s)$ arising in the equation of the Klein-Gordon component is more delicate. Observe that the commutator is a linear combination of the following three types of terms:
\bel{eq:liste} 
\aligned
& (\del^{I_1}L^{J_1}u) \del^{I_2}L^{L_2}\del_{\alpha}\del_{\beta}v, \qquad && I_1+I_2=I, \quad J_1+J_2=J,  \quad  |I_1|\geq 1, 
\\
& (L^{J_1'}u) \del^IL^{J_2'}\del_{\alpha}\del_{\beta}v, && J_1'+J_2'=J, \quad J_1'\geq 1, 
\\
&u\del_{\alpha}\del_{\beta}\del^IL^{J'}v, && J'\leq J-1. 
\endaligned
\ee
The first expression above is directly controled thanks to the available sharp decay estimate, while for the second and third ones and due to the presence of the term $L^Ju$, a refined decay estimates and a Hardy-type inequality (for the hyperboloidal foliation) must be used, as we now explain. 

Let us begin by discussing derivatives of higher-order and consider (for instance) the second type of terms in \eqref{eq:liste}: for all $|J_1'|\leq N-4$, we use the sharp decay bound $|L^Iu| \lesssim C_1\vep t^{-1}s^{k\delta}$ combined with the energy bound on $\del_{\alpha}\del_{\beta}v$ (implied by our bootstrap assumption). When $|I|+|J_2'|\leq N-4$, the sharp bound $|\del^IL^J \del_{\alpha}v| \lesssim C_1\vep \, (s/t)^{1-7\delta}s^{-3/2+k\delta}$ and Hardy's inequality are used. We thus find 
\bel{eq:948D}
\big\|[\del^IL^J,H^{\alpha\beta}\del_{\alpha}\del_{\beta}]v\big\|_{L_f^2(\Hcal_s)}\lesssim (C_1\vep)^2s^{-1/2+k\delta}.
\ee

Dealing with lower-order derivatives is easier and, again, we take the second type of terms in \eqref{eq:liste} as an example: for $|J_1'|\leq |I|+|J|\leq N-4$, we apply directly the sharp bound $|L^Iu| \lesssim C_1\vep \, t^{-1}s^{k\delta}$ and the energy bound given by our bootstrap assumption. This leads us to the stronger decay 
\bel{eq:948E}
\big\|[\del^IL^J,H^{\alpha\beta}\del_{\alpha}\del_{\beta}]v\big\|_{L_f^2(\Hcal_s)}\lesssim (C_1\vep)^2s^{-1+k\delta}.
\ee
In conclusion, in view of \eqref{eq:948A}--\eqref{eq:948E}, we can gain enough time decay for all of the terms arising in the evolution of our energy expressions and, therefore, the energy estimate on the hyperboloidal foliation leads us to the desired improved energy bounds.


\subsection{A general class of wave-Klein-Gordon systems}

The technique presented here applies immediately to a much broader class of systems. Indeed, it applies to the following system of wave--Klein-Gordon equations
\be
\label{main eq main}
\begin{cases}
\aligned
& \Box u_i + B^{j\alpha\beta} u_j \del_\alpha\del_\beta u_i 
= 
F_i(u,\del u, v, \del v) 
= P_i^{jk\alpha\beta} \del_\alpha u_j \del_\beta u_k + R_i v^2 + S_i^{\alpha\beta} \del_\alpha v \del_\beta v,
\\
& \Box v + B^{j\alpha\beta} u_j \del_\alpha \del_\beta v - c^2 v^2 = 0,
\endaligned
\\
\aligned
& w_i|_{t=2} =  {w_i}_0, \qquad \qquad && v|_{t=2} = v_0,
\\
& \del_t w_i|_{t=2} =  {w_i}_1,  && \del_t v|_{t=2} = v_1,
\endaligned
\end{cases}
\ee
with unknowns $u=(u_i)$ ($1\leq i \leq n$) and $v$ defined on Minkowski space $\RR^{3+1}$,
while ${w_i}_0, v_0,  {w_i}_1, v_1$ are prescribed initial data and $c>0$ is a constant.
As usual, we assume the symmetry conditions
\bel{pre condition symmetry}
B^{j\alpha\beta} = B^{j\beta\alpha}
\ee
and our main assumption is the null condition for the wave components $w_i$:
\index{null}
\bel{main structure c}
\aligned
&  B^{j\alphar\betar}\xi_{\alphar}\xi_{\betar}
= P_{i}^{jk\alphar\betar}\xi_{\alphar}\xi_{\betar} = 0
 \qquad
\text{ for all } (\xi_0)^2 - \sum_a (\xi_a)^2 = 0.
\endaligned
\ee

In the earlier work \cite{PLF-MY-book}, the nonlinear terms $B^{j\alpha\beta} w_j\del_\alpha\del_\beta v$
(actually denoted $B_{\ic}^{\jh \kc\alphar\betar} w_{\jh}\del_{\alphar}\del_{\betar}w_{\kc}$ therein) were assumed to be vanishing, and in fact this was our only genuine restriction 
since, with such terms, solutions may not have the same time decay and asymptotics of solutions as the ones of the homogeneous linear wave-Klein-Gordon equations in Minkowski space.
With the new technique in the present paper, the hyperboloidal foliation method does extend to encompass these terms (provided $B^{j\alpha\beta}$ is a null quadratic form).

Let us consider the initial value problem \eqref{main eq main} with sufficiently smooth initial data posed on the spacelike hypersurface $\{t= 2\}$ of constant time and compactly supported in the ball $\{t= 2; \, |x|\leq 1\}$. Under the conditions \eqref{pre condition symmetry}--\eqref{main structure c}, there exists a real $\eps_0>0$ such that, for all initial data
${w_i}_0, {w_i}_1, v_0, v_1: \RR^3 \to \RR$ satisfying the smallness condition
\be
\sum_i \| ({w_i}_0, v_0) \|_{\Hf^{N+1}(\RR^3)} + \| ({w_i}_1, v_1) \|_{\Hf^N(\RR^3)} < \eps_0,
\ee
the Cauchy problem  \eqref{main eq main} admits a unique, smooth global-in-time solution. In addition, the lower-order
energy of the wave components
remains globally bounded in time.


\section{The hyperboloidal foliation method}
\label{sec the hyper}

\subsection{The semi-hyperboloidal frame}

We begin with basic notions and consider the $(3+1)$-dimensional Minkowski space with signature $(-,+,+,+)$. In canonical Cartesian coordinates, we write $(t,x) = (x^0,x^1,x^2,x^3)$
and $r^2 := |x|^2 = (x^1)^2 + (x^2)^2 + (x^3)^2$. In addition to the partial derivative fields $\del_t=\del_0$ and $\del_a$, we will also use the {\sl Lorentz boosts} (for $a=1,2,3$): 
\be
L_a := x^a\del_t + t\del_a
      = x_a\del_0 - x_0\del_a,
\ee
where we raise and lower indices with the Minkowski metric.

More precisely, throughout, we analyze solutions defined in the interior of the future light cone
$$
\Kcal := \{(t,x) \, / \, r<t-1\}
$$
with vertex $(1,0,0,0)$, and we introduce the following foliation of the interior of the cone ${\big\{ (t,x) \, / \, |x| < t \big\}}$
by hyperboloidal hypersurfaces with hyperbolic radius $s$: 
$$
\Hcal_s := \big\{(t,x) \, / \, t^2-r^2 = s^2; \quad t>0 \big\}. 
$$
 The sub-domain of $\Kcal$ limited by two hyperboloids (with $s_0 < s_1$) is denoted by
$$
\Kcal_{[s_0,s_1]} := \big\{ (t,x) \, / \, s_0^2 \leq t^2-r^2 \leq s_1^2; \quad r<t-1 \big\}
\subset \Kcal. 
$$
Observe that the hyperboloids eventually ``exit'' the region $\Kcal$ and are asymptote to the same light cone
$\big\{ t - r = 0 \big\}$. 

The semi-hyperboloidal frame, as we call it, is defined by rescaling the Lorentz boosts: 
\be
\delu_0 := \del_t,
\qquad \delu_a:= \frac{x_a}{t}\del_t + \del_a \qquad (a=1,2,3). 
\ee
Observe that the vectors $\delu_a$ generates the tangent space to the hyperboloids.
Furthermore, we also introduce the vector field  
$\newperp  : = \del_t + \frac{x^a}{t}\del_a$,  
which is orthogonal to the hyperboloids for the Minkowski metric. (This vector field also coincides, up to an essential factor $1/t$, with the scaling vector field $S$.)

To make explicit the change of frame formulas 
$\delu_\alpha = \Phi_\alpha^\beta \del_\beta$ and
$\del_\alpha = \Psi_\alpha^\beta \delu_\beta$, we need the following matrices
$$
\big(\Phi_\alpha^{\beta}\big)
= \big({\Phi_\alpha}^{\beta}\big)
=
\left(
\aligned
&1 &&0 &&&0 &&&&0
\\
&x^1/t &&1 &&&0 &&&&0
\\
&x^2/t &&0 &&&1 &&&&0
\\
&x^3/t &&0 &&&0 &&&&1
\endaligned
\right),
\qquad
\qquad
\big(\Psi_\alpha^{\beta}\big)
= \big({\Psi_\alpha}^{\beta}\big)
=
\left(
\aligned
&1 &&0 &&&0 &&&&0
\\
-&x^1/t &&1 &&&0 &&&&0
\\
-&x^2/t &&0 &&&1 &&&&0
\\
-&x^3/t &&0 &&&0 &&&&1
\endaligned
\right).
$$
Any tensor can be expressed in either the Cartesian natural frame $\{\del_\alpha\}$ or the semi-hyperboloidal frame $\{\delu_\alpha\}$. We use standard letters for components in the Cartesian frame and we use underlined letters for components in the semi-hyperboloidal frame, so that, for example,
$
T^{\alpha\beta}\del_\alpha\otimes\del_\beta = \Tu^{\alpha\beta}\delu_\alpha\otimes \delu_\beta,
$
and the relations between $T^{\alpha\beta}$ and $\Tu^{\alpha\beta}$ are
$$
\Tu^{\alpha\beta} = \Psi_{\alpha'}^{\alpha}\Psi_{\beta'}^{\beta}T^{\alpha'\beta'},
\qquad
T^{\alpha\beta} = \Phi_{\alpha'}^{\alpha}\Phi_{\beta'}^{\beta}\Tu^{\alpha'\beta'}.
$$

Associated with the semi-hyperboloidal frame, we have the dual semi-hyperboloidal frame
\be
\thetau^0:= dt - \frac{x^a}{t}dx^a,\qquad \thetau^a: = dx^a,
\ee
and the relations between the semi-hyperboloidal dual frame and the standard dual frame are
$
\thetau^{\alpha} = \Psi_{\alpha'}^{\alpha}dx^{\alpha'}$, 
$
dx^{\alpha} = \Phi^{\alpha}_{\alpha'}\thetau^{\alpha'}.
$
Hence, for any two-tensor $T_{\alpha\beta}dx^\alpha\otimes dx^{\beta} = \Tu_{\alpha\beta} \thetau^{\alpha}\otimes \thetau^{\beta}$, we have the change of basis formulas
$$
\Tu_{\alpha\beta} = T_{\alpha'\beta'}\Phi_\alpha^{\alpha'}\Phi_\beta^{\beta'},
\qquad
T_{\alpha\beta} = \Tu_{\alpha'\beta'}\Psi_\alpha^{\alpha'}\Psi_\beta^{\beta'}.
$$

With the above notation, in the semi-hyperboloidal frame we can express the Minkowski metric and its inverse as
$$
\aligned
&\minu_{\alpha\beta} =
\left(
\begin{array}{cccc}
-1 &-x^1/t &-x^2/t &-x^3/t
\\
-x^1/t &1-(x^1/t)^2 &-x^1x^2/t^2 &-x^1x^3/t^2
\\
-x^2/t &-x^2x^1/t^2 &1-(x^2/t)^2 &-x^2x^3/t^2
\\
-x^3/t &-x^3x^1/t^2 &-x^3x^2/t^2 &1-(x^3/t)^2
\end{array}
\right),
\\
&\minu^{\alpha\beta} =
\left(
\begin{array}{cccc}
-s^2/t^2 &-x^1/t &-x^2/t &-x^3/t
\\
-x^1/t &1 &0&0
\\
-x^2/t &0 &1 &0
\\
-x^3/t &0 &0 &1
\end{array}
\right).
\endaligned
$$

Furthermore, given any multi-index $I = (\alpha_n,\alpha_{n-1},\ldots, \alpha_1)$ (where the order is chosen for convenience), we denote by 
$
\del^I := \del_{\alpha_n}\del_{\alpha_{n-1}} \ldots \del_{\alpha_1}
$
the product of $n=|I|$ partial derivatives  (with $0 \leq \alpha_i\leq 3$) and, similarly, by
$
L^J = L_{a_n}L_{a_{n-1}} \ldots L_{a_1}
$
the product of $n=|J|$ Lorentz boosts (with $1\leq a_i\leq 3$).


\subsection{The hyperbolic variables and the hyperboloidal frame}

Within the  future cone $\Kcal$, we introduce the change of variables
\bel{Hyper vairables}
\aligned
&\xb^0 = s: = \sqrt{t^2 - r^2},
\qquad
\xb^a = x^a,
\endaligned
\ee
together with the corresponding natural frame
\bel{Hyper frame}
\aligned
&\delb_0 := \del_s = \frac{s}{t}\del_t  = \frac{\sqrt{t^2-r^2}}{t}\del_t,
\\
&\delb_a := \del_{\xb^a} = \frac{\xb^a}{t}\del_t + \del_a = \frac{x^a}{t}\del_t + \del_a,
\endaligned
\ee
which we refer to as the {\sl hyperboloidal frame.}
The transition matrices between the hyperboloidal frame and the Cartesian frame are
$$
\big(\Phib^{\beta}_\alpha\big)
=
\big({\Phib^{\beta}}_\alpha\big)
 = \left(
\begin{array}{cccc}
s/t &0 &0 &0
\\
x^1/t &1 &0 &0
\\
x^2/t &0 &1 &0
\\
x^3/t &0 &0 &1
\end{array}
\right),
$$ 
$$
\big(\Phib^{\beta}_\alpha\big)^{-1}
 = \big(\Psib^{\beta}_\alpha\big) = \big({\Psib^{\beta}}_\alpha\big)
= \left(
\begin{array}{cccc}
t/s &0 &0 &0
\\
-x^1/s &1 &0 &0
\\
-x^2/s &0 &1 &0
\\
-x^3/s &0 &0 &1
\end{array}
\right),
$$
so that
$
\delb_\alpha = \Phib^{\beta}_\alpha\del_\beta
$
and
$
\del_\alpha = \Psib^{\beta}_\alpha\delb_\beta.
$

The dual hyperboloidal frame then reads
$
d\xb^0 := ds = \frac{t}{s}dt - \frac{x^a}{s}dx^a$ and
$d\xb^a := dx^a$. 
The Minkowski metric in the hyperboloidal frame reads\footnote{Our sign convention is opposite to the one in our monograph \cite{PLF-MY-book}, since the metric here has signature $(-, +, +, +)$.}
$$
\minb^{\alpha\beta} = \left(
\begin{array}{cccc}
-1 &-x^1/s &-x^2/s &-x^3/s
\\
-x^1/s &1 &0 &0
\\
-x^2/s &0 &1 &0
\\
-x^3/s &0 &0 &1
\end{array}
\right).
$$

In summary, an arbitrary tensor can be expressed in three different frames: the standard frame $\{\del_\alpha\}$, the semi-hyperboloidal frame $\{\delu_\alpha\}$, or the hyperboloidal frame $\{\delb_\alpha\}$.
We use symbols, underlined symbols, and overlined symbols for tensor components in these frames, respectively. For example, a tensor $T^{\alpha\beta}\del_\alpha\otimes \del_\beta$ is written as
$$
T^{\alpha\beta}\del_\alpha\otimes\del_\beta = \Tu^{\alpha\beta}\delu_\alpha\otimes \delu_\beta = \Tb^{\alpha\beta}\delb_\alpha\otimes\delb_\beta, 
$$
where 
$\Tb^{\alpha\beta} = \Psib_{\alpha'}^{\alpha}\Psib_{\beta'}^{\beta}T^{\alpha'\beta'}$
and, moreover, by setting $C :=\max_{\alpha\beta}|T^{\alpha\beta}|$, we have in the hyperboloidal frame
\be
|\Tb^{00}|\leq C(t/s)^2,
\qquad |\Tb^{a0}|\leq C \, (t/s),
\qquad |\Tb^{ab}|\leq C.
\ee


\subsection{Energy estimate on hyperboloids}

Throughout this paper, for any function $u=u(t,x)$ defined in $\RR^{3+1}$ (or a subset of it), we consider the integral of 
on the hyperboloids $\Hcal_s$ defined as follows: 
\bel{add-form}
\| u \|_{L_f^1(\Hcal_s)}
:= 
\int_{\Hcal_s} u \, dx = \int_{\RR^3}u\big(\sqrt{s^2+r^2},x\big)dx
\ee
The subscript refer the fact that we are endowing $\Hcal_s$ with the {\sl flat} Euclidian metric. We emphasize that this is {\sl not } an integration with respect to the induced Riemannian metric and volume form which should be $(s/t) \, dx$. This notation will be more convenient for our analysis in this paper but, of course, it is completely  straightforward to re-state all of our estimates by including the weight $s/t$. 

Consider the hyperboloidal foliation of a region $\Kcal_{[2,s_1]} = \bigcup_{2\leq s\leq s_1}\Hcal_{s}$, together with the  {\sl hyperboloidal energy} (associated to the Minkowski metric) at some hyperbolic time $s \in [2, s_1]$ 
\begin{equation}\label{eq hyper-energy}
\aligned
E_{m,c}(s,u):=& \int_{\Hcal_s}\Big((\del_tu)^2 + \sum_a(\del_au)^2 + 2(x^a/t)\del_tu\del_au + c^2u^2\Big) dx
\\
=&\int_{\Hcal_s} \Big(\big((s/t)\del_t u\big)^2 + \sum_a (\delu_a u)^2  + c^2u^2\Big) \, dx
\\
=&\int_{\Hcal_s}\left((\newperp  u)^2 + \sum_a\left((s/t)\del_a u\right)^2 + \sum_{a<b}\left(t^{-1}\Omega_{ab}u\right)^2 + c^2u \right) \, dx, 
\endaligned
\end{equation}
where we have also introduced the rotational vector fields $\Omega_{ab}: = x^a\del_b - x^b\del_a$ (not directly used here). When $c=0$, we also write $E_m(s,u): = E_{m,0}(s,u)$ for short.

We will also need the  {\sl hyperboloidal energy for the curved metric $g^{\alpha\beta} = m^{\alpha\beta} + h^{\alpha\beta}$}:
\bel{eq hyper-energy curved}
E_{g,c}(s,u):=  
E_{m,c}(s,u) + \int_{\Hcal_s}
\Big( 2
h^{\alpha\beta}\del_tv\del_\beta v X_\alpha  
-  h^{\alpha\beta}\del_\alpha v\del_\beta v \Big) \, dx, 
\ee
where we have used the notation $X^0 = 1$ and $X^a = -x^a/t$.

All of our estimates concern the interior of the light cone $\Kcal \cap \big\{ t \geq 2 \big\}$ away from the origin.  From here onwards, we assume that all the functions under consideration are {\sl spatially compact} and, in particular, vanish identically in a neighborhood of the light cone $\big\{ t-1 = |x|=r \big\}$.  
More precisely, we assume that the initial data on the slice $t=2$ are supported in the ball $|x| \leq M$ for some $M \in (0,1)$, and we construct solutions supported in the larger domain $|x| \leq M + t$. In short, we will say that the functions under consideration are {\bf spatially compactly supported in } $\Kcal$ or, in short, spatially compactly supported.

We easily adapt the energy estimate in \cite[Proposition 2.3.1]{PLF-MY-book} to  the equation \eqref{eq main}, as follows.

\begin{proposition}[Energy estimate for the hyperboloidal foliation]
\label{prop energy}
1. 
For every function $u$ which is defined in the region $\Kcal_{[2,s]}$ and spatially supported in $\Kcal$, 
 one has for all $s \geq 2$ 
\be
\label{ineq energy wave}
E_m(s,u)^{1/2}\leq E_m(2,u)^{1/2} + \int_2^s \| \Box u\|_{L_f^2(\Hcal_\sbar)} d\sbar.
\ee

2. Let $v$ be a solution to the Klein-Gordon equation on a curved space (with the definition of $\Boxt$ given earlier) 
\be
- \Boxt_g v + c^2 v = f,
\ee
defined the region $\Kcal_{[2,s]}$ and spatially supported in $\Kcal$. Suppose that $h^{\alpha\beta}
= g^{\alpha\beta} - m^{\alpha\beta}$ satisfies the following two conditions
(for some constant $\kappa \geq 1$ and some function $M$):
\begin{subequations}
\bel{ineq coersive 1}
\kappa^{-2} E_{g,c}(s,v) \leq E_{m,c}(s,v) \leq \kappa^2 E_{g,c}(s,v),
\ee
\be
\label{ineq coersive 1'}
\bigg|\int_{\Hcal_s}(s/t)\Big(\del_\alpha h^{\alpha\beta}
\del_t v\del_\beta v - \frac{1}{2}\del_t h^{\alpha\beta}\del_\alpha v\del_\beta v\Big) \, dx\bigg|
\leq M(s)E_{m,c}(s,v)^{1/2}.
\ee
\end{subequations}
Then, the evolution of the hyperboloidal energy is controlled (for all $s \geq 2$) by
\bel{ineq energy KG}
E_{m,c}(s,v)^{1/2}\leq \kappa^2 E_m(2,v)^{1/2} + \kappa^2 \int_2^s \Big(\|f\|_{L_f^2(\Hcal_\sbar)} + M(s)\Big) d\sbar.
\ee
\end{proposition}

\begin{proof} We apply the multiplier $\del_t u$ to $\Box u$ and, by a standard calculation, 
$$
\frac{1}{2}\del_t\Big((\del_t u)^2 + \sum_a(\del_au)^2\Big) - \sum_a\del_a(\del_t u \del_a u) = \del_t u \, \Box u.
$$
We integrate this identity in $\Kcal_{[2,s]}$ and apply Stokes' formula by observing that, by assumption, the functions under consideration are spatially supported in $\Kcal$,  
so that there is no ``boundary'' contribution, and we find (see \cite[Sec. 2.3]{PLF-MY-book})
$$
\frac{1}{2}E_m(s,u) - \frac{1}{2}E_m(2,u) = \int_2^s\int_{\Hcal_\sbar}(\sbar/t)\del_tu \, \Box u \, dx d\sbar.
$$
We differentiate this identity with respect to $s$ and apply Cauchy-Schwarz inequality, as follows:
$$
E_m(s,u)^{1/2}\frac{d}{ds}E_m(s,u)^{1/2} 
= \int_{\Hcal_s}(\sbar/t)\del_t v  \Box u \, dx \leq \| \Box u\|_{L_f^2(\Hcal_s)}\|(s/t)\del_t u\|_{L_f^2(\Hcal_s)}.
$$
Next, by recalling \eqref{eq hyper-energy}, we find
$
\frac{d}{ds}E_m(s,u)^{1/2} \leq \| \Box u \|_{L_f^2(\Hcal_s)}
$
and, by integration over $[2,s]$, the inequality \eqref{ineq energy wave} is established.

\vskip.3cm 

Next, for the derivation of \eqref{ineq energy KG}, we rely on the multiplier
$\del_t v$ and, by a standard calculation, we get
$$
\aligned
&\frac{1}{2}\del_t\big((\del_tv)^2+\sum_a(\del_av)^2+c^2v^2\big) + \sum_a\del_a\big(-\del_av\del_t v\big)
\\
&+ \del_{\alpha}\big(-h^{\alpha\beta}\del_{\beta}v\del_t v\big) + \frac{1}{2}\del_t\big(h^{\alpha\beta}\del_{\alpha}v\del_{\beta}v\big)
\\
&= \del_t v f - \del_{\alpha}g^{\alpha\beta}\del_\beta v\del_t v + \frac{1}{2}\del_tg^{\alpha\beta}\del_{\alpha}v\del_{\beta}v. 
\endaligned
$$

As was done in the derivation of \eqref{ineq energy wave}, we combine this identity with \eqref{ineq coersive 1}-\eqref{ineq coersive 1'} and obtain
$$
E_{g,c}(s,v)^{1/2}\frac{d}{ds}E_{g,c}(s,v)^{1/2} \leq \kappa \big(\| f \|_{L_f^2(\Hcal_s)} +M(s)\big) E_{g,c}(s,v)^{1/2}.
$$
We integrate this inequality on $[2,s]$ and find
$$
E_{g,c}(s,v)^{1/2}\leq E_{g,c}(2,v)^{1/2} + \kappa \int_2^s \big(\| f \|_{L_f^2(\Hcal_\sbar)} +M(\sbar)\big) ds. 
$$
Finally, we again apply \eqref{ineq coersive 1}, and \eqref{ineq energy KG} is proven.
\end{proof}


\subsection{Sobolev inequality on hyperboloids}

In order to turn an $L^2$ energy estimate into an $L^\infty$ estimate, we will rely on the following version of the Sobolev inequality (Klainerman \cite{Klainerman85}, H\"ormander \cite[Lemma 7.6.1]{Hormander}, LeFloch and Ma \cite[Section 5]{PLF-MY-book}).

\begin{proposition}[Sobolev-type estimate on hyperboloids]
\label{pre lem sobolev}
For any sufficiently smooth function $u=u(s,x)$ which is defined in $\Kcal_{[2, +\infty)}$ and 
is spatially supported in $\Kcal$,  
one has 
\be
\label{eq 1 sobolev}
\sup_{(s,x) \in \Hcal_s} (s + |x|)^{3/2} \, |u(s,x)|
\lesssim
\sum_{|I|\leq 2} \| L^I u (s, \cdot) \|_{L_f^2(\Hcal_s)},  \qquad s \geq 2,
\ee
where the implied constant is independent of $s$ and $u$.
\end{proposition}

\begin{proof} Recalling that $t = \sqrt{s^2 + |x|^2}$ on $\Hcal_s$, we consider the restriction to the hyperboloid
$$
w_s(x) := u(\sqrt{s^2+|x|^2},x).
$$
Fixing $s_0$ and a point $(t_0,x_0) \in \Hcal_{s_0}$ (with
$t_0 = \sqrt{s_0^2 + |x_0|^2}$), we observe that
\bel{eq:gh67}
\del_aw_{s_0}(x) = \delu_au\big(\sqrt{s_0^2+|x|^2},x\big) = \delu_au(t,x),
\ee
with $t = \sqrt{s_0^2 + |x|^2}$ and
\be
t\del_aw_{s_0}(x) = t\delu_au\big(\sqrt{s_0^2+|x|^2},t\big) = L_au(t,x).
\ee
We introduce the function $g_{s_0,t_0}(y) := w_{s_0}(x_0 + t_0\,y)$ and write
$$
g_{s_0,t_0}(0) = w_{s_0}(x_0) = u\big(\sqrt{s_0^2+|x_0|^2},x_0\big)=u(t_0,x_0).
$$
By applying the standard Sobolev inequality to the function $g_{s_0,t_0}$, we find
$$
\big|g_{s_0,t_0}(0)\big|^2\leq C\sum_{|I|\leq 2}\int_{B(0,1/3)}|\del^Ig_{s_0,t_0}(y)|^2 \, dy,
$$
where $B(0, 1/3) \subset \RR^3$ is the ball centered at the origin and with radius $1/3$.

Next, taking into account the identity (with $x = x_0 + t_0y$)
$$
\aligned
\del_ag_{s_0,t_0}(y)
& = t_0\del_aw_{s_0}(x_0 + t_0y) 
\\
& = t_0\del_aw_{s_0}(x) = t_0\delu_au\big(t,x)
\endaligned
$$
and, in view of \eqref{eq:gh67}, we find (for all $I$)
$\del^Ig_{s_0,t_0}(y) = (t_0\delu)^I u(t,x)$
and, thus,
$$
\aligned
\big|g_{s_0,t_0}(0)\big|^2 \leq& C\sum_{|I|\leq 2}\int_{B(0,1/3)}\big|(t_0\delu)^I u\big(t,x)\big)\big|^2dy
\\
= & C t_0^{-3}\sum_{|I|\leq 2}\int_{B((t_0,x_0),t_0/3)\cap \Hcal_{s_0}}\big|(t_0\delu)^I u\big(t,x)\big)\big|^2dx.
\endaligned
$$

We note that
$$
\aligned
(t_0\delu_a(t_0\delu_b w_{s_0}))
& = t_0^2\delu_a\delu_bw_{s_0}
\\
& = (t_0/t)^2(t\delu_a)(t\delu_b) w_{s_0} - (t_0/t)^2 (x^a/t)L_b w_{s_0}
\endaligned
$$
and that $x^a/t = x^a_0/t + yt_0/t = (x^a_0/t_0 + y)(t_0/t)$. Consequently, in the region $y\in B(0,1/3)$ of interest, the factor $|x^a/t|$ is bounded by $C(t_0/t)$ and we conclude that (for $|I| \leq 2$)
$$
|(t_0 \delu)^I u| \leq \sum_{|J| \leq |I|} | L^J u| (t_0/t)^2.
$$

On the other hand, in the region $|x_0|\leq t_0/2$, we have $t_0\leq \frac{2}{\sqrt{3}}s_0$ and thus
$$
t_0 \leq C s_0 \leq C \sqrt{|x|^2 + s_0^2} = Ct
$$
for some fixed constant $C>0$.  When $|x_0|\geq t_0/2$ then in the region $B((t_0,x_0),t_0/3)\cap \Hcal_{s_0}$ we have
$t_0 \leq C|x| \leq C\sqrt{|x|^2 + s_0^2} =Ct$ and, consequently,
$$
|(t_0 \delu)^I u| \leq C \, \sum_{|J| \leq |I|} | L^J u|
$$
and
$$
\aligned
\big|g_{s_0,t_0}(y_0)\big|^2
\leq& Ct_0^{-3}\sum_{|I|\leq 2}\int_{B(x_0,t_0/3)\cap\Hcal_{s_0}}\big|(t\delu)^I u\big(t,x)\big)\big|^2 \, dx
\\
\leq& Ct_0^{-3}\sum_{|I|\leq 2}\int_{\Hcal_{s_0}}\big|L^I u(t,x)\big|^2 \, dx. 
\endaligned
$$ 
\end{proof}


\subsection{Hardy-type estimate along the hyperboloidal foliation}

The following variant of Hardy's inequality was established in \cite[Section 5]{PLF-MY-book}. This inequality plays an essential role in order to estimate the $L^2$ norm of the wave component itself (but not only its gradient).

\begin{proposition}[Hardy-type estimate on the hyperboloidal foliation]
\label{pre Hardy hyper}
For any sufficiently smooth function which is defined in the future region $\Kcal_{[2, s]}$ and is 
spatially supported in $\Kcal$,  
one has for  $s \geq 2$
\bel{eq 1 Hardy ineq}
\aligned
\|s^{-1}u\|_{L_f^2(\Hcal_s)} \lesssim \,
&
 \|u\|_{L^2(\Hcal_2)} + \sum_{a}\|\delu_a u\|_{L_f^2(\Hcal_s)}
\\
&+ \sum_a \int_2^{s}\sbar^{-1} \Big( \|\delu_a u\|_{L_f^2(\Hcal_\sbar)} + \|(\sbar/t)\del_a u\|_{L_f^2(\Hcal_\sbar)} \Big) \, d\sbar,
\endaligned
\ee
where the implied constant is independent of $s$ and $u$.
\end{proposition}

The proof uses a version of the classical Hardy inequality on hyperboloids, 
as well as a vector field that will be introduced in the proof of the proposition, below.

\begin{lemma}
\label{pre Hardy lem1}
For any sufficiently smooth function which is defined in the future region $\Kcal_{[2, s]}$
and is spatially supported in $\Kcal$,  
one has for all $s \geq 2$ 
$$
\|r^{-1} u\|_{L_f^2(\Hcal_s)}\lesssim \sum_a \|\delu_a u\|_{L_f^2(\Hcal_s)},
$$
where the implied constant is independent of $s$ and $u$.
\end{lemma}

\begin{proof} As in the proof of Proposition~\ref{pre lem sobolev}, we consider the function
$w_s(x):= u\big(\sqrt{s^2+|x|^2},x\big)$, which satisfies
$\del_aw_s(x) = \delu_a u\big(\sqrt{s^2+|x|^2},x\big)$, and we apply the classical Hardy inequality to $w_s$. It follows that 
$$
\aligned
\int_{\mathbb{R}^3}|r^{-1} w_s(x)|^2dx
&\lesssim \int_{\mathbb{R}^3}|\nabla w_s(x)|^2dx
= C\sum_a\int_{\mathbb{R}^3}\big|\delu_au(\sqrt{s^2+r^2},x)\big|^2dx
\\
&\lesssim \sum_a\int_{\Hcal_s}\big|\delu_a u(t,x)\big|^2dx.
\endaligned
$$ 
\end{proof}

\begin{proof}[Proof of Proposition \ref{pre Hardy hyper}]
Let $\chi$ be a a smooth cut-off function satisfying
$$
\chi (r) =\begin{cases}
0, \quad & 0\leq r\leq 1/3
\\
1, \quad & 2/3\leq r,
\end{cases} 
$$
and let us distinguish between the region ``near'' and ``away'' from the light cone.
We consider the decomposition
$$
\|s^{-1}u\|_{L_f^2(\Hcal_s)} \leq \|\chi(r/t)s^{-1} u\|_{L_f^2(\Hcal_s)} + \|(1-\chi(r/t))s^{-1} u\|_{L_f^2(\Hcal_s)}.
$$
Our estimate of $\|(1-\chi(r/t))s^{-1} u\|_{L_f^2(\Hcal_s)}$ is based on the inequality 
$
\big(1-\chi(r/t)\big)s^{-1}\leq Ct^{-1}   
$
so that, by Lemma \ref{pre Hardy lem1},
\bel{eq pr2 Hardy}
\aligned
\|(1-\chi(r/t))us^{-1}\|_{L_f^2(\Hcal_s)}
& \leq \|t^{-1}u\|_{L_f^2(\Hcal_s)}
\\
& \leq \|r^{-1}u\|_{L_f^2(\Hcal_s)}\leq C\sum_a\|\delu_a u\|_{L_f^2(\Hcal_s)}.
\endaligned
\ee

The estimate near the light cone is more delicate and we first observe that, in the region $\Kcal_{[2,s]}$ of interest, 
$\chi(r/t)\lesssim \frac{\chi(r/t)r}{(1+r^2)^{1/2}}$
and, thus,
$$
\|\chi(r/t)s^{-1}u\|_{L_f^2(\Hcal_s)}\leq C\|r(1+r^2)^{-1/2}\chi(r/t)s^{-1}u\|_{L_f^2(\Hcal_s)},
$$
and the right-hand side of this inequality is controlled as follows. We introduce the vector field
$W = \big(0, -x^a  {t (u\chi(r/t))^2\over (1+r^2)s^2}\big)$
and compute its divergence
\bel{eq-itsdiv} 
\aligned
\text{div} \hskip.05cm  W
& = -2s^{-1} \del_a u \frac{r\chi(r/t)u}{(1+r^2)^{1/2}s}\cdot \frac{x^a t\chi(r/t)}{r(1+r^2)^{1/2}}
- 2s^{-1}\frac{u}{r} \frac{r\chi(r/t) u}{s(1+r^2)^{1/2}}\cdot \frac{\chi'(r/t)r}{(1+r^2)^{1/2}}
\\
& \quad
-\bigg(\frac{r^2t+3t}{(1+r^2)^2s^2} + 2 \frac{r^2t}{(1+r^2)s^4}\bigg)\big(u\chi(r/t)\big)^2.
\endaligned
\ee  
By applying Stokes'  theorem in the region $\Kcal_{[2,s_1]}$, we find
$$
\aligned
\int_{\Kcal_{[2,s_1]}}\!\!\!\!\text{div} \hskip.05cm W \, dxdt
=&\int_{\Hcal_{s}}W \cdot n \, d\sigma + \int_{\Hcal_{2}}W\cdot n \, d\sigma
\\
=&
\int_{\Hcal_{s}}\frac{r^2}{1+r^2}\big|u\chi(r/t)s^{-1}\big|^2 dx
-\int_{\Hcal_{2}}\frac{r^2}{1+r^2}\big|u\chi(r/t)s^{-1}\big|^2dx.
\endaligned
$$
Differentiating this identity with respect to $s$ leads us to
\bel{pre Hardy hyper eq1}
\aligned
&\frac{d}{ds}\bigg(\int_{\Kcal_{[2,s_1]}}\!\!\!\!\text{div} \hskip.05cm W \, dxdt\bigg) =
\frac{d}{ds}\bigg(\int_{\Hcal_{s}}\frac{r^2}{1+r^2}\big|u\chi(r/t)s^{-1}\big|^2 dx\bigg)
\\
&= 2 \, \bigg{\|}\frac{r u\chi(r/t)}{s(1+r^2)^{1/2}}\bigg{\|}_{L^2(\Hcal_{s})}\frac{d}{ds}\bigg{\|}\frac{r u\chi(r/t)}{s(1+r^2)^{1/2}}\bigg{\|}_{L^2(\Hcal_{s})}.
\endaligned
\ee

We then integrate \eqref{eq-itsdiv} in the region $\Kcal_{[2,s_1]}\subset \Kcal\cap \{2 \leq \sqrt{t^2-r^2}\leq s_1\}$:
$$
\aligned
\int_{\Kcal_{[2,s_1]}}\!\!\!\!\text{div} \hskip.05cm  W \, dxdt
&= -2\int_{\Kcal_{[2,s_1]}}\!\!\!\!s^{-1} \bigg( \del_a u \frac{r\chi(r/t)u}{(1+r^2)^{1/2}s}\frac{x^a t\chi(r/t)}{r(1+r^2)^{1/2}}\bigg) \, dxdt
\\
& \quad -2\int_{\Kcal_{[2,s_1]}}\!\!\!\! s^{-1}\frac{u}{r}  \frac{r\chi(r/t) u}{s(1+r^2)^{1/2}}\frac{\chi'(r/t)r}{(1+r^2)^{1/2}}dxdt
\\
&\quad -\int_{\Kcal_{[2,s_1]}}\!\!\!\!\bigg(\frac{r^2t+3t}{(1+r^2)^2s^2} + 2 \frac{r^2t}{(1+r^2)s^4}\bigg)\big(u\chi(r/t)\big)^2dxdt,
\endaligned
$$
which yields the identity
$$
\aligned
\int_{\Kcal_{[2,s_1]}}\!\!\!\!\text{div} \hskip.05cm  W \, dxdt
&= -2\int_{2}^{s_1}\int_{\Hcal_s}(s/t)s^{-1} \bigg( \del_a u \frac{r\chi(r/t)u}{(1+r^2)^{1/2}s}\frac{x^a t\chi(r/t)}{r(1+r^2)^{1/2}}\bigg) \, dxds
\\
&\quad -2\int_{2}^{s_1}\int_{\Hcal_s}(s/t) s^{-1}\frac{u}{r}  \frac{r\chi(r/t) u}{s(1+r^2)^{1/2}}\frac{\chi'(r/t)r}{(1+r^2)^{1/2}}dxds
\\
&\quad -\int_{2}^{s_1}\int_{\Hcal_s}(s/t)\bigg(\frac{r^2t+3t}{(1+r^2)^2s^2} + 2 \frac{r^2t}{(1+r^2)s^4}\bigg)\big(u\chi(r/t)\big)^2dxds
\\
&=:\int_{2}^{s_1}\big( T_1 + T_2 +T_3\big) \, ds.
\endaligned
$$
Here, we have $T_3\leq 0$, while
$$
\aligned
T_1 =& -2 s^{-1}\int_{\Hcal_s}(s/t) \bigg( \del_a u \frac{r\chi(r/t)u}{(1+r^2)^{1/2}s}\frac{x^a t\chi(r/t)}{r(1+r^2)^{1/2}}\bigg) \, dx
\\
\leq & 2s^{-1}\bigg{\|}\frac{r u\chi(r/t)}{s(1+r^2)^{1/2}}\bigg{\|}_{L_f^2(\Hcal_s)}
\\
& \qquad \qquad \cdot
\sum_a \|(s/t)\del_a u\|_{L_f^2(\Hcal_s)}\big{\|}\chi(r/t)x^atr^{-1}(1+r^2)^{-1/2}\big{\|}_{L^\infty(\Hcal_s)}
\\
\leq &Cs^{-1}\bigg{\|}\frac{r u\chi(r/t)}{s(1+r^2)^{1/2}}\bigg{\|}_{L_f^2(\Hcal_s)}\sum_a\|(s/t)\del_a u\|_{L_f^2(\Hcal_s)}
\endaligned
$$
and
$$
\aligned
T_2 =& -2 s^{-1}\int_{\Hcal_s}(s/t)\frac{u}{r}  \frac{r\chi(r/t) u}{s(1+r^2)^{1/2}}\frac{\chi'(r/t)r}{(1+r^2)^{1/2}}dx
\\
\leq &Cs^{-1}
\bigg{\|}\frac{r u\chi(r/t)}{s(1+r^2)^{1/2}}\bigg{\|}_{L_f^2(\Hcal_s)}
\|ur^{-1}\|_{L_f^2(\Hcal_s)}\big{\|}r\chi'(r/t)(1+r^2)^{-1/2}\big{\|}_{L^\infty(\Hcal_s)}
\\
\leq &Cs^{-1}\bigg{\|}\frac{r u\chi(r/t)}{s(1+r^2)^{1/2}}\bigg{\|}_{L_f^2(\Hcal_s)}\sum_{a}\|\delu_a u\|_{L_f^2(\Hcal_s)},
\endaligned
$$
where Lemma \ref{pre Hardy lem1} was used.

We write our identity in the form
$
\frac{d}{ds}\bigg(\int_{\Kcal_{[2,s]}}\!\!\!\!\text{div} \hskip.05cm W \, dxdt\bigg) = T_1 + T_2 + T_3
$
and obtain
\bel{pre Hardy hyper eq0}
\aligned
& \frac{d}{ds}\bigg(\int_{\Kcal_{[2,s_1]}}\!\!\!\!\text{div} \hskip.05cm W \, dxdt\bigg)
\\
& \lesssim s^{-1}\bigg{\|}\frac{r u\chi(r/t)}{s(1+r^2)^{1/2}}\bigg{\|}_{L_f^2(\Hcal_s)}
\sum_{a} \Big( \|(s/t)\del_a\|_{L_f^2(\Hcal_s)}+\|\delu_a u\|_{L_f^2(\Hcal_s)} \Big).
\endaligned
\ee

Finally,
combining \eqref{pre Hardy hyper eq0} and \eqref{pre Hardy hyper eq1} yields us
$$
\frac{d}{ds}\bigg{\|}\frac{r u\chi(r/t)}{s(1+r^2)^{1/2}}\bigg{\|}_{L^2(\Hcal_{s})}
\lesssim
s^{-1}\sum_a\big(\| {s \over t} \del_a u\|_{L^2(\Hcal_{s})} + \|\delu_a u\|_{L^2(\Hcal_{s})}\big)
$$
and, by integration over $[2,s]$,
\bel{pre Hardy hyper eq2}
\aligned
&
\big{\|}r(1+r^2)^{-1/2}\chi(r/t)s^{-1}u\big{\|}_{L_f^2(\Hcal_s)}
\\
&\leq \big{\|}r(1+r^2)^{-1/2}\chi(r/t)2^{-1}u\big{\|}_{L^2(\Hcal_{2})}
  + \sum_a\int_{2}^s\sbar^{-1}\big(\|  {\sbar \over t} \del_a u\|_{L_f^2(\Hcal_\sbar)} + \|\delu_a u\|_{L_f^2(\Hcal_\sbar)}\big)d\sbar.
\endaligned
\ee
From Lemma \ref{pre Hardy lem1}, we then deduce that
\bel{eq pr1 ch5 prop 1}
\aligned
\|\chi(r/t)s^{-1}u\|_{L_f^2(\Hcal_s)}
&\lesssim \|r(1+r^2)^{-1/2}\chi(r/t)s^{-1}u\|_{L_f^2(\Hcal_s)}
\\
&\lesssim \big{\|}2^{-1}u\big{\|}_{L^2(\Hcal_{2})} 
+ \sum_a\int_{2}^s\sbar^{-1}\big(\|  {\sbar \over t} \del_a u\|_{L_f^2(\Hcal_\sbar)} + \|\delu_a u\|_{L_f^2(\Hcal_\sbar)}\big)d\sbar, 
\endaligned
\ee
and remains to combine \eqref{eq pr2 Hardy} with \eqref{eq pr1 ch5 prop 1}.
\end{proof}


\section{Sup-norm estimates for the wave and Klein-Gordon equations}
\label{sec supnorm}

\subsection{Statement for the wave equation}

\begin{proposition}[A sup-norm estimate for the wave equation with source]
\label{Linfini wave}
Let $u$ be a spatially compactly supported to the wave equation
\be
\aligned
&-\Box u = f,
\\
& u|_{t=2} = 0,\qquad \del_t u|_{t=2} = 0,
\endaligned
\ee
where the source $f$ is spatially compactly supported in $\Kcal$  
and satisfies the estimate
$$
|f|\leq C_f t^{-2-\nu}(t-r)^{-1+\mu}
$$
for some constants $C_f>0$, $0<\mu\leq 1/2$, and $0< |\nu|\leq 1/2$. 
Then, the following estimate holds:
\be
\label{Linfini wave ineq}
|u(t,x)|
\lesssim
\begin{cases}
\frac{C_f}{\nu\mu}(t-r)^{\mu-\nu} t^{-1}, \qquad & 0< \nu\leq 1/2,
\\
\frac{C_f}{|\nu|\mu}(t-r)^{\mu} t^{-1 -\nu}, &-1/2\leq \nu < 0.
\end{cases}
\ee
\end{proposition}

We recall that the energy estimate on wave equation does not control the solution itself but only its gradient. So when we apply the Sobolev inequality and obtain a sup-norm estimate (cf.~for example \cite{PLF-MY-book}), there is no immediate estimate on the sup-norm of the solution itself. The estimate above yields a (sharp) sup-norm estimate on the solution itself and will play an essential role for the control of
the quasi-linear term $u\del_t\del_t v$ in our model problem. We emphasize that the range $-1/2\leq \nu < 0$ will only be used in the second part \cite{PLF-YM-two}.


\subsection{Proof of the sup-norm estimate for the wave equation}

We now state a technical lemma and give the proof of Proposition \ref{Linfini wave}, but postpone the proof of the lemma to the end of this section. Let $d\sigma$ be the Lebesgue measure on the sphere $\{|y| = 1-\lambda \}$ and $x\in \mathbb{R}^3$ with $r = |x|$. We are interested in controling the integral
$$
I(\lambda)=I(\lambda, t, x/t):=\int_{|y| = 1-\lambda,|\frac{x}{t}-y|\leq \lambda-t^{-1}}
 \frac{d\sigma(y)}{\big(\lambda-\big|\frac{x}{t} - y\big|\big)^{1-\mu}}.
$$
Our bounds below are consistent with the obvious estimate where $x = 0$:
\begin{equation}\label{last 1}
I(\lambda, t, 0) = 4\pi(2\lambda - 1)^{-1+\mu}(1-\lambda)^2.
\end{equation}
Clearly, when $0< \lambda\leq \frac{t-r+1}{2t}$, one has $I(\lambda) = 0$.

\begin{lemma}\label{wave lemma1}
When $\frac{t-r+1}{2t}\leq \lambda\leq 1$,  the following estimate holds:
$$
I(\lambda) \lesssim
\left\{
\aligned
&\frac{\lambda t(1-\lambda)}{\mu r}\left(\frac{t-r}{t}\right)^{\mu},
\qquad
&&\frac{t-r+1}{2t}\leq \lambda\leq \frac{t+r+1}{2t},
\\
&(1-\lambda)\left(\frac{t+r}{t}-\lambda\right)\left(2\lambda - \frac{t+r}{t}\right)^{-1+\mu}, \qquad
&&\frac{t+r+1}{2t}\leq \lambda\leq \frac{t-r}{t},
\\
& &&  \hskip2.cm  \text{ provided } \frac{t+r+1}{2t}\leq  \frac{t-r}{t},
\\
&\frac{(1-\lambda)t}{\mu r}\left(\frac{t-r}{t}\right)^\mu,\qquad
&&\max\left(\frac{t-r}{t},\frac{t+r+1}{2t}\right) \leq \lambda\leq 1.
\endaligned
\right.
$$
\end{lemma}

\begin{proof}[Proof of Proposition \ref{Linfini wave}]
From the explicit expression
\bel{eq:form-onde}
u(t,x) = \frac{1}{4\pi}\int_2^t\frac{1}{t-\sbar}\int_{|y| = t-\sbar} f(\sbar,x-y) \, d\sigma d\sbar,
\ee
in which  the integration is made on the intersection of the cone $\big\{(\sbar,y) \, / \, |y-x| = t-\sbar, 2\leq \sbar\leq t\big\}$ and $\big\{(t,x) \, / \, r<t-1, t^2-r^2\leq s^2, t\geq 2 \big\}$, we obtain
$$
\aligned
|u(t,x)|
&\leq \frac{C_f}{4\pi}\int_2^t\int_{|y| = t-\sbar,|x-y|\leq \sbar-1}\frac{\sbar^{-2-\nu}(\sbar-|x-y|)^{-1+\mu}}{t-\sbar}d\sigma d\sbar
\\
& = \frac{C_f}{4\pi t^{1+\nu-\mu}}\int_{\frac{2}{t}}^1\int_{|y'|= 1-\lambda,|\frac{x}{t}-y'|\leq \lambda-t^{-1}}
\frac{(1-\lambda)^{-1}\lambda^{-2-\nu}d\sigma d\lambda}{\big(\lambda-\big|\frac{x}{t} - y'\big|\big)^{1-\mu}}
\quad(\lambda:= \sbar/t,\quad y' := y/t)
\\
& = \frac{C_f}{4\pi t^{1+\nu-\mu}}\int_{\frac{2}{t}}^1(1-\lambda)^{-1}\lambda^{-2-\nu} 
\int_{|y'| = 1-\lambda,|\frac{x}{t}-y'|\leq \lambda-t^{-1}} \frac{d\sigma}{\big(\lambda-\big|\frac{x}{t} - y'\big|\big)^{1-\mu}} \, d\lambda. 
\endaligned
$$
When $|\frac{x}{t} - y'|\leq \lambda - t^{-1}$ holds, we obtain $\frac{t-r+1}{2t}\leq\lambda\leq 1$. For  convenience in the notation, in the following calculation we replace $y'$ by $y$. We first assume that $r>0$ and we distinguish between two main cases:

\

\noindent{\bf Case 1:} $\frac{t-r}{t}> \frac{t+r+1}{2t}\Leftrightarrow r\leq \frac{t-1}{3}$.
In Lemma \ref{wave lemma1}, all three cases are possible:
$$ 
\aligned
|u(t,x)|
\leq
& \frac{C_f}{4\pi t^{1+\nu-\mu}}\int_{\frac{t-r+1}{2t}}^1(1-\lambda)^{-1}\lambda^{-2-\nu} 
\int_{|y| = 1-\lambda,|\frac{x}{t}-y|\leq \lambda-t^{-1}} \frac{d\sigma}{\big(\lambda-\big|\frac{x}{t} - y\big|\big)^{1-\mu}} \, d\lambda
\\
\lesssim &\frac{C_f}{\mu t^{1+\nu-\mu}} \int_{\frac{t-r+1}{2t}}^{\frac{t+r+1}{2t}}(1-\lambda)^{-1}\lambda^{-2-\nu} \frac{\lambda t(1-\lambda)}{r}\left(\frac{t-r}{t}\right)^{\mu} d\lambda
\\
&+ \frac{C_f}{ t^{1+\nu-\mu}}\int_{\frac{t+r+1}{2t}}^{\frac{t-r}{t}}
(1-\lambda)^{-1}\lambda^{-2-\nu}(1-\lambda)\left(\frac{t+r}{t}-\lambda\right)\left(2\lambda - \frac{t+r}{t}\right)^{-1+\mu} \, d\lambda
\\
&+ \frac{C_f}{\mu t^{1+\nu-\mu}}\int_{\frac{t-r}{t}}^{1}
(1-\lambda)^{-1}\lambda^{-2-\nu}\frac{(1-\lambda)t}{ r}\left(\frac{t-r}{t}\right)^\mu d\lambda, 
\endaligned
$$
thus 
$$ 
\aligned
|u(t,x)|
\lesssim
&\frac{C_f}{\mu t^{1+\nu-\mu}}\frac{t}{r}\bigg(\frac{t-r}{t}\bigg)^{\mu}\int_{\frac{t-r+1}{2t}}^{\frac{t+r+1}{2t}}\lambda^{-1-\nu} \, d\lambda
\\
&+\frac{C_f}{t^{1+\nu-\mu}}\int_{\frac{t+r+1}{2t}}^{\frac{t-r}{t}}\lambda^{-2-\nu}\left(\frac{t+r}{t}-\lambda\right)\left(2\lambda - \frac{t+r}{t}\right)^{-1+\mu} \, d\lambda
\\
&+\frac{C_f}{\mu t^{1+\nu-\mu}}\frac{t}{r}\left(\frac{t-r}{t}\right)^{\mu}\int_{\frac{t-r}{t}}^1\lambda^{-2-\nu} \, d\lambda.
\endaligned
$$
For the first integral, we recall that $r\leq \frac{t-1}{3}$ and  that $0<|\nu|\leq1/2$, and write 
$$
\frac{t}{r}\int_{\frac{t-r+1}{2t}}^{\frac{t+r+1}{2t}}\lambda^{-1-\nu} \, d\lambda
\lesssim \bigg(\frac{t}{t-r}\bigg)^{1+\nu} \lesssim 1, 
$$
so that 
$$
\left|\frac{C_f}{\mu t^{1+\nu-\mu}}\frac{t}{r}\bigg(\frac{t-r}{t}\bigg)^{\mu}\int_{\frac{t-r+1}{2t}}^{\frac{t+r+1}{2t}}\lambda^{-1-\nu} \, d\lambda\right | \lesssim C_f\mu^{-1}(t-r)^{\mu}t^{-1-\nu}.
$$
Next, for the second integral in the right-hand-side,   
 we just remark that
$$
\aligned
&\int_{\frac{t+r+1}{2t}}^{\frac{t-r}{t}}\lambda^{-2-\nu}\left(\frac{t+r}{t}-\lambda\right)\left(2\lambda - \frac{t+r}{t}\right)^{-1+\mu} \, d\lambda
\\
&\lesssim \int_{\frac{t+r+1}{2t}}^{\frac{t-r}{t}}\left(2\lambda - \frac{t+r}{t}\right)^{-1+\mu} \, d\lambda
 = \frac{1}{\mu}\left(2\lambda - \frac{t+r}{t}\right)^{\mu}\bigg|_{\frac{t+r+1}{2t}}^{\frac{t-r}{t}} 
\lesssim \frac{1}{\mu}.
\endaligned
$$
This leads to
$$
\frac{C_f}{t^{1+\nu-\mu}}\int_{\frac{t+r+1}{2t}}^{\frac{t-r}{t}}\lambda^{-2-\nu}\left(\frac{t+r}{t}-\lambda\right)\left(2\lambda - \frac{t+r}{t}\right)^{-1+\mu} \, d\lambda
\lesssim
 \frac{C_f}{\mu t^{1+\nu-\mu}}.
$$
For the third term, in view of $\frac{t-r}{t}\geq \frac{t+r+t1}{2t}\geq \frac{1}{2}$, we obtain
$$
\aligned
\frac{C_f}{\mu t^{1+\nu-\mu}}\frac{t}{r}\left(\frac{t-r}{t}\right)^{\mu}\int_{\frac{t-r}{t}}^1\lambda^{-2-\nu} \, d\lambda
\lesssim &\frac{C_f}{\mu t^{1+\nu-\mu}}\frac{t}{r}\left(\frac{t-r}{t}\right)^{\mu}\int_{\frac{t-r}{t}}^1 2^{2+\mu} \, d\lambda
\\
\lesssim & C_f\mu^{-1} (t-r)^{\mu}t^{-1-\nu}.
\endaligned
$$
So we conclude that in the case $0<r\leq \frac{t-1}{3}$, $|u(t,x)|\lesssim C_f\mu^{-1}(t-r)^{\mu}t^{-1-\nu}$.

\

\noindent {\bf Case 2:} $\frac{t+r+1}{2t}\geq \frac{t-r}{t}\Leftrightarrow r\geq \frac{t-1}{3}$.
The second case in Lemma \ref{wave lemma1} is not possible, and we have
$$
|u(t,x)|
\lesssim\frac{C_f}{\mu t^{1+\nu-\mu}}\left(\frac{t-r}{t}\right)^{\mu}
\left(\int_{\frac{t-r+1}{2t}}^{\frac{t+r+1}{2t}}\lambda^{-1-\nu} \, d\lambda + \int_{\frac{t+r+1}{2t}}^1\lambda^{-2-\nu} \, d\lambda\right).
$$
Since $\frac{t+r+1}{2t}\geq 1/2$, the second integral is bounded by a constant $C$. For the first integral, we see that when $\nu> 0$,
$$
\int_{\frac{t-r+1}{2t}}^{\frac{t+r+1}{2t}}\lambda^{-1-\nu} \, d\lambda
\lesssim
 \frac{1}{\nu}\left(\frac{t-r+1}{t}\right)^{-\nu}.
$$
So in this case when $\nu>0$, we obtain
$|u(t,x)|\lesssim C_f(\mu\nu)^{-1}(t-r)^{\mu-\nu}t^{-1}$.

When $\nu<0$, we write 
$$
\int_{\frac{t-r+1}{2t}}^{\frac{t+r+1}{2t}}\lambda^{-1-\nu} \, d\lambda
\lesssim
\frac{1}{|\nu|}\left(\frac{t+r+1}{t}\right)^{-\nu}
\lesssim \frac{1}{|\nu|}
$$
and, therefore, we obtain
$|u(t,x)|\lesssim C_f(\mu|\nu|)^{-1}(t-r)^{\mu}t^{-1-\nu}$.

When $r=0$, we make the following direct calculation, remark that in this case, $\frac{t+1}{2t}\leq \lambda\leq 1$, by \eqref{last 1}:
$$
\aligned
|u(t,x)| &\leq \frac{C_f}{4\pi t^{1+\nu-\mu}}\int_{\frac{t-r+1}{2t}}^1(1-\lambda)^{-1}\lambda^{-2-\nu} \, d\lambda
\int_{|y| = 1-\lambda,|\frac{x}{t}-y|\leq \lambda-t^{-1}} \frac{d\sigma}{\big(\lambda-\big|\frac{x}{t} - y\big|\big)^{1-\mu}}
\\
&\lesssim\frac{C_f}{t^{1+\nu-\mu}}\int_{\frac{t+1}{2t}}^1(1-\lambda)^{-1}\lambda^{-2-\nu}(2-\lambda)^{-1+\mu}(1-\lambda)^2 d\lambda
\\
&\lesssim \frac{C_f}{\mu}t^{-1-\nu+\mu} 
=\frac{C_f}{\mu}(t-r)^{\mu}t^{-1-\nu} \qquad \text{(since $r=0$),} 
\endaligned
$$
which completes the proof.   
\end{proof}

\begin{proof}[Proof of Lemma \ref{wave lemma1}] When $r=0$, the estimate is trivial. When $r>0$,
without loss of generality, let $x = (r,0,0)$. The surface $S_{\lambda} := \{|y| = 1-\lambda\}\cap \{\left|\frac{x}{t} - y\right|\leq \lambda - t^{-1}\}$ is parameterized as follows:
\begin{itemize}

\item $\theta$: angle from $(1,0,0)$ to $y$ with $0\leq \theta\leq \pi$,

\item $\phi$: angle from the plane determined by $(1,0,0)$ and $(0,1,0)$ to the plane determined by $y$ and $(1,0,0)$,
with $0\leq \phi\leq 2\pi$. 

\end{itemize} 
Then, we have
$y = (1-\lambda)\big(\cos\theta,\sin\theta  \cos\phi,\sin\theta \sin\phi\big)$ and we distinguish between two cases, as follows.

\

\noindent {\bf Case 1.}
When $\frac{t-r+1}{2t}\leq\lambda\leq \frac{t+r+1}{2t}$, 
we only have a part of the sphere $\{|y|=1-\lambda\}$ contained in the ball $\{\left|\frac{x}{t}-y\right|\leq \lambda - t^{-1}\}$ where $\cos(\theta) \geq \frac{(r/t)^2+(1-\lambda)^2 - \left(\lambda-t^{-1}\right)^2}{(2r/t)(1-\lambda)}$. So we set 
$\theta_0 := \arccos\left(\frac{(r/t)^2+(1-\lambda)^2 - \left(\lambda-t^{-1}\right)^2}{(2r/t)(1-\lambda)}\right)$
and see that 
$$
\lambda - \big|\frac{x}{t} - y\big|  = \lambda -\sqrt{\frac{r^2}{t^2} + (1-\lambda)^2 - 2\frac{r}{t}(1-\lambda)\cos\theta}
$$
and
$d\sigma = (1-\lambda)^2\sin(\theta)d\theta d\phi$. 
The integral is estimated as follows:
$$
\aligned
& \hskip-.3cm \int_{|y| = 1-\lambda,|\frac{x}{t}-y|\leq \lambda-t^{-1}} \frac{d\sigma}{\big(\lambda-\big|\frac{x}{t} - y\big|\big)^{1-\mu}}
\\
=&\int_0^{2\pi}d\phi\int_0^{\theta_0}(1-\lambda)^2\sin\theta
\bigg(\lambda -\sqrt{\frac{r^2}{t^2} + (1-\lambda)^2 - 2\frac{r}{t}(1-\lambda)\cos\theta}\bigg)^{-1+\mu}d\theta
\\
=&2\pi \int_0^{\theta_0}(1-\lambda)^2\sin\theta
\bigg(\lambda -\sqrt{\frac{r^2}{t^2} + (1-\lambda)^2 - 2\frac{r}{t}(1-\lambda)\cos\theta}\bigg)^{-1+\mu}d\theta
\\
=&-2\pi(1-\lambda)^2\int_0^{\theta_0}
\bigg(\lambda -\sqrt{\frac{r^2}{t^2} + (1-\lambda)^2 - 2\frac{r}{t}(1-\lambda)\cos\theta}\bigg)^{-1+\mu}d\cos\theta
\endaligned
$$
thus by setting $\omega = \cos \theta$ 
$$
\aligned
& \hskip-.3cm \int_{|y| = 1-\lambda,|\frac{x}{t}-y|\leq \lambda-t^{-1}} \frac{d\sigma}{\big(\lambda-\big|\frac{x}{t} - y\big|\big)^{1-\mu}}
\\
=&2\pi(1-\lambda)^2\int_{\cos\theta_0}^1
\bigg(\lambda -\sqrt{\frac{r^2}{t^2} + (1-\lambda)^2 - 2\frac{r}{t}(1-\lambda)\omega}\bigg)^{-1+\mu}d\omega
\\
=&\frac{\pi t(1-\lambda)}{r}\int_{|\frac{r}{t} - (1-\lambda)|^2}^{(\lambda-t^{-1})^2}\big(\lambda - \sqrt{\gamma}\big)^{-1+\mu} \, d\gamma
=2 \frac{\pi t(1-\lambda)}{r}\int_{t^{-1}}^{\lambda-|\frac{r}{t}-(1-\lambda)|}\zeta^{-1+\mu}(\lambda-\zeta) \, d\zeta, 
\endaligned
$$
where we have used $\gamma = \frac{r^2}{t^2} + (1-\lambda)^2 - 2\frac{r}{t}(1-\lambda)\omega$ and
$\zeta := \lambda - \sqrt{\gamma}$. Then, we distinguish between the following two sub-cases.

\

\noindent {\bf Case 1.1:} $\frac{r}{t}\leq 1-\lambda$ or, equivalently, $\lambda \leq \frac{t-r}{t}$. We now find
$$
\aligned
 &2 \frac{\pi t(1-\lambda)}{r}\int_{t^{-1}}^{\lambda-|\frac{r}{t}-(1-\lambda)|}\zeta^{-1+\mu}(\lambda-\zeta) \, d\zeta
\\
&=2 \frac{\pi t(1-\lambda)}{r}\int_{t^{-1}}^{2(\lambda-\frac{t-r}{2t})}\zeta^{-1+\mu}(\lambda-\zeta) \, d\zeta
\lesssim \frac{\lambda t(1-\lambda)}{\mu r}\frac{(t-r)^{\mu}}{t^{\mu}}.
\endaligned
$$

\

\noindent {\bf Case 1.2:}  $1-\lambda<\frac{r}{t}$ or, equivalently, $\lambda>\frac{t-r}{t}$. We find
$$
\aligned
&2 \frac{\pi t(1-\lambda)}{r}\int_{t^{-1}}^{\lambda-|\frac{r}{t}-(1-\lambda)|}\zeta^{-1+\mu}(\lambda-\zeta) \, d\zeta
\\
&=2 \frac{\pi t(1-\lambda)}{r}\int_{t^{-1}}^{\frac{t-r}{t}}\zeta^{-1+\mu}(\lambda-\zeta) \, d\zeta
\lesssim
 \frac{\lambda t(1-\lambda)}{\mu r}\frac{(t-r)^{\mu}}{t^{\mu}}.
\endaligned
$$

\

\noindent {\bf Case 2.} When $\frac{t+r+1}{2t}\leq \lambda\leq 1$, the sphere $\{|y|=1-\lambda\}$ is entirely contained in $\{\left|(x/t)-y\right|\leq \lambda-t^{-1} \}$: 
$$
\aligned
&\int_{|y| = 1-\lambda,|\frac{x}{t}-y|\leq \lambda-t^{-1}} \frac{d\sigma}{\big(\lambda-\big|\frac{x}{t} - y\big|\big)^{1-\mu}}
=\int_{|y| = 1-\lambda}\frac{d\sigma}{\big(\lambda-\big|\frac{x}{t} - y\big|\big)^{1-\mu}}
\\
&=2\pi \int_0^{\pi}(1-\lambda)^2\sin\theta
\bigg(\lambda -\sqrt{\frac{r^2}{t^2} + (1-\lambda)^2 - 2\frac{r}{t}(1-\lambda)\cos\theta}\bigg)^{-1+\mu}d\theta
\\
&=2\pi(1-\lambda)^2\int_{-1}^1
\bigg(\lambda -\sqrt{\frac{r^2}{t^2} + (1-\lambda)^2 - 2\frac{r}{t}(1-\lambda)\omega}\bigg)^{-1+\mu}d\omega
\endaligned
$$
and thus 
$$
\aligned
\int_{|y| = 1-\lambda,|\frac{x}{t}-y|\leq \lambda-t^{-1}} \frac{d\sigma}{\big(\lambda-\big|\frac{x}{t} - y\big|\big)^{1-\mu}}
&=2 \frac{\pi t(1-\lambda)}{r}\int_{\lambda - (\frac{r}{t} + (1-\lambda))}^{\lambda-|\frac{r}{t}-(1-\lambda)|}\zeta^{-1+\mu}(\lambda-\zeta) \, d\zeta
\\
&=2 \frac{\pi t(1-\lambda)}{r}\int_{2\lambda - \frac{t+r}{t}}^{\lambda-|\frac{r}{t}-(1-\lambda)|}\zeta^{-1+\mu}(\lambda-\zeta) \, d\zeta.
\endaligned
$$
We now distinguish between two sub-cases.

\

\noindent {\bf Case 2.1: } When $\frac{r}{t}\leq 1-\lambda$ or, equivalently, $\lambda \leq \frac{t-r}{t}$, we find
$$
\aligned
&2 \frac{\pi t(1-\lambda)}{r}\int_{2\lambda - \frac{t+r}{t}}^{\lambda-|\frac{r}{t}-(1-\lambda)|}\zeta^{-1+\mu}(\lambda-\zeta) \, d\zeta
\\
&=2 \frac{\pi t(1-\lambda)}{r}\int_{2\lambda - \frac{t+r}{t}}^{2\lambda-\frac{t-r}{t}}\zeta^{-1+\mu}(\lambda-\zeta) \, d\zeta
\leq C(1-\lambda)\left(\frac{t+r}{t}-\lambda\right)\left(2\lambda - \frac{t+r}{t}\right)^{-1+\mu},
\endaligned
$$
where we have observed that in the integral the function $\zeta^{-1+\mu}(\lambda-\zeta)$ is decreasing and we can bound this integral by the value of the function taken at the inferior boundary (which is $2\lambda - \frac{t+r}{t}$) times the length of the interval which is $2r/t$.

\

\noindent {\bf Case 2.2:} When $1-\lambda<\frac{r}{t}$ or, equivalently, $\lambda>\frac{t-r}{t}$, we have
$$
\aligned
& 2 \frac{\pi t(1-\lambda)}{r}\int_{2\lambda - \frac{t+r}{t}}^{\lambda-|\frac{r}{t}-(1-\lambda)|}\zeta^{-1+\mu}(\lambda-\zeta) \, d\zeta
\\
&=2 \frac{\pi t(1-\lambda)}{r}\int_{2\lambda - \frac{t+r}{t}}^{\frac{t-r}{t}}\zeta^{-1+\mu}(\lambda-\zeta) \, d\zeta
\leq C(1-\lambda)\frac{t}{r}\int_{2\lambda - \frac{t+r}{t}}^{\frac{t-r}{t}}\zeta^{-1+\mu}d\zeta
\\
&\leq\frac{C(1-\lambda)t}{\mu r}\zeta^{\mu}\bigg|_0^{\frac{t-r}{r}} =\frac{ C(1-\lambda)t}{\mu r}\left(\frac{t-r}{t}\right)^\mu.
\endaligned
$$
When $\frac{t+r+1}{2t}\leq \frac{t-r}{t}$, both case above may occur, while only Case 2.2 is possible if the opposite inequality holds true. 
\end{proof}


\subsection{Statement for the Klein-Gordon equation}
\label{sec subsec KG-supnorm}

Consider (sufficiently smooth and spatially compactly supported) solutions to a Klein-Gordon equation on a curved space and, specifically,
\bel{Linfini KG eq}
\aligned
&-\Boxt_g v + c^2 v = f,
\\
&v|_{\Hcal_2} = v_0,
\qquad
\del_t v|_{\Hcal_2} = v_1,
\endaligned
\ee
with initial data $v_0, v_1$ given on $\Hcal_2$ and compactly supported in $\Hcal_2\cap \Kcal$, 
and the metric has the form $g^{\alpha\beta} = m^{\alpha\beta} - h^{\alpha\beta}$ with $h^{\alpha\beta}$ is spatially compactly supported in $\Kcal$ 
with $\sup |\hb^{00}| \leq 1/3$.

Before we can state our estimate, we need some notation.
Given a constant $C>0$ and using the notation $s = \sqrt{t^2 - r^2}$, we consider the function 
$$
h_{t,x}(\lambda):= \hb^{00}(\lambda t/s,\lambda x/s),
$$
and, by denoting by $h_{t,x}'(\lambda)$ for the derivative with respect to $\lambda$, 
$$
\aligned
h_{t,x}'(\lambda) 
& = {t \over s} \del_t \hb^{00}(\lambda t/s,\lambda x/s) + {x^a \over s} \del_a \hb^{00}(\lambda t/s,\lambda x/s)
\\
& = {t \over s} \newperp \hb^{00}(\lambda t/s,\lambda x/s).
\endaligned 
$$
We set
\bel{eq:szero}
s_0 :=\left
\{
\aligned
& 2, \quad &&0\leq r/t \leq 3/5,
\\
& \sqrt{\frac{t+r}{t-r}},\quad &&3/5\leq r/t\leq 1, 
\endaligned
\right.
\ee
and introduce the following function $V$ which is defined by distinguishing between the regions ``near" and ``far" from
the light cone:
$$
V:=
 \left\{
\aligned
 & \big( \|v_0\|_{L^\infty(\Hcal_2)} + \|v_1\|_{L^\infty(\Hcal_2)} \big)
\Big(1+\int_2^s|h_{t,x}'(\sbar)|e^{C\int_\sbar^s|h_{t,x}'(\lambda)|d\lambda} \, d\sbar \Big)
\\
& \hskip4.cm  + F(s) + \int_2^s F(\sbar)|h_{t,x}'(\lambda)|e^{C\int_\sbar^s|h_{t,x}'(\lambda)|d\lambda} \, d\sbar,
\hskip.cm  && 0\leq r/t\leq 3/5,
\\
& 
F(s) + \int_{s_0}^s F(\sbar)|h_{t,x}'(\sbar)|e^{C\int_\sbar^s|h_{t,x}'(\lambda)|d\lambda} \, d\sbar,
\hskip2.cm  &&3/5<r/t<1, 
\endaligned
\right.
$$ 
with 
$$
F(\sbar): = \int_{s_0}^\sbar \bigg(\big(R_1[v] + R_2[v] + R_3[v]\big)(\lambda t/s,\lambda x/s) + \lambda^{3/2}f(\lambda t/s,\lambda x/s) \bigg) d\lambda
$$
and
$$
\aligned
R_1[v] &:= s^{3/2}\sum_a\delb_a\delb_a v  + \frac{x^ax^b}{s^{1/2}}\delb_a\delb_b v + \frac{3}{4s^{1/2}} v + \sum_a\frac{3x^a}{s^{1/2}}\delb_a v,
\\
R_2[v] &:=\hb^{00}\bigg(\frac{3v}{4s^{1/2}} + 3s^{1/2}\delb_0 v\bigg)
- s^{3/2}\big(2\hb^{0b}\delb_0\delb_bv + \hb^{ab}\delb_a\delb_bv + h^{\alpha\beta}\del_\alpha\Psib^{\beta'}_\beta\,\delb_{\beta'}v\big),
\\
R_3[v] &:= \hb^{00}\bigg(2x^as^{1/2}\delb_0\delb_a v + \frac{2x^a}{s^{1/2}}\delb_a v +\frac{x^ax^b}{s^{1/2}}\delb_a\delb_bv\bigg).
\endaligned
$$

\noindent With these notations, our result is as follows.

\begin{proposition}[A sup-norm estimate for the Klein-Gordon equation with source]
\label{Linfini KG}
Considering the Klein-Gordon problem \eqref{Linfini KG eq} for a every sufficiently smooth and spatially compactly supported solution $v$ defined the future region $\Kcal_{[2, +\infty)}$, one has (for all relevant $(t,x)$)
\bel{Linfty KG ineq a}
s^{3/2}|v(t,x)| + (s/t)^{-1}s^{3/2}|\newperp  v(t,x)|\lesssim V(t,x),
\ee 
\end{proposition}

This result is motivated by a pioneering work by Klainerman \cite{Klainerman85} and the decomposition in Lemma~\ref{lem 0 K-G} below. An analogue statement in two spatial dimensions and flat Minkowski spacetime is discussed in \cite{Ma}; see also the earlier work  \cite{Delort04}.


\subsection{Proof of the sup-norm estimate for the Klein-Gordon equation}

We begin with two technical results.

\begin{lemma}[A decomposition identity]
\label{lem 0 K-G}
For every sufficiently smooth solution $v$ to \eqref{Linfini KG eq}, the function
$$
w_{t,x}(\lambda) := \lambda^{3/2}v(\lambda t/s, \lambda x/s), \qquad (t,x)\in\Kcal, 
$$ 
satisfies the following second-order ODE in $\lambda$
$$
\aligned
&
\frac{d^2}{d\lambda^2}w_{t,x}(\lambda) + \frac{c^2}{1+\hb^{00}(\lambda t/s,\lambda x/s)} w_{t,x}(\lambda)
\\
&= \big(1+\hb^{00}(\lambda t/s,\lambda x/s)\big)^{-1}\big(R_1[v] + R_2[v] + R_3[v] +s^{3/2}f\big)(\lambda t/s,\lambda x/s).
\endaligned
$$
\end{lemma}

\begin{lemma}[Technical ODE estimate]
\label{lem 1 K-G}
Let $G$ be a function defined on an interval $[s_0,s_1]$ and
satisfying $\sup | G| \leq 1/3$.
and $k$ be an integrable function defined on $[s_0,s_1]$.
Then, the solution $z$ to the ordinary differential equation
\be
\aligned
& z''(\lambda) + \frac{c^2}{1+ G(\lambda)} z(\lambda) = k(\lambda),
\\
& z(s_0) = z_0, \qquad z'(s_0) = z_1,
\endaligned
\ee
with prescribed initial data $z_0, z_1$ 
satisfies the uniform bound
\bel{Linfini ineq ODE}
\aligned
&
|z(s)| + |z'(s)|
\lesssim  \big(|z_0| + |z_1| + K(s)\big) + \int_{s_0}^s\Big(|z_0|+|z_1| + K(\sbar)\Big) \, | G'(\sbar)|e^{C\int_\sbar^s| G'(\lambda)|d\lambda} \, d\sbar
\endaligned
\ee 
for all $s \in [s_0, s_1]$ and with $K(s) := \int_{s_0}^s |k(\sbar)| \, d\sbar$ and a suitable constant $C>0$.
\end{lemma}

\begin{proof}[Proof of Lemma~\ref{lem 0 K-G}]
{\bf 1. Decomposition of the flat wave operator.} By recalling $s=\sqrt{t^2-r^2}$ and $r=|x|$, an elementary calculation shows that the flat wave operator $\Box$ in the hyperboloidal frame reads
\bel{Hyper box1}
-\Box = \delb_0 \delb_0 - \sum_a\delb_a\delb_a + 2 \sum_a \frac{x^a}{s}\delb_0\delb_a + \frac{3}{s}\delb_0.
\ee
Given a function $v$, we can set
$$
w(t,x) = s^{3/2}v(t,x) = (t^2-|x|^2)^{3/4}v(t,x),
$$
and obtain
\bel{Hyper box2}
-s^{3/2}\Box v = \delb_0\delb_0 w - \sum_a\delb_a\delb_a w + 2\sum_a \frac{x^a}{s}\delb_0\delb_a w - \frac{3w}{4s^2} - \sum_a\frac{3x^a\delb_a w}{s^2}.
\ee
Again, we define a function of a single variable by 
$$
w_{t,x}(\lambda) := w(\lambda t/s,\lambda x/s)
= \lambda^{3/2}v(\lambda t/s,\lambda x/s).
$$
We see that
$$
\frac{d}{d\lambda}w_{t,x}(\lambda) = \big(\delb_0 + s^{-1}x^a\delb_a\big)w(\lambda t/s,\lambda x/s)
= \frac{t}{s}\newperp  w\left(\lambda t/s,\lambda x/s\right)
$$
and
\bel{Hyper 2order}
\frac{d^2}{d\lambda^2}w_{t,x}(\lambda) = \bigg(\delb_0\delb_0 + 2 \frac{x^a}{s}\delb_0\delb_a +\frac{x^ax^b}{s^2}\delb_a\delb_b\bigg)w(\lambda t/s,\lambda x/s).
\ee
Combining this with \eqref{Hyper box2} and recalling that $w(t,x) = s^{3/2}v(t,x)$, we get 
\bel{Hyper 2order-box1}
\aligned
&\bigg(\delb_0\delb_0 + 2 \frac{x^a}{s}\delb_0\delb_a +\frac{x^ax^b}{s^2}\delb_a\delb_b\bigg)w
\\
& = -s^{3/2}\Box v + \sum_a\delb_a\delb_a w + \frac{x^ax^b}{s^2}\delb_a\delb_b w + \frac{3}{4s^2}w + \sum_a\frac{3x^a}{s^2}\delb_a w
  = - s^{3/2}\Box v + R_1[v].
\endaligned
\ee


\noindent {\bf 2.  Decomposition of the curved wave operator.} We write
$$
-\Box v = h^{\alpha\beta}\del_\alpha\del_\beta v  - c^2v + f
$$
and, by performing a change of frame,
$$
\aligned
h^{\alpha\beta}\del_\alpha\del_\beta v =& \hb^{\alpha\beta}\delb_\alpha\delb_\beta v + h^{\alpha\beta}\del_\alpha\Psib^{\beta'}_\beta\,\delb_{\beta'}v
\\
=& \hb^{00}\delb_0\delb_0 v
+ 2\hb^{0b}\delb_0\delb_bv + \hb^{ab}\delb_a\delb_bv + h^{\alpha\beta}\del_\alpha\Psib^{\beta'}_\beta\,\delb_{\beta'}v.
\endaligned
$$
Then, we obtain
$$
\aligned
-s^{3/2}\Box v
=& - s^{3/2}\hb^{00}\delb_0\delb_0 v
      - s^{3/2}\big(2\hb^{0b}\delb_0\delb_bv + \hb^{ab}\delb_a\delb_bv + h^{\alpha\beta}\del_\alpha\Psib^{\beta'}_\beta\,\delb_{\beta'}v\big)
 - c^2s^{3/2}v + s^{3/2}f
\\
=& - \hb^{00}\delb_0\delb_0 \big(s^{3/2}v\big)- c^2s^{3/2}v
\\
& + \hb^{00}\bigg(\frac{3v}{4s^{1/2}} + 3s^{1/2}\delb_0 v\bigg)
 -  s^{3/2}\big(2\hb^{0b}\delb_0\delb_bv + \hb^{ab}\delb_a\delb_bv + h^{\alpha\beta}\del_\alpha\Psib^{\beta'}_\beta\,\delb_{\beta'}v\big) + s^{3/2}f,
\endaligned
$$
and we conclude with
\bel{Hyper KG}
\aligned
-s^{3/2}\Box v
& =  - \hb^{00}\delb_0\delb_0 w - c^2w + \hb^{00}\bigg(\frac{3v}{4s^{1/2}} + 3s^{1/2}\delb_0 v\bigg)
\\
& \quad -  s^{3/2}\big(2\hb^{0b}\delb_0\delb_bv + \hb^{ab}\delb_a\delb_bv + h^{\alpha\beta}\del_\alpha\Psib^{\beta'}_\beta\,\delb_{\beta'}v\big) + s^{3/2}f
\\
&= - \hb^{00}\delb_0\delb_0 w - c^2w + R_2[v] + s^{3/2}f.
\endaligned
\ee
We then combine \eqref{Hyper 2order-box1} with \eqref{Hyper KG} and obtain
\bel{Hyper 2order-box2}
\delb_0\delb_0w + 2 \frac{x^a}{s}\delb_0\delb_a w+\frac{x^ax^b}{s^2}\delb_a\delb_bw
- \hb^{00}\delb_0\delb_0 w + c^2w
= R_1[v] + R_2[v] + s^{3/2}f.
\ee

\

\noindent {\bf 3. Conclusion.} We continue with \eqref{Hyper 2order-box2} and write
$$
\aligned
& \big(1 + \hb^{00}\big)\bigg(\delb_0\delb_0 + 2 \frac{x^a}{s}\delb_0\delb_a +\frac{x^ax^b}{s^2}\delb_a\delb_b\bigg)w + c^2w
\\
& = \hb^{00}\bigg( 2 \frac{x^a}{s}\delb_0\delb_a +\frac{x^ax^b}{s^2}\delb_a\delb_b\bigg)w + R_1[v] + R_2[v] + s^{3/2}f
\endaligned
$$
and, so, we have
\bel{Linfini eq1}
\aligned
& \bigg(\delb_0\delb_0 + 2 \frac{x^a}{s}\delb_0\delb_a +\frac{x^ax^b}{s^2}\delb_a\delb_b\bigg)w
+ \frac{c^2w}{1 + \hb^{00}} 
\\
& = \big(1 + \hb^{00}\big)^{-1}\big(R_1[v] + R_2[v] + R_3[v] +s^{3/2}f\big).
\endaligned
\ee
It follows that
\be
\label{Linfini ineq ODE0}
\aligned
& \frac{d^2}{d\lambda^2}w_{t,x}(\lambda) + \frac{c^2w_{t,x}(\lambda)}{1 + \hb^{00}(\lambda t/s,\lambda x/s)}
\\
&= \big(1 + \hb^{00}(\lambda t/s,\lambda x/s)\big)^{-1}\big(R_1[v] + R_2[v] + R_3[v] +s^{3/2}f\big)(\lambda t/s,\lambda x/s).
\endaligned
\ee
\end{proof}


\begin{proof}[Proof of Lemma~\ref{lem 1 K-G}] We simply need to integrate out the ODE.
We consider the vector field $b(\lambda) = \big(z(\lambda),z'(\lambda)\big)^T$ and the matrix $A(\lambda) := \left(
\begin{array}{cc}
0 &1
\\
-c^2(1+G)^{-1} &0
\end{array}
\right).
$
We write
$
b' = Ab +
\left(
\begin{array}{c}
0\\
k
\end{array}
\right).
$
and introduce the diagonalization of $A = PQP^{-1}$ with
$$
Q =
\left(
\begin{array}{cc}
ic\big(1+G\big)^{-1/2} &0
\\
0 &-ic\big(1+G \big)^{-1/2}
\end{array}
\right)
$$
and
$$
P =
\left(
\begin{array}{cc}
1 & 1
\\
\frac{ic}{(1+G)^{1/2}} &-\frac{ic}{(1+G)^{1/2}}
\end{array}
\right),
\qquad
\quad
P^{-1} =
\left(
\begin{array}{cc}
1/2 & \frac{(1+G)^{1/2}}{2ic}
\\
1/2 & -\frac{(1+G)^{1/2}}{2ic}
\end{array}
\right).
$$
We thus have
$b' = PQP^{-1}b +
\left(
\begin{array}{c}
0
\\
k
\end{array}
\right)$,
which leads us to
$$
\big(P^{-1}b\big)' = Q\big(P^{-1}b\big) + \big(P^{-1}\big)'b + P^{-1}
\left(
\begin{array}{c}
0
\\
k
\end{array}
\right).
$$
We regard the term $\big(P^{-1}\big)'b$ as a source term and, by a standard formula,
$$
\aligned
P^{-1}b(s) = e^{\int_{s_0}^s Q(\sbar)d\sbar} P^{-1}b(s_0)
& + \int_{s_0}^s e^{\int_{\lambda}^s Q(\sbar)d\sbar}P^{-1}
\left(
\begin{array}{c}
0
\\
k
\end{array}
\right)
d\lambda
\\
& + \int_{s_0}^s e^{\int_{\lambda}^s Q(\sbar)d\sbar} \big(P^{-1}\big)'(\lambda)\,b(\lambda) \, d\lambda.
\endaligned
$$
Recall that when $\sup_{\lambda\in[1,s]} |G(\lambda)|\leq 1/3$, the norm of $P(\lambda)$ and $P^{-1}(\lambda)$ are bounded for $\lambda\in[s_0,s]$. We also remarks that the norm of $\big(P^{-1}\big)'(\lambda)$ is bounded by $C|G'(\lambda)|$ with $C$ a constant depending only on $c$, and the norm of $Q$ is also bounded by a constant $C>0$. Furthermore, we observe that
$$
\int_{\lambda}^sQ(\sbar)d\sbar =
\left(
\begin{array}{cc}
ic\int_\lambda^s (1+G)^{-1/2}(\sbar)d\sbar & 0
\\
0 & -ic\int_\lambda^s (1+G)^{-1/2}(\sbar)d\sbar
\end{array}
\right)
$$
and thus
$$
e^{\int_{\lambda}^sQ(\sbar)d\sbar} =
\left(
\begin{array}{cc}
e^{ic\int_\lambda^s (1+G)^{-1/2}(\sbar)d\sbar} & 0
\\
0 & e^{-ic\int_\lambda^s (1+G)^{-1/2}(\sbar)d\sbar}
\end{array}
\right).
$$
The norm of the matrix $e^{\int_{\lambda}^sQ(\sbar)d\sbar}$ is uniformly bounded by a constant, and the following estimate is now proven:
$$
|z(s)| + |z'(s)|\leq C (|z(s_0)|+|z'(s_0)|) + C \, K(s) + C\int_{s_0}^s| G'(\lambda)|\big(|z(\lambda)| + |z'(\lambda)|\big) \, d\lambda, 
$$
and we conclude with Gronwall's lemma.
\end{proof}

\begin{proof}[Proof of Proposition \ref{Linfini KG}]
The proof is based on a combination of the bounds \eqref{Linfini ineq ODE} and \eqref{Linfini ineq ODE0}.
By recalling the definition of $w_{t,x}(\lambda)$, we have 
$$
\aligned
&w_{t,x}(\lambda) = \lambda^{3/2}v(\lambda t/s, \lambda x/s),
\\
&w'_{t,x}(\lambda) = \frac{3}{2}\lambda^{1/2}v(\lambda t/s, \lambda x/s) + \frac{t}{s}\lambda^{3/2}\newperp v(\lambda t/s, \lambda x/s).
\endaligned
$$
That is, $w_{t,x}$ is the restriction of $w(t,x) = s^{3/2}v(t,x)$ on the line segment $\big\{(\lambda t/s, \lambda x/s),\lambda\in[s_0,s] \big\}$.
We then apply \eqref{Linfini ineq ODE} and \eqref{Linfini ineq ODE0}  to this line segment, with
$$
s_0 =\left
\{
\aligned
& 2, \quad 0\leq r/t \leq 3/5,
\\
& \sqrt{\frac{t+r}{t-r}},\quad 3/5\leq r/t\leq 1.
\endaligned
\right.
$$
This segment is the part of the line $\{(\lambda t/s, \lambda x/s)\}$ between the point $(t,x)$ and the boundary of $\Kcal_{[s_0,+\infty)}$.

Recall that $v$ is supported in $\Kcal$. and the restriction of $v$ on the initial hyperboloid $\Hcal_2$ is supported in $\Hcal_2\cap \Kcal$. We recall that when $3/5 \leq r/t\leq 1$, $w_{t,x}(s_0) = 0$ and when $0\leq r/t\leq 3/5$, $w_{t,x}(s_0)$ is determined by $v_0$.

When $0\leq r/t\leq 3/5$, we apply \eqref{Linfini ineq ODE} with $s_0 = 2$. When $\lambda = 2$, we write 
$
w_{t,x}(2) = w(2t/s,2x/s) = 2^{3/2}v(2t/s,2x/s) = 2^{3/2}v_0(2x/s),
$
and
$$
\aligned
w'_{s,x}(2) =& \frac{d}{d\lambda}\big(\lambda^{3/2}v(\lambda t/s,\lambda x/s)\big)\big|_{\lambda = 2}
\\
=& \frac{3\sqrt{2}}{2}v(2t/s,2x/s) + 2^{3/2}(s/t)^{-1}\newperp v(2t/s,2x/s)
\\
=& \frac{3\sqrt{2}}{2}v(2t/s,2x/s) + 2^{3/2}(s/t)^{-1}\del_tv(2t/s,2x/s) + 2^{3/2}(x^a/s)\del_av(2t/s,2x/s)
\\
=& \frac{3\sqrt{2}}{2}v_0(2x/s) + 2^{3/2}(x^a/s)\del_av_0(2x/s) + 2^{3/2}(s/t)^{-1}v_1(2t/s,2x/s).
\endaligned
$$
Recall that when $0\leq r/t\leq 3/5$, we have $4/5\leq s/t\leq 1$. So we see that $|w_{t,x}(s_0)| + |w'_{t,x}(s_0)|\leq C(\|v_0\|_{L^\infty(\Hcal_2)} + \|v_1\|_{L^\infty(\Hcal_2)})$. Then by \eqref{Linfini ineq ODE} and \eqref{Linfini ineq ODE0} we have
$$
\aligned
|w_{t,x}(s)| + |w'_{t,x}(s)| \leq
\, & C(\|v_0\|_{L^\infty(\Hcal_2)} + \|v_1\|_{L^\infty(\Hcal_2)}) + C F(s)
\\
&+ C(\|v_0\|_{L^\infty(\Hcal_2)} + \|v_1\|_{L^\infty(\Hcal_2)})\int_2^s|h_{t,x}'(\sbar)|e^{C\int_{\sbar}^s |h_{t,x}'(\lambda)| d\lambda} \, d\sbar
\\
&+ C\int_2^s F(\sbar)|h_{t,x}'(\sbar)|e^{C\int_\sbar^s |h_{t,x}'(\lambda)|d\lambda} \, d\sbar.
\endaligned
$$

We recall that when $3/5\leq r/t \leq 1$, $w_{t,x}(s_0) = w_{t,x}'(s_0) = 0$ and so we have
$$
\aligned
|w_{t,x}(s)| + |w'_{t,x}(s)| \leq \, & C F(s) + C\int_{s_0}^s F(\sbar)|h_{t,x}'(\sbar)|e^{C\int_\sbar^s |h_{t,x}'(\lambda)|d\lambda} \, d\sbar, 
\endaligned
$$
which leads to 
$
|w_{t,x}(s)| + |w'_{t,x}(s)|\lesssim V(t,x).
$
It remains to recall the relation between $v$ and $w$, that is, 
$v(t,x) = s^{3/2}w_{t,x}(s)$
and
$$
(s/t)^{-1}s^{3/2}\newperp  v(t,x) = w_{t,x}'(s) - \frac{3}{2}s^{1/2}v(t,x) = w_{t,x}'(s) - \frac{3}{2}s^{-1}w_{t,x}(s), 
$$
and the desired estimate is established.
\end{proof}


\section{Commutator estimates}

\subsection{Algebraic decomposition of the commutators}

We consider the commutators $ [X,Y]u := X(Yu)- Y(Xu)$ of operators associated with our vector fields
when the function $u$ is defined in the future cone
$\Kcal=\{|x|< t-1\}$. Our uniform bounds rely on homogeneity arguments and
on the observation that the coefficients of our decompositions are smooth in $\Kcal$.

First of all, the vector fields $\del_\alpha, $ and $L_a$ are Killing fields for the (flat) wave operator $\Box$, so that
 the following commutation relations hold:
\be
[\del_\alpha, \,\Box]=0, \qquad [L_a,\, \Box] =0.
\ee
By introducing the notation
\be
\label{pre commutator base L-P}
\aligned
\, [L_a,\del_\beta]
& =: \Theta_{a\beta}^{\gamma}\del_{\gamma},
\qquad
[\del_\alpha,\delu_\beta]
 =: t^{-1}\Gammau_{\alpha\beta}^{\gamma}\del_{\gamma},
\qquad
[L_a,\delu_\beta]
 =: \Thetau_{a\beta}^{\gamma}\delu_{\gamma},
\endaligned
\ee
we find easily that
\bel{pre commutator base'}
\aligned
&\Theta_{a0}^{\gamma} = -\delta_a^{\gamma},
\qquad
&&\Theta_{ab}^{\gamma} = -\delta_{ab}\delta_0^{\gamma},
\\
&\Gammau_{0b}^{\gamma} = -\frac{x^b}{t}\delta_0^{\gamma} = \Psi^0_b\delta_0^{\gamma},
\quad&&
 \Gammau_{\alpha0}^{\gamma}  = 0,
\qquad
&&&\Gammau_{ab}^{\gamma}= \delta_{ab}\delta_0^{\gamma},\quad
\\
&\Thetau_{a0}^{\gamma} = -\delta^{\gamma}_a + \frac{x^a}{t}\delta^{\gamma}_0 =  -\delta^{\gamma}_a + \Phi_0^a\delta^{\gamma}_0,
\qquad
&&\Thetau_{ab}^{\gamma} = -\frac{x^b}{t}\delta^{\gamma}_a = \Psi^0_b\delta^{\gamma}_a,
\endaligned
\ee
where $\Phi$ and $\Psi$ were defined at the beginning of Section \ref{sec the hyper}.
All of these coefficients are smooth in the (open) cone $\Kcal$ and homogeneous of degree $0$.
Furthermore, we can also check that
\bel{pre commutator base''}
\Thetau_{ab}^0 = 0, \quad \text{ so that } \quad
[L_a,\delu_b] =
 \Thetau_{ab}^c\delu_c, 
\ee
which means that the commutator of a ``good'' derivative $\delu_b$ with $L_a$ is again a ``good'' derivative. (That is, these derivatives enjoy better decay compared to the gradient itself.)

\begin{lemma}
[Algebraic decomposition of commutators. I]
There exist constants $\lambda_{aJ}^I$ such that
\bel{pre lem commutator pr1}
[\del^I, L_a] = \sum_{|J|\leq|I|}\lambda^I_{aJ}\del^J.
\ee
\end{lemma}

\begin{proof} We proceed by induction on $|I|$.
For $|I| = 1$, the result is guaranteed by \eqref{pre commutator base L-P}. Suppose that \eqref{pre lem commutator pr1} holds for all $|I_1|\leq m$, we will prove that it is still valid for $|I|\leq m+1$. Let $I = (\alpha,\alpha_m,\alpha_{m-1},\dots,\alpha_1)$ and  $I_1=(\alpha_m,\alpha_{m-1}, \ldots, \alpha_1)$, so that $\del^I = \del_\alpha\del^{I_1}$. Then we have 
$$
\aligned
\,[\del^I,L_a]
= \, & [\del_\alpha\del^{I_1},L_a] = \del_\alpha\big([\del^{I_1},L_a]\big)
 + [\del_\alpha,L_a]\del^{I_1}
= \del_\alpha\bigg(\sum_{|J|\leq |I_1|}\lambda_{aJ}^{I_1}\del^J \bigg) - \Theta_{a\alpha}^\gamma\del_{\gamma}\del^{I_1}
\\
= \, &\sum_{|J|\leq |I_1|}\lambda_{aJ}^{I_1}\del_\alpha\del^J - \Theta_{a\alpha}^\gamma\del_{\gamma}\del^{I_1},
\endaligned
$$
which yields the desired statement for $|I| = m+1$.
\end{proof}

\begin{lemma}
[Algebraic decomposition of commutators. II]
There exist constants $\theta_{\alpha J}^{I\gamma}$ such that
\be
\label{pre lem commutator pr2}
[L^I, \del_\alpha] = \sum_{|J|\leq|I|-1,\gamma}\theta_{\alpha J}^{I\gamma}\del_{\gamma}L^J.
\ee
\end{lemma}

\begin{proof} We proceed by induction and observe that the case $|I|=1$ is already covered by \eqref{pre commutator base L-P}. We assume that \eqref{pre lem commutator pr2} is valid for $|I|\leq m$ and we will prove that it is still valid when $|I|=m+1$. For this purpose, we take
$
L^I = L_aL^{I_1}$ with $|I_1| = m$, and we have 
$$
\aligned
\,[L^I,\del_\alpha]
 = &[L_aL^{I_1},\del_\alpha]
 = L_a\big([L^{I_1},\del_\alpha]\big) + [L_a,\del_\alpha]L^{I_1}
\\
=& L_a\bigg(\sum_{|J|\leq |I_1|-1,\gamma}\theta_{\alpha J}^{I_1\gamma}\del_{\gamma}L^J \bigg)
+ \sum_{\gamma}\Theta_{a\alpha}^{\gamma}\del_{\gamma}L^{I_1}
\\
=&\sum_{|J|\leq |I_1|-1,\gamma}\theta_{\alpha J}^{I_1\gamma}L_a\del_{\gamma}L^J
 + \sum_{\gamma}\Theta_{a\alpha}^{\gamma}\del_{\gamma}L^{I_1}
\endaligned
$$
and, therefore,
$$
\aligned
\,[L^I,\del_\alpha]
=&\sum_{|J|\leq |I_1|-1,\gamma}\theta_{\alpha J}^{I_1\gamma}\del_{\gamma}L_aJ^J
+\sum_{|J|\leq |I_1|-1,\gamma}\theta_{\alpha J}^{I_1\gamma}[L_a,\del_{\gamma}]J^J
+ \sum_{\gamma}\Theta_{a\alpha}^{\gamma}\del_{\gamma}L^{I_1}
\\
=&\sum_{|J|\leq |I_1|-1,\gamma}\theta_{\alpha J}^{I_1\gamma}\del_{\gamma}L_aJ^J
+\sum_{|J|\leq |I_1|-1,\gamma}\theta_{\alpha J}^{I_1\gamma}\Theta_{a\gamma}^{\gamma'}\del_{\gamma'}L^J
+ \sum_{\gamma}\Theta_{a\alpha}^{\gamma}\del_{\gamma}L^{I_1}.
\endaligned
$$
\end{proof}

An immediate consequence of \eqref{pre lem commutator pr2} is
\be
[\del^IL^J,\del_\alpha]u = \sum_{|J'|<|J|,\gamma}\theta_{\alpha J'}^{J\gamma}\del_{\gamma}\del^{I}L^{J'} u.
\ee

\begin{lemma}
[Algebraic decomposition of commutators. III]
\label{lem-com2}
In the future cone $\Kcal$,
 the following identity holds:
\bel{pre lem commutator pr2 NEW}
[\del^IL^J,\delu_\beta]
 = \sum_{|J'| \leq |J|,|I'|\leq|I|\atop |I'|+|J'|<|I|+|J|}\thetau_{\beta I'J'}^{IJ \gamma}\del_{\gamma}\del^{I'}L^{J'},
\ee
where the coefficients $\thetau_{\beta I'J'}^{IJ\gamma}$ are smooth functions
and satisfy (in $\Kcal$)
\bel{pre lem commutator pr4a}
 \big|\del^{I_1}L^{J_1}\thetau_{\beta I'J'}^{IJ\gamma}\big| \leq
\begin{cases}
C\big(|I|,|J|,|I_1|,|J_1|\big) \, t^{-|I_1|}
  &\text{ when } |J'| < |J|,
\\
C\big(|I|,|J|,|I_1|,|J_1|\big) \, t^{-|I_1|-1}     &\text{ when } |I'| < |I|.
\end{cases}
\ee
\end{lemma}

\begin{proof} Consider the  identity
$$
\aligned
\,[\del^IL^J,\delu_\beta]
= [\del^IL^J, \Phi_\beta^{\gamma}\del_{\gamma}]
=&
\Phi_\beta^{\gamma}[\del^IL^J,\del_{\gamma}]
+ \sum_{I_1+I_2=I, J_1+J_2=J \atop |I_1|+|J_1|<|I|+|J|}
\del^{I_1}L^{J_1}\Phi_\beta^{\gamma}\del^{I_2}L^{J_2}\del_{\gamma}.
\endaligned
$$
In the first sum, we commute $\del^{I_2}L^{J_2}$ and $\del_{\gamma}$ and obtain  
$$
\aligned
\,[\del^IL^J,\delu_\beta]
=& \Phi_\beta^{\gamma}[\del^IL^J,\del_{\gamma}]
\\
& + \sum_{I_1+I_2=I,J_1+J_2 = J \atop |I_1|+|J_1|<|I|+|J|}
\del^{I_1}L^{J_1}\Phi_\beta^{\gamma}\del_{\gamma}\del^{I_2}L^{J_2}
 + \sum_{I_1+I_2=I,J_1+J_2= J\atop |I_1|+|J_1|<|I|+|J|}
\del^{I_1}L^{J_1}\Phi_\beta^{\gamma}[\del^{I_2}L^{J_2},\del_{\gamma}]
\\
=&  \sum_{I_1+I_2=I,J_1+J_2= J\atop |I_1|+|J_1|<|I|+|J|}
\del^{I_1}L^{J_1}\Phi_\beta^{\gamma}\del_{\gamma}\del^{I_2}L^{J_2}
 + \sum_{I_1+I_2=I\atop J_1+J_2=J}
\del^{I_1}L^{J_1}\Phi_\beta^{\gamma}[\del^{I_2}L^{J_2},\del_{\gamma}]
\\
=& \sum_{I_1+I_2=I,J_1+J_2= J\atop |I_1|+|J_1|<|I|+|J|}
\del^{I_1}L^{J_1}\Phi_\beta^{\gamma}\del_{\gamma}\del^{I_2}L^{J_2}
+ \sum_{I_1+I_2=I\atop J_1+J_2=J}\sum_{|J_2'|<|J_2|}
\big(\del^{I_1} L^{J_1}\Phi_\beta^{\gamma} \big) \, \theta_{\gamma J_2'}^{J_2\delta}\del_{\delta} \del^{I_2}L^{J_2'}.
\endaligned
$$
Hence, $\thetau_{\gamma I'J'}^{IJ\alpha}$ are linear combinations of $\del^{I_1}L^{J_1} \Phi_\beta^{\gamma}$ and
$\big( \del^{I_1}L^{J_1}\Phi_\beta^{\gamma} \big) \theta_{\gamma J_2'}^{J_2\delta}$ and $J_1+J_2=J$,
which yields \eqref{pre lem commutator pr2 NEW}.
Note that $\theta_{\gamma J_2'}^{J_2\delta}$ are constants, so that
$$
\del^{I_3}L^{J_3}\big(\del^{I_1}L^{J_1}\Phi_{\beta}^{\gamma}\theta_{\gamma J_2'}^{J_2\delta}\big)
= \theta_{\gamma J_2'}^{J_2\delta}\del^{I_3}L^{J_3}\del^{I_1}L^{J_1}\Phi_{\beta}^{\gamma}.
$$ 
By definition, $\Phi_{\beta}^{\gamma}$ is a homogeneous function of degree zero, so that $\del^{I_1}L^{J_1}\Phi_{\beta}^{\gamma}$ is again homogeneous but with degree $\leq 0$. We thus arrive at  \eqref{pre lem commutator pr4a}.
\end{proof}

\begin{lemma}[Algebraic decomposition of commutators. IV]
Within the future cone $\Kcal$, the following identity holds
\bel{pre lem commutator pr3}
[L^I,\delu_c]
 = \sum_{|J|<|I|}\sigma^{Ia}_{cJ}\delu_aL^J,
\ee
where the coefficients $\sigma_{c J}^{Ia}$ are
smooth functions and
satisfy (in $\Kcal$)
\bel{pre lem commutator pr3b}
\big|\del^{I_1}L^{J_1}\sigma_{c J}^{Ia}\big| \leq C(|I|,|J|,|I_1|,|J_1|)t^{-|I_1|}.
\ee
\end{lemma}

\begin{proof}
This is also by induction. Again, when $|I|=1$, \eqref{pre lem commutator pr3} together with \eqref{pre lem commutator pr3b} are guaranteed by \eqref{pre commutator base''}. Assume that \eqref{pre lem commutator pr3} and \eqref{pre lem commutator pr3b} hold for $|I|\leq m$, we will prove that they are valid for $|I| = m+1$. We take $L^I = L_aL^J$ with $|J| = m$, and obtain
$$
\aligned
\,[L^I,\delu_c]  =&[L_aL^J,\delu_c]  = L_a\big([L^J,\delu_c] \big) + [L_a,\delu_c]L^J
\\
=&L_a\bigg(\sum_{|J'|<|J|}\sigma^{Ja}_{cJ'}\delu_aL^{J'}  \bigg) + \Thetau_{ac}^b\delu_b L^J
\\
=&\sum_{|J'|<|J|}L_a\sigma^{Jb}_{cJ'}\delu_bL^{J'}
 + \sum_{|J'|<|J|}\sigma^{Jb}_{cJ'}L_a\delu_bL^{J'}  + \Thetau_{ac}^b\delu_b L^J, 
\endaligned
$$
so that 
$$
\aligned
\,[L^I,\delu_c]  
=&\sum_{|J'|<|J|}L_a\sigma^{Jb}_{cJ'}\delu_bL^{J'}  + \sum_{|J'|<|J|}\sigma^{Jb}_{cJ'}\delu_b L_a L^{J'}
+ \sum_{|J'|<|J|}\sigma^{Jb}_{cJ'}[L_a,\delu_b]L^{J'}  + \Thetau_{ac}^b\delu_b L^J
\\
=&\sum_{|J'|<|J|}L_a\sigma^{Jb}_{cJ'}\delu_bL^{J'}  + \sum_{|J'|<|J|}\sigma^{Jb}_{cJ'}\delu_b L_a L^{J'}
+ \sum_{|J'|<|J|}\sigma^{Jb}_{cJ'}\Thetau_{ab}^d\delu_dL^{J'}  + \Thetau_{ac}^b\delu_b L^J.
\endaligned
$$
In each term the coefficients are homogeneous of degree $0$ (by applying \eqref{pre lem commutator pr3b}),
and the desired result is proven.
\end{proof}

The following result is also checked by induction along the same lines as above, and so its proof is omitted. 

\begin{lemma}[Algebraic decomposition of commutators. V]
Within the future cone $\Kcal$, the following identity holds:
\bel{pre lem commutator pr4}
[\del^I,\delu_c]
=  t^{-1}\!\!\!\!\sum_{|J|\leq|I|}\rho_{cJ}^{I}\del^{J},
\ee
where the coefficients $\rho_{cJ}^{I}$ are smooth functions
and satisfy (in $\Kcal$)
\bel{pre lem commutator pr4b}
\big|\del^{I_1}L^{J_1}\rho_{cJ}^{I}\big| \leq C(|I|,|J|,|I_1|,|J_1|)t^{-|I_1|}.
\ee
\end{lemma}


\subsection{Estimates for the commutators}
\label{sec:33}

The following statements are now immediate in view of \eqref{pre lem commutator pr1}, \eqref{pre lem commutator pr2}, and \eqref{pre lem commutator pr3}, and \eqref{pre lem commutator pr4}.

\begin{proposition}[Estimates on commutators. I]
\label{lem commutator esti I}
For all sufficiently regular functions $u$ defined in the future cone $\Kcal$, the following estimates hold:
\bel{pre lem commutator pr5}
\big|[\del^IL^J,\del_\alpha]u\big|\leq C(|I|, |J|)\sum_{|J'|<|J|,\beta}|\del_\beta\del^IL^{J'}u|,
\ee
\bel{pre lem commutator pr5 NEW}
\big|[\del^IL^J,\delu_c]u\big|
\leq
C(|I|,|J|)
\Bigg(
\sum_{|J'|<|J|,a\atop |I'|\leq |I|} |\delu_a \del^{I'}L^{J'}u|
+ t^{-1}\!\!\!\!\sum_{|I|\leq|I'|\atop |J|\leq|J'|}|\del^{I'}L^{J'}u| \Bigg),
\ee
\bel{pre lem commutator trivial} 
\left|[\del^IL^J,\delu_{\alpha}u]\right|
\leq C(|I|,|J|)t^{-1}\sum_{\beta,|I'|<|I|\atop |J'|\leq|J|}\left|\del_{\beta}\del^{I'}L^{J'}u\right|
 +C(|I|,|J|)\sum_{\beta,|I'|\leq|I|\atop |J'|<|J|}\left|\del_{\beta}\del^{I'}L^{J'}u\right|, 
\ee
\bel{pre lem commutator second-order}
\big|[\del^IL^J,\del_\alpha \del_\beta] u \big|
\leq  C(|I|,|J|)\sum_{\gamma,\gamma'\atop |I|\leq|I'|,|J'|<|I|} \big|\del_{\gamma}\del_{\gamma'}\del^{I'}L^{J'} u\big|,
\ee
\bel{pre lem commutator second-order bar}
\aligned
&\big|[\del^IL^J, \delu_a\delu_\beta] u\big| + \big|[\del^IL^J, \delu_\alpha \delu_b] u\big|
\\
&\leq C(|I|,|J|) \Bigg(
\sum_{c,\gamma,|I'|\leq |I|\atop |J'| < |J|}\big|\delu_c \delu_{\gamma} \del^{I'}L^{J'}u\big|
+
 t^{-1} \sum_{c,\gamma,|I'| < |I|\atop |J'| \leq |J|}\big|\delu_c \delu_{\gamma} \del^{I'}L^{J'}u\big|
+ t^{-1}\sum_{\gamma,|I'|\leq|I|\atop |J'|\leq|J|}\big|\del_{\gamma}\del^{I'}L^{J'}u\big|
\Bigg).
\endaligned
\ee 
\end{proposition}

Further estimates will be also needed, as now stated.

\begin{proposition}[[Estimates on commutators. II]
\label{pre lem commutator s/t}
For all sufficiently regular functions $u$ defined in the future cone $\Kcal$, the following estimate holds (for all $I, J, \alpha$)
\bel{pre lem commutator T/t'}
\big|\del^IL^J\big((s/t)\del_\alpha u\big)\big| \leq \big|(s/t)\del_\alpha \del^IL^J u\big|
+ C(|I|,|J|)\sum_{\beta,|I'|\leq|I|\atop |J'|\leq |J|}\big|(s/t)\del_\beta\del^{I'}L^{J'}u\big|.
\ee
\end{proposition}

Recall that the proof of the above result (given in \cite{PLF-MY-book}) relies on the following technical observation, concerning  products of first-order linear operators with homogeneous coefficients of order $0$ or $1$.

\begin{lemma}
\label{pre lem lem commutator s/t}
For all multi-indices $I$, the function
$$
\Xi^{I,J}  := (t/s) \del^I L^J (s/t),
$$
defined in the closed cone $\overline{\Kcal} = \{|x|\leq t-1\}$, is smooth and all of its derivatives (of any order)
are bounded in $\overline{\Kcal}$. Furthermore, it is homogeneous of degree $\eta$ with $\eta\leq 0$.
\end{lemma}


\section{Initialization of the bootstrap argument}
\label{sec:5}

From this section onwards, we begin the proof of Theorem \ref{thm main}, which is a rather involved bootstrap argument along the lines of the method presented in \cite{PLF-MY-book}. We fix some integer $N \geq 8$ throughout, and first 
summarize our strategy, as follows.

Let $(u,v)$ be the local-in-time solution to the Cauchy problem associated with the system \eqref{eq main}. From a standard local existence result (cf., for instance, \cite[Section 11]{PLF-MY-book}), we can construct a local-time solution from the data given on the initial hypersurface and, consequently, guarantee that on the {\sl initial hyperboloid} and for all $|I|+|J|\leq N$,
$$
E_m(2,\del^IL^J u)^{1/2} \leq C_0\vep,
\qquad E_m(2,\del^IL^J v)^{1/2} \leq C_0\vep
$$
for some uniform constant $C_0>0$. 
On some (hyperbolic) time interval $[2,s_1]$, we can thus assume the following energy conditions for some constants $C_1, \vep, \delta>0$ (yet to be determined):
\bel{ineq energy assumption}
\aligned
&E_m(s,\del^IL^J u)^{1/2}\leq C_1\vep s^{k\delta}, &&|J|=k, \quad &&&|I|+|J|\leq N, \qquad
&&&& \text{wave / high-order,}
\\
&E_m(s,\del^IL^J u)^{1/2}\leq C_1\vep, &&\quad &&&|I|+|J|\leq N-4, &&&& \text{wave / low-order,}
\\
&E_m(s,\del^IL^J v)^{1/2}\leq C_1\vep s^{1/2+k\delta}, &&|J|=k, \quad &&&|I|+|J|\leq N, &&&& \text{Klein-Gordon / high-order,}
\\
&E_m(s,\del^IL^J v)^{1/2}\leq C_1\vep s^{k\delta}, &&|J|=k, \quad &&&|I|+|J|\leq N-4 &&&& \text{Klein-Gordon / low-order.}
\endaligned
\ee
We will prove that on the same interval the following {\sl improved energy bounds} are valid when $\vep$ is sufficiently 
small and $C_1>C_0$ with $\frac{1}{10N}\leq\delta\leq \frac{1}{5N}$ (fixed once for all):
\bel{ineq energy assumption'}
\aligned
&E_m(s,\del^IL^J u)^{1/2}\leq \frac{1}{2}C_1\vep s^{k\delta}, &&|J|=k, \quad
&&&|I|+|J|\leq N,
\qquad
&&&& \text{wave / high-order,}
\\
&E_m(s,\del^IL^J u)^{1/2}\leq \frac{1}{2}C_1\vep, \quad &&
&&&|I|+|J|\leq N-4,
&&&& \text{wave / low-order,}
\\
&E_m(s,\del^IL^J v)^{1/2}\leq \frac{1}{2}C_1\vep s^{1/2+k\delta}, &&|J|=k, \quad
&&&|I|+|J|\leq N,
&&&& \text{Klein-Gordon / high-order,}
\\
&E_m(s,\del^IL^J v)^{1/2}\leq \frac{1}{2}C_1\vep s^{k\delta}, &&|J|=k, \quad
&&&|I|+|J|\leq N-4
&&&& \text{Klein-Gordon / low-order.}
\endaligned
\ee

Once this property is proven, we set
$$
s_1 := \sup \Big\{s \, / \, \eqref{ineq energy assumption} \text{ holds on }[2,s] \Big\}
$$
and we can deduce that $s_1 = +\infty$. Indeed, by a continuity argument, $C_1>C_0$ implies $s_1>2$. 
Again by a continuity argument, we deduce that when $s = s_1$, at least one of the inequalities \eqref{ineq energy assumption} must be an equality. But, when \eqref{ineq energy assumption'} holds, none of them can become an equality. This means that $s_1 = +\infty$ and the rest of our work consists of proving \eqref{ineq energy assumption'}.

\begin{proposition}[Formulation of the bootstrap argument]
\label{prop bootstrap}
Given any integer $N \geq 8$ and  $\frac{1}{10N}<\delta<\frac{1}{5N}$, there exist constants $C_1,\vep>0$ satisfying $\vep C_1 < 1$ such that any local-in-time solution $(u,v)$ to \eqref{eq main}, defined in the time interval $[2,s_1]$ 
and satisfying the energy conditions \eqref{ineq energy assumption} for some $\vep \in (0, \vep_0]$, also satisfies 
the improved energy bounds \eqref{ineq energy assumption'}.
\end{proposition}

The remaining text is devoted to the proof of this proposition, which we decompose into three parts. First, we derive a series of $L^2$ and sup-norm estimates directly from \eqref{ineq energy assumption}, and from the Sobolev inequality on hyperboloids \eqref{eq 1 sobolev} and the commutator estimates (i.e.~Propositions~\ref{lem commutator esti I} and \ref{pre lem commutator s/t}). Second, we improve the sup-norm estimates by using \eqref{Linfini wave ineq} and  \eqref{Linfty KG ineq a}. Finally, we combine the improved sup-norm estimates and the $L^2$ estimates established in the first part and we get the improved energy estimates \eqref{ineq energy assumption'}.

From now on, we consider a {\bf solution satisfying the energy bound \eqref{ineq energy assumption}} and 
throughout we set $|J| =k$.  


\section{Basic estimates}
\label{newsection-6}

\subsection{Basic $L^2$ estimates of the first generation}

Throughout, we always assume that $2\leq s\leq s_1$. We begin by stating the $L^2$ type estimates provided to us by the energy assumption \eqref{ineq energy assumption}.
For all $|I|+|J|\leq N$ with $|J|=k$, we have the high-order bounds
\be
\label{L2 basic 1ge1 a}
\aligned
\|(s/t)\del_\alpha\del^I L^Ju\|_{L_f^2(\Hcal_s)} + \|(s/t)\delu_\alpha\del^I L^Ju\|_{L_f^2(\Hcal_s)}
&
\lesssim C_1\vep s^{k\delta},
\\
\|\delu_a\del^I L^Ju\|_{L_f^2(\Hcal_s)} + \|\newperp \del^I L^Ju\|_{L_f^2(\Hcal_s)}
& \lesssim C_1\vep s^{k\delta},
\\
\|(s/t)\del_\alpha\del^I L^Jv\|_{L_f^2(\Hcal_s)} + \|(s/t)\delu_\alpha\del^I L^Jv\|_{L_f^2(\Hcal_s)}
& \lesssim C_1\vep s^{1/2 + k\delta},
\\
 \|\delu_a\del^I L^Jv\|_{L_f^2(\Hcal_s)} + \|\newperp \del^I L^Jv\|_{L_f^2(\Hcal_s)}
& \lesssim C_1\vep s^{1/2 + k\delta},
\\
\|\del^IL^J v\|_{L_f^2(\Hcal_s)}
& \lesssim C_1\vep s^{1/2 + k\delta}, 
\endaligned
\ee
where the last estimate implies, for all $|I|+|J|\leq N-1$ and $|J| = k$, the following estimate
\bel{L2 basic 2ge1}
\|\del_\alpha\del^IL^J v\|_{L_f^2(\Hcal_s)} \lesssim C_1\vep s^{1/2 + k\delta},
\ee
as well as,
for $|I|+|J|\leq N-4$ with $|J|=k$, the low-order energy bounds imply:
\be
\label{L2 basic 3ge1}
\aligned
\|(s/t)\del_\alpha\del^I L^Ju\|_{L_f^2(\Hcal_s)} + \|(s/t)\delu_\alpha\del^I L^Ju\|_{L_f^2(\Hcal_s)}
& \lesssim C_1\vep,
\\
\|\delu_a \del^I L^Ju\|_{L_f^2(\Hcal_s)} + \|\newperp \del^I L^Ju\|_{L_f^2(\Hcal_s)}
& \lesssim C_1\vep,
\\
\|(s/t)\del_\alpha\del^I L^Jv\|_{L_f^2(\Hcal_s)} + \|(s/t)\delu_\alpha\del^I L^Jv\|_{L_f^2(\Hcal_s)}
& \lesssim C_1\vep s^{k\delta},
\\
\|\newperp \del^I L^Jv\|_{L_f^2(\Hcal_s)} + \|\delu_a\del^I L^Jv\|_{L_f^2(\Hcal_s)}
& \lesssim C_1\vep s^{ k\delta},
\\
\|\del^IL^J v\|_{L_f^2(\Hcal_s)}
& \lesssim C_1\vep s^{ k\delta}.
\endaligned
\ee
In addition, they also imply, for all $|I|+|J|\leq N-5$ with $|J| = k$,
\bel{L2 basic 4ge1}
\|\del_\alpha\del^IL^J v\|_{L_f^2(\Hcal_s)} \lesssim C_1s^{k\delta}.
\ee


\subsection{Basic $L^2$ estimates of the second generation}

The following estimates are obtained by applying the above energy estimate combined with the commutator estimates presented in Proposition \ref{lem commutator esti I}. For all $|I|+|J|\leq N$ with $|J|=k$, we have the high-order bounds
\be
\label{L2 basic 1ge2}
\aligned
\|(s/t)\del^I L^J\del_\alpha u\|_{L_f^2(\Hcal_s)} + \|(s/t)\del^I L^J\delu_\alpha u\|_{L_f^2(\Hcal_s)}
& \lesssim C_1\vep s^{k\delta},
\\
\|\del^I L^J\delu_au\|_{L_f^2(\Hcal_s)} + \|\del^I L^J\newperp  u\|_{L_f^2(\Hcal_s)}
& \lesssim C_1\vep s^{k\delta},
\\
\|(s/t)\del^I L^J\del_\alpha v\|_{L_f^2(\Hcal_s)} + \|(s/t)\del^I L^J\delu_\alpha v\|_{L_f^2(\Hcal_s)}
&\lesssim C_1\vep s^{1/2 + k\delta},
\\
\|\del^I L^J\newperp v\|_{L_f^2(\Hcal_s)} + \|\del^I L^J\delu_av\|_{L_f^2(\Hcal_s)}
& \lesssim C_1\vep s^{1/2 + k\delta},
\\
\|\del^IL^J v\|_{L_f^2(\Hcal_s)}& \lesssim C_1\vep s^{1/2 + k\delta},
\endaligned
\ee
which, for $|I| + |J|\leq N-1$ with $|J| = k$, imply the low-order bounds (for instance, by expressing $t\delu_a = L_a$ in the first inequality): 
\be
\label{L2 basic 2ge2}
\aligned
\|t\delu_a\del^IL^J v\|_{L_f^2(\Hcal_s)}
& \lesssim C_1\vep s^{1/2 +(k+1)\delta},
\\
\|t\del^IL^J \delu_av\|_{L_f^2(\Hcal_s)}
& \lesssim C_1\vep s^{1/2 + (k+1)\delta},
\\
\|\del^IL^J\del_\alpha v\|_{L_f^2(\Hcal_s)} + \|\del^IL^J \delu_\alpha v\|_{L_f^2(\Hcal_s)}
& \lesssim C_1\vep s^{1/2 + k\delta},
\\
\|(s/t)\del^IL^J\del_\alpha\del_\beta v\|_{L_f^2(\Hcal_s)} +  \|(s/t)\del^IL^J\delu_\alpha\delu_\beta v\|_{L_f^2(\Hcal_s)}
&\lesssim C_1\vep s^{1/2 + k\delta},
\\
\|s\del^IL^J\delu_\alpha\delu_bv\|_{L_f^2(\Hcal_s)} + \|s\del^IL^J\delu_a \delu_\alpha v\|_{L_f^2(\Hcal_s)}
&\lesssim C_1\vep s^{1/2 + (k+1)\delta}.
\endaligned
\ee
Observe especially that, for derivation of the last term, we write 
$$
\del^IL^J\delu_a\delu_{\alpha}v = \del^IL^J\big(t^{-1}L_a\delu_{\alpha}v\big) = \sum_{I_1+I_2=I\atop J_1+J_2=J}\del^{I_1}L^{J_1}L_a\delu_{\alpha}v \del^{I_2}L^{J_2}(t^{-1})
$$
and then we proceed by homogeneity (by noting that $t^{-1}$ is homogeneous of degree $-1$ and its derivatives are homogeneous of degree $\leq -1$):
$$
\aligned
\|s\del^IL^J\delu_a\delu_{\alpha}v\|_{L_f^2(\Hcal_s)}
&\lesssim\sum_{|I_1|\leq |I|\atop |J_1|\leq|J|} \|st^{-1}\del^{I_1}L^{J_1}L_a\delu_{\alpha}v\|_{L_f^2(\Hcal_s)}
\lesssim \sum_{|I_1|\leq |I|\atop |J_1|\leq|J|} \|(s/t)\del^{I_1}L^{J_1}L_a\delu_{\alpha}v\|_{L_f^2(\Hcal_s)}
\\
&\lesssim C_1\vep s^{1/2+(|J_1|+1)\delta}\lesssim C_1\vep s^{1/2+(k+1)\delta}.
\endaligned
$$

We also have, for $|I|+|J|\leq N-2$,
\bel{L2 basic 2.5ge2}
\|\del^IL^J\del_\alpha\del_\beta v\|_{L_f^2(\Hcal_s)}\lesssim C_1\vep s^{1/2 + k\delta}.
\ee
For $|I|+|J|\leq N-4$ with $|J|=k$, we have
\be
\label{L2 basic 3ge2}
\aligned
\|(s/t)\del^I L^J\del_\alpha u\|_{L_f^2(\Hcal_s)} + \|(s/t)\del^I L^J\delu_\alpha u\|_{L_f^2(\Hcal_s)}
& \lesssim C_1\vep,
\\
\|\del^I L^J\delu_au\|_{L_f^2(\Hcal_s)} + \|\del^I L^J\newperp  u\|_{L_f^2(\Hcal_s)}
& \lesssim C_1\vep,
\\
\|(s/t)\del^I L^J\del_\alpha v\|_{L_f^2(\Hcal_s)} + \|(s/t)\del^I L^J\delu_\alpha v\|_{L_f^2(\Hcal_s)}
& \lesssim C_1\vep s^{k\delta},
\\
\|\del^I L^J\newperp v\|_{L_f^2(\Hcal_s)} +\|\del^I L^J\delu_av\|_{L_f^2(\Hcal_s)}
& \lesssim C_1\vep s^{ k\delta},
\\
\|\del^IL^J v\|_{L_f^2(\Hcal_s)}
& \lesssim C_1\vep s^{ k\delta}.
\endaligned
\ee
For $|I| + |J|\leq N-5$, $|J| = k$, we have
\be
\label{L2 basic 4ge2}
\aligned
\|t\delu_a\del^IL^J v\|_{L_f^2(\Hcal_s)} + \|t\del^IL^J \delu_av\|_{L_f^2(\Hcal_s)}
& \lesssim C_1\vep s^{(k+1)\delta},
\\
\|\del^IL^J\del_\alpha v\|_{L_f^2(\Hcal_s)} + \|\del^IL^J \delu_\alpha v\|_{L_f^2(\Hcal_s)}
& \lesssim C_1\vep s^{k\delta},
\\
\|(s/t)\del^IL^J\del_\alpha\del_\beta v\|_{L_f^2(\Hcal_s)} +  \|(s/t)\del^IL^J\delu_\alpha\delu_\beta v\|_{L_f^2(\Hcal_s)}
&\lesssim C_1\vep s^{ k\delta},
\\
\|s\del^IL^J\delu_\alpha\delu_b v\|_{L_f^2(\Hcal_s)} + \|s\del^IL^J\delu_a \delu_\alpha v\|_{L_f^2(\Hcal_s)}
&\lesssim C_1\vep s^{(k+1)\delta}.
\endaligned
\ee
For $|I|+|J|\leq N-6$, we have
\be
\aligned
\|\del^IL^J\del_\alpha\del_\beta v\|_{L_f^2(\Hcal_s)} + \|\del^IL^J\delu_\alpha\delu_\beta v\|_{L_f^2(\Hcal_s)}
& \lesssim C_1\vep s^{k\delta},
\\
\|t\del^IL^J\delu_a\delu_\beta v\|_{L_f^2(\Hcal_s)} + \|t\del^IL^J\delu_\beta\delu_av\|_{L_f^2(\Hcal_s)}
& \lesssim C_1\vep s^{(k+1)\delta}. 
\endaligned
\ee


\subsection{Basic sup-norm estimates of the first generation}
\label{subsec basic Linfini 1}

We combine the Sobolev inequality on hyperboloids \eqref{eq 1 sobolev} with our $L^2$ estimates. In view of the high-order $L^2$ bounds, for $|I| + |J|\leq N-2$ with $|J|=k$ we have 
\bel{decay basic 1ge1R} 
\aligned
\sup_{\Hcal_s}\big(t^{1/2}s|\del_\alpha\del^IL^J u|\big) + \sup_{\Hcal_s}\big(t^{1/2}s|\delu_\alpha\del^IL^J u|\big)
& \lesssim C_1\vep s^{(k+2)\delta},
\\ 
\sup_{\Hcal_s}\big(t^{3/2}|\delu_a\del^IL^J u|\big) + \sup_{\Hcal_s}\big(t^{3/2}|\newperp \del^IL^J u|\big) 
&\lesssim C_1\vep s^{(k+2)\delta},
\\ 
\sup_{\Hcal_s}\big(t^{1/2}s|\del_\alpha\del^IL^J v|\big) + \sup_{\Hcal_s}\big(t^{1/2}s|\delu_\alpha\del^IL^J v|\big)
& \lesssim C_1\vep s^{1/2+(k+2)\delta},
\\ 
\sup_{\Hcal_s}\big(t^{3/2}|\newperp \del^IL^J v|\big) + \sup_{\Hcal_s}\big(t^{3/2}|\delu_a\del^IL^J v|\big)
&\lesssim C_1\vep s^{1/2+(k+2)\delta},
\\ 
\sup_{\Hcal_s}\big(t^{3/2}|\del^IL^J v|\big)
&\lesssim C_1\vep s^{1/2+(k+2)\delta}.
\endaligned
\ee 
For $|I|+|J|\leq N-3$ with $|J|=k$, we have 
\bel{decay basic 2ge1}
\sup_{\Hcal_s}\big(t^{3/2}|\del_\alpha\del^IL^J v|\big)\lesssim C_1\vep s^{1/2+(k+2)\delta}.
\ee

From the low-order $L^2$ bounds,
for $|I|+|J|\leq N-6$ with $|J|=k$ we have  
\bel{decay basic 3ge1RR}
\aligned
\sup_{\Hcal_s}\big(t^{1/2}s|\del_\alpha\del^IL^J u|\big) + \sup_{\Hcal_s}\big(t^{1/2}s|\delu_\alpha\del^IL^J u|\big)
&\lesssim C_1\vep,
\\ 
\sup_{\Hcal_s}\big(t^{3/2}|\delu_a\del^IL^J u|\big) + \sup_{\Hcal_s}\big(t^{3/2}|\newperp \del^IL^J u|\big)
&\lesssim C_1\vep,
\\ 
\sup_{\Hcal_s}\big(t^{1/2}s|\del_\alpha\del^IL^J v|\big) + \sup_{\Hcal_s}\big(t^{1/2}s|\delu_\alpha\del^IL^J v|\big)
& \lesssim C_1\vep s^{(k+2)\delta},
\\ 
\sup_{\Hcal_s}\big(t^{3/2}|\newperp \del^IL^J v|\big) + \sup_{\Hcal_s}\big(t^{3/2}|\delu_a\del^IL^J v|\big)
&\lesssim C_1\vep s^{(k+2)\delta},
\\ 
\sup_{\Hcal_s}\big(t^{3/2}|\del^IL^J v|\big)
& \lesssim C_1\vep s^{(k+2)\delta}.
\endaligned
\ee 
For $|I|+|J|\leq N-7$ with $|J|=k$, we have 
\bel{decay basic 4ge1}
\sup_{\Hcal_s}\big(t^{3/2}|\del_\alpha\del^IL^J v|\big)\lesssim C_1 \vep s^{(k+2)\delta}.
\ee


\subsection{Basic sup-norm estimates of the second generation}
\label{subsec basic Linfini 2}

For $|I| + |J|\leq N-2$ with $|J|=k$, we have the high-order bounds 
\be
\label{decay basic 1ge2}
\aligned 
\sup_{\Hcal_s}\big(t^{1/2}s|\del^IL^J \del_\alpha u|\big) + \sup_{\Hcal_s}\big(t^{1/2}s|\del^IL^J \delu_\alpha u|\big)
& \lesssim C_1\vep s^{(k+2)\delta},
\\ 
\sup_{\Hcal_s}\big(t^{3/2}|\del^IL^J \delu_au|\big) + \sup_{\Hcal_s}\big(t^{3/2}|\del^IL^J \newperp  u|\big)
& \lesssim C_1\vep s^{(k+2)\delta},
\\ 
\sup_{\Hcal_s}\big(t^{1/2}s|\del^IL^J \del_\alpha v|\big) + \sup_{\Hcal_s}\big(t^{1/2}s|\del^IL^J \delu_\alpha v|\big)
& \lesssim C_1\vep s^{1/2+(k+2)\delta},
\\ 
\sup_{\Hcal_s}\big(t^{3/2}|\del^IL^J \delu_av|\big)
& \lesssim C_1\vep s^{1/2+(k+2)\delta},
\\ 
\sup_{\Hcal_s}\big(t^{3/2}|\del^IL^J v|\big)
& \lesssim C_1\vep s^{1/2+(k+2)\delta}
\endaligned
\ee 
and,  
for $|I|+|J|\leq N-3$ with $|J|=k$,
\be
\label{decay basic 2ge2}
\aligned 
\sup_{\Hcal_s}\big(t^{5/2}|\delu_a\del^IL^J v|\big)
& \lesssim C_1\vep s^{1/2+(k+3)\delta},
\\
\sup_{\Hcal_s}\big(t^{5/2}|\del^IL^J\delu_a v|\big)
& \lesssim C_1\vep s^{1/2+(k+3)\delta},
\\
\sup_{\Hcal_s}\big(t^{3/2}|\del^IL^J \del_\alpha v|\big) + \sup_{\Hcal_s}\big(t^{3/2}|\del^IL^J\delu_\alpha v|\big)
& \lesssim C_1\vep s^{1/2+(k+2)\delta},
\\
\sup_{\Hcal_s}\big(t^{1/2}s |\del^IL^J \del_\alpha\del_\beta v|\big) + \sup_{\Hcal_s}\big(t^{1/2}s |\del^IL^J \delu_\alpha\delu_\beta v|\big)
& \lesssim C_1 s^{1/2+(k+2)\delta},
\\
\sup_{\Hcal_s}\big(t^{3/2}s |\del^IL^J \delu_\alpha\delu_bv|\big) + \sup_{\Hcal_s}\big(t^{3/2}s |\del^IL^J \delu_a \delu_\beta v|\big)
& \lesssim C_1 s^{1/2+(k+3)\delta}.
\endaligned
\ee
For $|I|+|J|\leq N-6$ with $|J|=k$, we have
\be
\label{decay basic 3ge2}
\aligned
\sup_{\Hcal_s}\big(t^{1/2}s|\del^IL^J \del_\alpha u|\big) + \sup_{\Hcal_s}\big(t^{1/2}s|\del^IL^J \delu_\alpha u|\big)
& \lesssim C_1\vep,
\\
\sup_{\Hcal_s}\big(t^{3/2}|\del^IL^J \delu_au|\big) + \sup_{\Hcal_s}\big(t^{3/2}|\del^IL^J \newperp  u|\big)
& \lesssim C_1\vep,
\\
\sup_{\Hcal_s}\big(t^{1/2}s|\del^IL^J \del_\alpha v|\big) + \sup_{\Hcal_s}\big(t^{1/2}s|\del^IL^J \delu_\alpha v|\big)
& \lesssim C_1\vep s^{(k+2)\delta},
\\
\sup_{\Hcal_s}\big(t^{3/2}|\del^IL^J v|\big)
& \lesssim C_1\vep s^{(k+2)\delta}.
\endaligned
\ee
In addition, for $|I|+|J|\leq N-7$ with $|J|=k$, we have
\be
\label{decay basic 4ge2}
\aligned
\sup_{\Hcal_s}\big(t^{5/2}|\delu_a\del^IL^J v|\big)
& \lesssim C_1 \vep s^{(k+3)\delta},
\\
\sup_{\Hcal_s}\big(t^{5/2}|\del^IL^J \delu_av|\big)
& \lesssim C_1 \vep s^{(k+3)\delta},
\\
\sup_{\Hcal_s}\big(t^{3/2}|\del^IL^J \del_\alpha v|\big) + \sup_{\Hcal_s}\big(t^{3/2}|\del^IL^J \delu_\alpha v|\big)
& \lesssim C_1 \vep s^{(k+2)\delta},
\\
\sup_{\Hcal_s}\big(t^{1/2}s |\del^IL^J \del_\alpha\del_\beta v|\big) + \sup_{\Hcal_s}\big(t^{1/2}s |\del^IL^J \delu_\alpha\delu_\beta v|\big)
&\lesssim C_1 s^{(k+2)\delta},
\\
\sup_{\Hcal_s}\big(t^{3/2}s |\del^IL^J \delu_\alpha\delu_bv|\big) + \sup_{\Hcal_s}\big(t^{3/2}s |\del^IL^J \delu_a \delu_\beta v|\big)
& \lesssim C_1 s^{(k+3)\delta}.
\endaligned
\ee 

Moreover, $|I|+|J|\leq N-8$ with $|J|=k$, we have 
\be
\label{new-formules}
\aligned
\sup_{\Hcal_s}\big(t^{3/2} |\del^IL^J \del_\alpha\del_\beta v|\big) 
+ \sup_{\Hcal_s}\big(t^{3/2} |\del^IL^J \delu_\alpha\delu_\beta v|\big)
&\lesssim C_1 s^{(k+2)\delta},
\\ 
\sup_{\Hcal_s}\big(t^{3/2} | \del_\alpha\del_\beta \del^IL^J  v|\big) 
+ \sup_{\Hcal_s}\big(t^{3/2} |\delu_\alpha\delu_\beta \del^IL^J  v|\big)
&\lesssim C_1 s^{(k+2)\delta}. 
\endaligned
\ee 


\subsection{Estimates based on Hardy's inequality on hyperboloids}

We now substitute the basic $L^2$ estimates in Hardy's inequality \eqref{eq 1 Hardy ineq} and find  
\begin{subequations}\label{L2 basic 1ge3}
\bel{L2 basic 1ge3 a}
\|s^{-1}L^J u\|_{L_f^2(\Hcal_s)}\lesssim C_0\vep  + C_1\vep s^{k\delta}, \qquad |L|\leq N,
\ee
as well as the inequality (which will not be used in the following) 
\bel{L2 basic 1ge3 b}
\|s^{-1}L^J u\|_{L_f^2(\Hcal_s)}\lesssim (C_0+C_1)\vep + C_1\vep\ln s, \qquad |L|\leq N-4.
\ee
\end{subequations}


\subsection{Estimate based on integration along radial rays}

By the first estimate in \eqref{decay basic 3ge1RR} and since $t^{-1/2} s^{-1} \lesssim t^{-1} (t-r)^{-1/2}$ (in the domain of interest), we obtain 
$$
|\del_r\del^IL^Ju(t,x)|\lesssim C_1\vep t^{-1}(t-r)^{-1/2},\qquad |I|+|J|\leq N-6.
$$
Then we integrate this inequality in space along the rays $(t,\lambda x)|_{0\leq \lambda\leq t-1}$ for any 
$x\in \mathbb{S}^3$:
\bel{decay basic ray-integral}
|\del^IL^Ju(t,x)|\lesssim C_1\vep t^{-1}(t-r)^{1/2} \simeq C_1\vep t^{-3/2}s,
\qquad |I|+|J|\leq N-6.
\ee


\section{Refined sup-norm estimates}

\subsection{Overview of the analysis in this section}

We now proceed by using the structure of the nonlinear wave system under consideration, and relying now on sharp sup-norm estimates.
In the following sections we are going to establish the following estimates:
For $|I|\leq N-4$, we have 
\begin{subequations}\label{decay refined 0de}
\bel{decay refined 0de w}
\sup_{\Kcal_{[s_0,s_1]}} t|u| \lesssim C_1\vep, 
\ee
\bel{decay refine 0de KG}
\sup_{\Hcal_s}\big( (s/t)^{-3/2+4\delta}t^{3/2}|\newperp \del^Iv| \big) 
 + \sup_{\Hcal_s} \big( (s/t)^{-1/2+4\delta}t^{3/2}|\del^Iv|\big) 
\lesssim C_1\vep,
\ee
\end{subequations}
and, more generally, for $|I|+|J|\leq N-4$ with $|J|=k$,
\begin{subequations}\label{decay refined kde}
\bel{decay refined kde w}
\sup_{\Hcal_s} t|L^J u| \lesssim C_1\vep s^{k\delta},
\ee
\bel{decay refined kde KG}
\sup_{\Hcal_s} \big(  (s/t)^{-3/2+4\delta}t^{3/2}|\newperp \del^IL^J v| \big) 
 + \sup_{\Hcal_s}  \big( (s/t)^{-1/2+4\delta}t^{3/2}|\del^IL^J v| \big)
\lesssim C_1\vep s^{k\delta}.
\ee
\end{subequations}
The property \eqref{decay refined 0de} is essentially a special case of \eqref{decay refined kde}: we will establish it first and it will next serve in the proof of \eqref{decay refined kde}, done by induction on $k$.

The sup-norm estimate for the Klein-Gordon component \eqref{Linfty KG ineq a} and the sup-norm estimate for the wave equation \eqref{Linfini wave ineq} will now be used. We proceed with the following calculation:
\be
-\Box \big( \del^IL^J u\big)
 = P^{\alpha\beta}\del^IL^J\big( \del_\alpha v\del_\beta v\big) + R\del^IL^J(v^2),
\ee
\be
-\Box \big(\del^IL^J v \big)
+ H^{\alpha\beta}u\del_\alpha\del_\beta\del^IL^J v + c^2\del^IL^J v
 = [H^{\alpha\beta}u\del_\alpha\del_\beta,\del^IL^J]v.
\ee
We also recall by \eqref{Linfty KG ineq a}, the $R_i$ terms in this context (with $h^{\alpha\beta} = H^{\alpha\beta}u$) read as follows: 
$$
\aligned
R_1[\del^IL^Jv] &= \bigg(s^{3/2}\sum_a\delb_a\delb_a  + \frac{x^ax^b}{s^{1/2}}\delb_a\delb_b  + \frac{3}{4s^{1/2}} + \sum_a\frac{3x^a}{s^{1/2}}\delb_a\bigg)\del^IL^Jv,
\\
R_2[\del^IL^Jv] &=\hb^{00}\bigg(\frac{3\del^IL^Jv}{4s^{1/2}} + 3s^{1/2}\delb_0 \del^IL^Jv\bigg)
\\
& \quad - s^{3/2}\big(2\hb^{0b}\delb_0\delb_b\del^IL^Jv + \hb^{ab}\delb_a\delb_b\del^IL^Jv + h^{\alpha\beta}\del_\alpha\Psib^{\beta'}_\beta\,\delb_{\beta'}\del^IL^Jv\big),
\\
R_3[\del^IL^Jv] &= \hb^{00}\bigg(2x^as^{1/2}\delb_0\delb_a + \frac{2x^a}{s^{1/2}}\delb_a  + \frac{x^ax^b}{s^{1/2}}\delb_a\delb_b\bigg)\del^IL^Jv. 
\endaligned
$$
Hence, the following four terms must be controlled:
\be
\del^IL^J\big(\del_\alpha v\del_\beta v\big),\quad \del^IL^J(v^2),\quad R_i[\del^IL^Jv],
\quad [H^{\alpha\beta}u\del_\alpha\del_\beta,\del^IL^J]v.
\ee


\subsection{First improvement of the sup-norm of the wave component}

We now  present estimates which use only the basic sup-norm estimates already established in Sections \ref{subsec basic Linfini 1} and \ref{subsec basic Linfini 2}. We first estimate the terms $\del^IL^J\big(\del_\alpha v\del_\beta v\big)$ and $\del^IL^J(v^2)$.

\begin{lemma}
\label{lem refine1 Ws}
If the energy bounds \eqref{ineq energy assumption} hold, then for all $|I| + |J|\leq N-7$ with $|J|=k$, the following estimate holds in the region $\Kcal_{[2,s_1]}$:
\be
\big|\del^IL^J\big(\del_\alpha v\del_\beta v\big)\big| + \big|\del^IL^J\big(v^2\big)\big|\leq C(C_1\vep)^2 t^{-3}s^{(k+4)\delta}.
\ee
\end{lemma}

\begin{proof} We have 
\be
\label{new-formule2}
\del^IL^J\big(\del_\alpha v\del_\beta v\big) = \sum_{I_1+I_2 = I\atop J_1+J_2 = J}\del^{I_1}L^{J_1}\del_\alpha v \del^{I_2}L^{J_2}v, 
\ee
where, in the right-hand side, each term satisfies $|I_1|+|I_2| = |I|$ and $|J_1|+|J_2|=|J|$. Then we obtain 
$$
\big|\del^{I_1}L^{J_1}\del_\alpha v \del^{I_2}L^{J_2}v\big|\leq C(C_1\vep)^2 s^{(|J_1|+2)\delta}s^{(|J_2|+2)\delta}t^{-3}
= C(C_1\vep)^2t^{-3}s^{(k+4)\delta},
$$
where we have used the third inequality in \eqref{decay basic 4ge2} for each term. The estimate of $\del^IL^J\big(v^2\big)$ is derived similarly.
\end{proof}

We improve the bound on $u$, as follows.

\begin{proposition}[First improvement of the sup-norm of the wave component]
\label{lem refine1 W} 
For $|I|+|J|\leq N-7$ one has 
\bel{ineq lem refine1 W}
|\del^IL^Ju(t,x)|\lesssim C_0\vep t^{-3/2} + (C_1\vep)^2(s/t)^{(k+4)\delta}t^{-1}s^{(k+4)\delta}.
\ee
\end{proposition}

\begin{proof}
The proof is a direct application of the sup-norm estimate for the wave equation \eqref{Linfini wave ineq}. First of all, $\del^IL^Ju$ solves the Cauchy problem
$$
\aligned
&\Box \del^I L^J u = \del^I L^J\big(P^{\alpha\beta}\del_\alpha v\del_\beta v\big) + \del^I L^J\big(Rv^2\big),
\\
&\del^IL^J u(2,x) = U_0(I,J,x), \quad \del_t \del^IL^J u(2,x) = U_1(I,J,x),
\endaligned
$$
where $U_0(I,J,x)$ and $U_1(I,J,x)$ are restrictions of $\del^IL^J u$ and $\del^IL^J u$ on the initial hyperplane $\{t=2\}$. We remark that they are  linear combinations of $\del_x^{I'} u$ and $\del_t\del_x^{I'}u$ with $|I'|\leq |I|+|J|$. Hence, $u$ is decomposed as follows:
$$
u(t,x) = w_1(t,x) + w_2(t,x)
$$
with
$$
\aligned
&\Box L^Jw_2 = L^J\big( P^{\alpha\beta}\del_\alpha v\del_\beta v\big) + L^J\big(Rv^2\big),
\\
& w_2(2,x) = \del_tw_2(2,x) = 0,
\endaligned
$$
while
$$
\aligned
&\Box w_1 = 0,
\\
& w_1(2,x) = U_1(I,J,x),\quad \del_t w_1(2,x) = U_2(I,J,x).
\endaligned
$$

The sup-norm bound for $w_1$ comes directly from the explicit expression of the solutions (cf., for instance, \cite{Sogge}) while for $w_2$ we apply \eqref{Linfini wave ineq}. Observe that for the terms in the right-hand side:
\bel{eq prq lem refine1 W}
\aligned
\big|\del^I L^J\big( P^{\alpha\beta}\del_\alpha v\del_\beta v\big)\big| + \big|\del^I L^J\big(Rv^2\big)\big|
\lesssim & (C_1\vep)^2t^{-3}s^{(k+4)\delta}
\\
\lesssim & (C_1\vep)^2t^{-2-(1-(2+k/2)\delta)}(t-r)^{-1+(1+(2+k/2)\delta)}
\endaligned
\ee
in $\Kcal_{[2,s]}\subset \Kcal_{[2,s_1]}$. Recall that this estimate also holds in $\{(t,x)|r<t-1,t^2-r^2\leq s_1^2,t\geq 2\}$. Then by \eqref{Linfini wave ineq}, the desired result is proven.
\end{proof}

We next estimate the terms $R_i[\del^IL^Jv]$.

\begin{lemma}
\label{lem refine1 R}
For $|I|+|J|\leq N-4$ with $|J| = k$, the following estimates hold in $\Kcal_{[2,s_1]}$:
\begin{subequations}
\bel{ineq refine1 R a}
\big|R_1[\del^IL^Jv]\big|\lesssim C_1\vep (s/t)^{3/2}s^{-3/2+(k+4)\delta},
\ee
\bel{ineq refine1 R b}
\big|R_2[\del^IL^J v]\big| \lesssim C_1\vep |tu|(s/t)^{3/2}s^{-3/2+(k+3)\delta} + (C_1\vep)^2(s/t)^{3/2}s^{-3/2+(k+3)\delta},
\ee
\be
\big|R_3[\del^IL^J v]\big|\lesssim (C_1\vep)^2(s/t)s^{-2+(k+4)\delta}
+ C_1\vep|tu|(s/t)^{3/2}s^{-3/2+(k+3)\delta}.
\ee
\end{subequations}
\end{lemma}

\begin{proof}
The proof is a substitution of the basic sup-norm estimates into the corresponding expression. We begin with $R_1$ and focus first on $\delb_a\delb_b\del^IL^Jv$: 
\bel{eq pr1 lem refine1 R}
\aligned
\delb_a\delb_b \del^IL^Jv
= \, & t^{-1}L_a\big(t^{-1}L_b\del^IL^J v\big)
= t^{-1}L_a\big(t^{-1}\del^IL_bL^Jv + t^{-1}[L_b,\del^I]L^Jv\big)
\\
= \, &t^{-1}L_a\big(t^{-1}\big)\del^IL_bL^Jv + t^{-2}L_a\del^IL_bL^Jv + t^{-1}L_a\big(t^{-1}\big)[L_b,\del^I]L^Jv
\\
& + t^{-2}L_a[L_b,\del^I]L^Jv
\\
= \, &t^{-1}L_a\big(t^{-1}\big)\del^IL_bL^Jv + t^{-2}\del^IL_aL_bL^Jv + t^{-2}[L_a,\del^I]L_bL^Jv
\\
&+t^{-1}L_a\big(t^{-1}\big)[L_b,\del^I]L^Jv + t^{-2}L_a[L_b,\del^I]L^Jv.
\endaligned
\ee
For the last term, we apply \eqref{pre lem commutator pr1} as follows:
$$
\aligned
t^{-2}L_a[L_b,\del^I]L^Jv
=& - t^{-2}\sum_{|I'|\leq |I|}\lambda_{bI'}^{I}L_a\del^{I'}L^Jv
\\
=& - t^{-2}\sum_{|I'|\leq |I|}\lambda_{bI'}^{I}\del^{I'}L_a L^Jv - t^{-2}\sum_{|I'|\leq |I|}\lambda_{bI'}^{I}[L_a,\del^{I'}]L^Jv
\\
=& - t^{-2}\sum_{|I'|\leq |I|}\lambda_{bI'}^{I}\del^{I'}L_a L^Jv
+ t^{-2}\sum_{|I'|\leq |I|}\lambda_{bI'}^{I}\sum_{|I''|\leq|I'|}\lambda_{aI''}^{I'}\del^{I''}L^Jv, 
\endaligned
$$
and the term $[L_a,\del^I]L_bL^J v$ is bounded in the same manner. Then we conclude that
$$
\big|t^{-2}L_a[L_b,\del^I]L^Jv\big|\leq Ct^{-2}\sum_{|I'|\leq|I|\atop |J'|\leq |J|+1}\big|\del^{I'}L^{J'}v\big|.
$$
In view of the inequality
$\big|L_a\big(t^{-1}\big)\big|\leq Ct^{-1}$ (in $\Kcal$),
  the terms in the right-hand side of \eqref{eq pr1 lem refine1 R} are bounded by
$
Ct^{-2}\sum_{|I'|\leq|I|\atop |J'|\leq |J|+2}\big|\del^{I'}L^{J'}v\big|.
$
Then, by the last equation in \eqref{decay basic 1ge1R}, 
we  have 
$$
\big|s^{3/2}\delb_a\delb_b \del^IL^Jv\big|\lesssim C_1\vep (s/t)^{7/2}s^{-3/2+(k+4)\delta}
$$
and, similarly,
$$
\aligned
\big|x^as^{-1/2}\delb_a  \del^IL^Jv \big|
&\lesssim C_1\vep (s/t)^{3/2}s^{-3/2+(k+3)\delta},
\\
\big|x^ax^b s^{-1/2}\delb_a\delb_b  \del^IL^Jv\big|
&\lesssim C_1\vep (s/t)^{3/2}s^{-3/2+(k+4)\delta},
\\
\big|s^{-1/2}\del^IL^Jv\big|
& \lesssim C_1 \vep (s/t)^{3/2}s^{-3/2 +(k+2)\delta}.
\endaligned
$$
So we conclude that
$$
\big|R_1[\del^IL^Jv]\big|\lesssim C_1\vep t^{-3/2}s^{(k+4)\delta} \lesssim C_1\vep (s/t)^{3/2}s^{-3/2+(k+4)\delta}.
$$


For the derivation in the paragraph above, let us provide some more details by observing that, for the first term, 
$$
|\delb_a\delb_b \del^IL^J v| = |t^{-1}L_a\big(t^{-1}L_b\del^IL^J v|\lesssim t^{-2}|L_aL_b\del^IL^J v| + t^{-1}|L_a(t^{-1})|\cdot|L_b\del^IL^J v|. 
$$
Here, we remark that
$$
s^{3/2}t^{-2}|L_aL_b\del^IL^J v|\lesssim t^{-2}C_1\vep s^{3/2}t^{-2}s^{1/2+(|I|+2+2)\delta} t^{-3/2}
\lesssim (s/t)^{7/2}s^{-3/2+ (k+4)\delta}.
$$
In the first inequality, we have used the fact that
$$
|L_aL_b\del^IL^J v|\lesssim |\del^IL_aL_bL^Jv| + |[L_aL_b,\del^I]L^Jv|. 
$$
For the first term in the right-hand-side of the above inequality, we get the upper bound $CC_1\vep t^{-3/2}s^{1/2+(|J|+2+2)\delta}t^{-3/2}$ (for $|I|+|J|+2\leq N-2 $) and for the second we recall the estimate on commutator (using \eqref{pre lem commutator pr1} twice), and see that
$$
|[L_aL_b,\del^I]L^Jv|\lesssim \sum_{|I'|\leq|I|\atop|J'|\leq|J+1|}\del^{I'}|L^{J'}\del^{I'}L^{J'}v|\lesssim C_1\vep t^{-3/2}s^{1/2+(|J|+1+2)\delta}.
$$

The estimates of $R_2$ and $R_3$ are quite similar. We just need to observe that, by \eqref{decay basic ray-integral}
and by recalling that $|\Hb^{00}|\leq C(t/s)^2$, $|\Hb^{a0}|\leq C(t/s)$, and $|\Hb^{ab}|\leq C$, we obtain
\bel{eq pr2 lem refine1 R}
|\hb^{00}| = |\Hb^{00} u|\lesssim C_1\vep t^{1/2}s^{-1},
\qquad |\hb^{a0}|\lesssim C_1\vep t^{-1/2}, 
\qquad |\hb^{ab}u|\lesssim C_1\vep t^{-3/2}s
\ee
and
$$
\aligned
\big|\delb_0\delb_a \del^IL^J v\big| =& (s/t)\big|\del_t\delb_a \del^IL^J v\big|
\\
\leq &(s/t)t^{-2}\big|L_a \del^IL^J v\big| + t^{-1}(s/t)\big|\del_t L_a \del^IL^J v\big|.
\endaligned
$$
As was done in \eqref{eq pr1 lem refine1 R} and by applying \eqref{pre lem commutator pr1} and the fifth equation in
 \eqref{decay basic 1ge2} 
we find
\bel{eq pr3 lem refine1 R}
\big|\delb_0\delb_a \del^IL^J v\big|\lesssim C_1\vep (s/t)^{7/2} s^{-2+(k+3)\delta}.
\ee
Equipped with \eqref{eq pr2 lem refine1 R} and \eqref{eq pr3 lem refine1 R}, we see that in $R_2[\del^IL^J v]$,
$$
\aligned
\big|s^{-1/2}\hb^{00}\del^IL^Jv\big|
& \lesssim (C_1\vep)^2 (s/t)s^{-2 + (k+2)\delta}\lesssim (C_1\vep)^2(s/t)^{3/2}s^{-3/2+(k+2)\delta},
\\
\big|s^{1/2}\hb^{00}\delb_0\del^IL^Jv\big|
&\lesssim C_1\vep |tu|(s/t)^{3/2}s^{-3/2+(k+2)\delta},
\\
s^{3/2}\big|\hb^{0b}\delb_0\delb_b\del^IL^J v\big|
&\lesssim C_1\vep |tu|(s/t)^{5/2}s^{-3/2+(k+3)\delta},
\\
s^{3/2}\big|\hb^{ab}\delb_a\delb_b\del^IL^J v\big|
&\lesssim (C_1\vep)^2 (s/t)^5s^{-2+(k+4)\delta},
\\
s^{3/2}\big|h^{\alpha\beta}\del_\alpha\Psib^{0}_\beta\,\delb_{0}\del^IL^Jv\big|
&\lesssim C_1\vep|tu|(s/t)^{3/2}s^{-3/2+(k+2)\delta},
\\
s^{3/2}\big|h^{\alpha\beta}\del_\alpha\Psib^{b}_\beta\,\delb_{b}\del^IL^Jv\big|
& = 0,  
\endaligned
$$
while, in the expression $R_3[\del^IL^J]v$,
$$
\aligned
\big|\hb^{00}x^as^{1/2}\delb_0\delb_a\del^IL^J v\big|
&\lesssim C_1\vep |tu|(s/t)^{3/2}s^{-3/2+(k+3)\delta},
\\
\left|\hb^{00}x^as^{-1/2}\delb_a\del^IL^Jv\right|&\lesssim (C_1\vep)^2(s/t)^{3/2}s^{-3/2+(k+3)\delta},
\\
s^{-1/2}\big|\hb^{00}x^ax^b\delb_a\delb_b \del^IL^J v\big|
&\lesssim (C_1\vep)^2(s/t)s^{-2+(k+4)\delta}.
\endaligned
$$
\end{proof}


\subsection{Second improvement on the wave component and first improvement on the Klein-Gordon component}

We now establish \eqref{decay refined 0de w}-\eqref{decay refine 0de KG} and, for latter use, we first derive the following improved estimates on the terms $R_i$.

\begin{lemma}
\label{lem refine2 R} 
For $|I|+|J|\leq N-4, |J| = k$, the following estimates hold in $\Kcal_{[2,s_1]}$:
\bel{ineq refine2 R}
\aligned
\sum_{i=1}^3\big|R_i[\del^IL^Jv]\big|
\lesssim \, & C_1\vep (s/t)^{3/2}s^{-3/2 + (k+7)\delta}.
\endaligned
\ee
\end{lemma}

\begin{proof}
This is a combination of Lemma \ref{lem refine1 R} and \eqref{ineq lem refine1 W} (take $k=0$ then considering the condition $C_1\vep \leq 1$) and the fact that in $\Kcal, t^{1/2}\leq s\leq t$.
\end{proof}

Then we need to estimate the term $|h_{t,x}'(\lambda)|$ in Proposition \ref{Linfini KG}.

\begin{lemma}
\label{lem h_{t,x}'} 
The following estimate holds for $(t,x)\in \Kcal_{[2,s_1]}$:
\bel{ineq lem h_{t,x}'}
\int_{s_0}^s |h_{t,x}'(\lambda)|d\lambda \lesssim C_1\vep,
\ee
where
$
h_{t,x}(\lambda) := \hb^{00}(\lambda t/s,\lambda x/s) = \Hb^{00}(\lambda t/s,\lambda x/s)u(\lambda t/s,\lambda x/s).
$
\end{lemma}

\begin{proof} With the notation of Proposition \ref{Linfini KG}, we have
$\hb^{00} = \Hb^{00}u$, and we observe that
$$
\Hb^{00} = H^{\alpha\beta}\Psib^0_\alpha\Psib^0_\beta = H^{00} (t/s)^2  - 2\sum_{a}H^{0a}(x^at/s) + \sum_{a,b}H^{ab}(x^ax^b/s^2).
$$
Note that
$\Hb^{00} (\lambda t/s,\lambda x/s) = \Hb^{00}(t,x)$,
so that  $\Hb^{00}$ is constant along the segment $(\lambda t/s, \lambda x/s), s_0\leq \lambda\leq s$. So we find
$$
h_{t,x}'(\lambda) = \Hb^{00}(t,x)(t/s)\newperp  u(\lambda t/s,\lambda x/s),
$$
and we conclude that
$$
|h_{t,x}'(\lambda)|\leq C (t/s)^3 |\newperp u(\lambda t/s,\lambda x/s)|.
$$

Next, we observe the identity
\be
\newperp  u = \frac{s^2}{t^2}\del_t u + \frac{x^a}{t}\delu_a u = \frac{s^2}{t^2}\del_t u + \frac{x^a}{t^2}L_a u
\ee
and, by the first inequality in \eqref{decay basic 3ge2} and \eqref{ineq lem refine1 W} with $\del^IL^J = L_a$,
$$
|\newperp  u|\lesssim C_1\vep (s/t)t^{-3/2} + C_1\vep (s/t)^{5\delta}t^{-2}s^{5\delta}.
$$
Therefore, we obtain
$$
|h_{t,x}'(\lambda)| \lesssim C_1\vep(s/t)^{-1/2}\lambda^{-3/2} + C_1\vep (s/t)^{-1+5\delta}\lambda^{-2+5\delta}.
$$

Then, to apply the sup-norm estimate for the Klein-Gordon equation \eqref{Linfty KG ineq a}, we proceed as follows. In the range $0\leq r/t\leq 3/5$, we have $4/5\leq s/t\leq 1$, and 
$$
\int_{s_0}^s |h_{t,x}'(\lambda)|d\lambda \lesssim C_1\vep \int_2^s \lambda ^{-3/2} \, d\lambda\lesssim C_1\vep.
$$
In the range $3/5< r/t< 1$, we obtain
$$
\aligned
\int_{s_0}^s |h_{t,x}'(\lambda)|d\lambda
\lesssim \, &  C_1\vep (s/t)^{-1/2}\int_{s_0}^s\lambda^{-3/2} \, d\lambda 
+ C_1\vep (s/t)^{-1+5\delta}\int_{s_0}^s\lambda^{-2+5\delta} \, d\lambda
\\
\lesssim \, & C_1\vep (s/t)^{-1/2}s_0^{-1/2} + C_1\vep (s/t)^{-1+5\delta}s_0^{-1+5\delta}.
\endaligned
$$
We recall that, when $3/5<r/t<1$, $s_0 = \sqrt{\frac{t+r}{t-r}}\geq t/s$, so  that
$
\int_{s_0}^s |h_{t,x}'(\lambda)|d\lambda \lesssim C_1\vep,
$
and the desired result is established.
\end{proof}

Now we give a second application of the sup-norm estimate for the Klein-Gordon component \eqref{Linfty KG ineq a}.

\begin{proposition}[Second improvement on the wave component and first improvement on the Klein-Gordon component]
\label{lem refine2 KG-W} 
The following estimate also holds in $\Kcal_{[2,s_1]}$:
\begin{subequations}\label{ineq lem refine2 KG-W-0}
\bel{ineq lem refine2 KG-W-0 a}
|v(t,x)| +
{t \over s} |\newperp  v(t,x)|
\lesssim C_1\vep (s/t)^{2-7\delta}s^{-3/2},
\ee
\bel{ineq lem refine2 KG-W-0 b}
|u(t,x)| \lesssim C_1\vep t^{-1}.
\ee
\end{subequations}
\end{proposition}

\begin{proof} We rely here on \eqref{Linfty KG ineq a} and the sup-norm estimate for the wave equation \eqref{Linfini wave ineq}, and we first establish \eqref{ineq lem refine2 KG-W-0 a}. In view of \eqref{ineq lem h_{t,x}'}, we have
$$
\aligned
\int_{s_0}^s |h_{t,x}'(\sbar)|e^{\int_{\sbar}^s|h_{t,x}'(\lambda)|d\lambda} \, d\sbar
& \leq \int_{s_0}^s |h_{t,x}'(\sbar)|e^{\int_{s_0}^s|h_{t,x}'(\lambda)|d\lambda} \, d\sbar
\\
& \lesssim 
 \int_{s_0}^s |h_{t,x}'(\sbar)|e^{CC_1\vep} \, d\sbar\lesssim C_1\vep.
\endaligned
$$
On the other hand, to estimate $F(\sbar)$, we write 
$$
\aligned
F(s) =& \int_{s_0}^s \big(R_1[v] + R_2[v] + R_3[v]\big)(\lambda t/s,\lambda x/s) \, d\lambda
\\
\lesssim \, & C_1\vep\int_{s_0}^s (s/t)^{3/2}\lambda^{-3/2+7\delta} d\lambda
\lesssim C_1\vep (s/t)^{3/2}s_0^{-1/2+7\delta}.
\endaligned
$$
Now for $0\leq r/t\leq 3/5$ we see that  $4/5\leq s/t\leq 1$ and $s_0 = 2$, and we have
$$
F(s)\lesssim C_1\vep (s/t)^{3/2}s_0^{-1/2+7\delta}\lesssim C_1\vep (s/t)^{2-7\delta}.
$$
For $3/5< r/t <1$, we see that $s_0 = \sqrt{\frac{t+r}{t-r}}\leq t/s$, so that
\bel{pr2 lem refine2 KG-W}
F(s)\lesssim C_1\vep (s/t)^{2-7\delta}.
\ee
Then, by combining \eqref{ineq lem h_{t,x}'}, \eqref{pr2 lem refine2 KG-W} and \eqref{Linfty KG ineq a}, we conclude that \eqref{ineq lem refine2 KG-W-0 a} holds.
On the other hand,  \eqref{ineq lem refine2 KG-W-0 b} follows directly from substituting \eqref{ineq lem refine2 KG-W-0 a} into \eqref{Linfini wave ineq}.


\newcommand {\Pu}{\underline{P}}

Let us explain in more detail the above argument. In the equation 
$$
\Box u = P_{\alpha\beta}\del_{\alpha}v\del_{\beta}v + Rv^2, 
$$ 
we need to estimate $\left|P^{\alpha\beta}\del_{\alpha}v\del_{\beta}v\right|$ and $\left|Rv^2\right|$. First we rewrite the expresison  $P^{\alpha\beta}\del_{\alpha}v\del_\beta v$ in the semi-hyperboloidal frame: 
$$
\aligned
P^{\alpha\beta}\del_\alpha v\del_\beta v =& \Pu^{\alpha\beta}\delu_\alpha v\delu_\beta v
\\
=&\Pu^{00}\del_tv\del_tv + \Pu^{a0}\delu_av\del_t v + \Pu^{0b}\del_t v\delu_b v + \Pu^{ab}\delu_a v\delu_b v. 
\endaligned
$$
The last three terms in the right-hand-side can be controlled by applying the first and the third inequalities in 
\eqref{decay basic 4ge2} 
$$
\left|P^{a0\beta}\delu_av\delu_\beta v\right| + \left|P^{\alpha b}\delu_{\alpha}\delu_b v\right|\leq C(C_1\vep)^2t^{-4}s^{5\delta}.
$$

For the first term, we see that
$$
\del_tv = \frac{t^2}{s^2}\left(\delu_{\perp}v - \frac{x^a}{t}\delu_a v\right).
$$
Then by \eqref{ineq lem refine2 KG-W-0 a} 
and the third inequality in \eqref{decay basic 4ge2}, we obtain  
$$
\aligned
\left|\del_tv\right|\leq& CC_1\vep (s/t)^{1-7\delta}s^{-3/2} + CC_1\vep t^{-5/2}s^{3\delta}
\\
\leq&CC_1\vep (s/t)^{1-7\delta}s^{-3/2}.
\endaligned
$$
This leads to
$$
\left|\Pu^{00}\del_tv \del_tv\right|\leq C(C_1\vep)^2(t-r)^{-1+(1/2 - 7\delta/2)}t^{-2-(1/2-7\delta/2)}.
$$
The term $\left|Rv^2\right|$ is bounded directly by \eqref{ineq lem refine2 KG-W-0 a}, and we have  
$$
\left|Rv^2\right|\leq C(C_1\vep)^2(s/t)^{4-14\delta}s^{-3}\leq CC_1\vep t^{-3}.
$$
Then by applying \eqref{Linfini wave ineq},  
the desired bound on $u$ is guaranteed. 
\end{proof}

With \eqref{ineq lem refine2 KG-W-0 b}, we can improve again the estimate on $R_i$. Namely, the proof of the following estimate is immediate by substituting \eqref{ineq lem refine2 KG-W-0 b} into \eqref{new-formule2}, and using \eqref{Linfini wave ineq}. 

\begin{lemma}\label{last 2} 
The following estimates hold:
\bel{lem refine2' R}
\sum_{i=1}^3 R_i[\del^I L^J v]\lesssim C_1\vep (s/t)^{3/2}s^{-3/2+(k+4)\delta}.
\ee
\end{lemma}


\subsection{Second improvement on the Klein-Gordon component}

We now establish \eqref{decay refine 0de KG} and, to do so, our first task is to estimate the commutator $[H^{\alpha\beta}u\del_\alpha\del_\beta,\del^IL^J]$. First of all, from the following identities 
\be
\del_t = \frac{t^2}{s^2}\big(\newperp  - (x^a/t)\delu_a\big),\qquad \del_a = -\frac{tx^a}{s^2}\newperp  + \frac{x^ax^b}{t^2}\delu_b + \delu_a,
\ee
the following estimates are immediate:
\bel{ineq Ur-partial}
\aligned
&\big|\del_t\del^IL^J v\big|\leq (t/s)^2\big|\newperp  \del^IL^Jv\big| + (t/s)^2\sum_a\big|\delu_a\del^IL^Jv\big|,
\\
&\big|\del_a\del^IL^J v\big|\leq (t/s)^2\big|\newperp  \del^IL^Jv\big| + C(t/s)^2\sum_a\big|\delu_a\del^IL^Jv\big|.
\endaligned
\ee
Based on this result, we estimate the commutator $[H^{\alpha\beta}u\del_\alpha\del_\beta,\del^I]$. (In the statement below, as usual, a sum over the empty set vanishes.)  

\begin{lemma}
\label{lem refine2 commu} 
The following estimates are valid in $\Kcal$ for $|I|+|J|\leq N-4$ and $|J| = k$:
\bel{ineq lem refine2 commu}
\aligned
&\big|[H^{\alpha\beta}u\del_\alpha\del_\beta,\del^IL^J]v\big|
\\
&\lesssim C_1\vep t^{-1}(s/t)^{-2}\sum_{|I_2|\leq |I|,\beta\atop |J_2|\leq |J|-1}|\newperp \del_\beta\del^{I_2}L^{J_2}v|
+\sum_{J_1+J_2 =J\atop |J_1|\geq 1} |L^{J_1}u| \,  |\del^IL^{J_2}\del_\alpha\del_\beta v|
\\
&\quad + C_1\vep t^{-3/2}(s/t)^{-3}\sum_{|I_2|+|J_2|\leq |I|+|J|-1,\beta\atop N-7\leq |I_2|+|J_2|\leq N-5}|\newperp \del_\beta\del^{I_2}L^{J_2}v|
\\
&\quad +(C_1\vep)^2(s/t)^{3/2}s^{-3+(k+4)\delta}. 
\endaligned
\ee 
\end{lemma}

\begin{proof} We write the decomposition 
$$
\aligned
\,[H^{\alpha\beta}u\del_\alpha\del_\beta,\del^IL^J]v
=&\sum_{{I_1+I_2 = I\atop J_1+J_2=J}\atop |I_1|+|J_1|\geq 1} H^{\alpha\beta}\del^{I_1}\del^{J_1}u\del^{I_2}L^{J_2} \del_\alpha\del_\beta v
 + H^{\alpha\beta}u\del^I ([L^J,\del_\alpha\del_\beta]v)
\\
=:& T_0 + T_7.
\endaligned
$$
We recall that by \eqref{pre lem commutator second-order} and \eqref{ineq lem refine2 KG-W-0 b}, $T_7$ is bounded as follows:
$$
\aligned
|T_7|\lesssim \, & C_1\vep t^{-1}\sum_{\alpha,\beta\atop |J_2|\leq |J|-1}|\del^I\del_\alpha\del_\beta L^{J_2}v|
\\
\lesssim & C_1\vep t^{-1}(t/s)^2 \sum_{\beta\atop |J_2|\leq|J|-1}|\newperp \del_\beta\del^IL^{J_2}v|
 + C_1\vep t^{-1}(t/s)^2 \sum_{a,\beta\atop |J_2|\leq|J|-1}|\delu_a\del_\beta\del^IL^{J_2}v|.
\endaligned
$$
The second term in the right-hand side is bounded as follows:
$$
\aligned
C_1\vep t^{-1}(t/s)^2|\delu_a\del_\beta\del^IL^{J_2}v|\lesssim \, & C_1\vep t^{-1}(t/s)^2 C_1\vep t^{-5/2}s^{1/2+(k+3)\delta}
\\
\lesssim & (C_1\vep)^2 (s/t)^{3/2}s^{-3+(k+3)\delta}.
\endaligned
$$

We then write 
$$
\aligned
\big|T_0\big|
\leq & \sum_{I_1+I_2=I,|I_1|\geq 1\atop J_1+J_2=J,\alpha,\beta,\gamma}
|\del^{I_1}L^{J_1}u| \, |\del^{I_2}L^{J_2}\del_\alpha\del_\beta v|
+ \sum_{J_1+J_2 =J\atop |J_1|\geq 1,\alpha,\beta} |L^{J_1}u| \,  |\del^IL^{J_2}\del_\alpha\del_\beta v|
\\
=:& T_1 + T_2.
\endaligned
$$
Then we see that $T_1$ is again decomposed as follows:
$$
\aligned
T_1 =
& \sum_{{I_1+I_2=I,|I_1|\geq 1\atop J_1+J_2=J,\alpha,\beta,\gamma}\atop |I_2|+|J_2|\leq N-8}
|\del^{I_1}L^{J_1}u| \, |\del^{I_2}L^{J_2}\del_\alpha\del_\beta v|
+ \sum_{{I_1+I_2=I,|I_1|\geq 1\atop J_1+J_2=J,\alpha,\beta,\gamma}\atop N-7\leq |I_2|+|J_2|\leq N-5}
|\del^{I_1}L^{J_1}u| \, |\del^{I_2}L^{J_2}\del_\alpha\del_\beta v|
\\
=:& T_3 + T_4.
\endaligned
$$
We have 
$$
\aligned
T_3 \lesssim C_1\vep t^{-1/2}s^{-1+(|J_1|+2)\delta} C_1\vep t^{-3/2}s^{(|J_2|+2)\delta}
\lesssim (C_1\vep)^2(s/t)^2 s^{-3+(k+4)\delta},
\endaligned
$$
where we applied \eqref{new-formules} 
 and the fourth estimate in \eqref{decay basic 4ge2}. 
Then by applying \eqref{ineq Ur-partial} and in view of  \eqref{pre lem commutator second-order}, the term $T_4$ is bounded by 
$$
\aligned
T_4 \leq & \sum_{{I_1+I_2=I,|I_1|\geq 1\atop |J_1|+|J_2|\leq |J|,\alpha,\beta,\gamma} \atop N-7\leq |I_2|+|J_2|\leq N-5}
|\del^{I_1}L^{J_1}u| \,  |\del_\alpha\del_\beta\del^{I_2}L^{J_2}v|
\\
\leq & (t/s)^2\sum_{{I_1+I_2=I,|I_1|\geq 1\atop |J_1|+|J_2|\leq |J|\alpha,\beta,\gamma} \atop N-7\leq |I_2|+|J_2|\leq N-5}
|\del^{I_1}L^{J_1}u| \,  |\newperp \del_\beta\del^{I_2}L^{J_2}v|
\\
&\quad +(t/s)^2\sum_{{I_1+I_2=I,|I_1|\geq 1\atop |J_1|+|J_2|\leq |J|,a,\beta,\gamma} \atop N-7\leq |I_2|+|J_2|\leq N-5}
|\del^{I_1}L^{J_1}u| \,  |\delu_a\del_\beta\del^{I_2}L^{J_2}v|
=:& T_5 + T_6.
\endaligned
$$
Then, in the expression $T_5$, $N-7\leq |I_2|+|J_2|\leq N-5$ implies $|I_1|+|J_1|\leq 3\leq N-6$ and recall $|I_1|\geq 1$, so we see $|\del^{I_1}L^{J_1}u|\leq C\sum_{\gamma}\del_{\gamma}\del^{I_1'}L^{J_1}u$ were $|I_1'|+|J_1|\leq 2\leq N-6$. Then by the first estimate in \eqref{decay basic 3ge1RR}, we find 
%
$$
T_5\lesssim C_1\vep t^{-3/2}(s/t)^{-3}\sum_{|I_2|+|J_2|\leq |I|+|J|-1\atop N-7\leq |I_2|+|J_2|\leq N-5}|\newperp \del_\beta\del^{I_2}L^{J_2}v|.
$$
Furthermore, we have
$$
\aligned
T_6\lesssim \, & (t/s)^2 C_1\vep t^{-1/2}s^{-1}\, C_1\vep t^{-5/2}s^{1/2+(|J_2|+3)\delta}
\lesssim (C_1\vep)^2 (s/t) s^{-7/2+(k+2)\delta}
\\
\lesssim & (C_1\vep)^2(s/t)^{3/2}s^{-3+(k+3)\delta}
\endaligned
$$
and the desired estimate is established.
\end{proof}

We are now in a position to establish the desired bound \eqref{decay refine 0de KG}.

\begin{proposition}[Second improvement on the Klein-Gordon component]
\label{prop refine3 W0dg} 
The following estimate holds in $\Kcal_{[2,s_1]}$ for $|I|\leq N-4$:
\begin{subequations}\label{ineq prop refine3 W0dg}
\bel{ineq prop refine3 W0dg a}
|\newperp  \del^I v(t,x)|\lesssim C_1\vep(s/t)^{3/2-4\delta} t^{-3/2},
\ee
\bel{ineq prop refine3 W0dg b}
|\del^I v(t,x)|\lesssim C_1 \vep (s/t)^{1/2-4\delta}t^{-3/2}.
\ee
\end{subequations}
\end{proposition}

\begin{proof}
We first discuss the case where $|I|-1\geq N-7$ and, in this case,  using \eqref{ineq lem refine2 commu} 
$$
\aligned
\big|[H^{\alpha\beta}u\del_\alpha\del_\beta,\del^I]v \big|
\lesssim \, & C_1\vep t^{-3/2}(s/t)^{-3}\sum_{|I_2|\leq |I|-1,\beta\atop N-7\leq |I_2|\leq N-5}|\newperp  \del_\beta \del^{I_2}v|
+ (C_1\vep)^2 (s/t)^{3/2}s^{-3+4\delta}.
\endaligned
$$
For all $\sbar\in[s_0,s]$, using \eqref{lem refine2' R} and the above estimate we have 
$$
\aligned
F(\sbar) \leq \, & \sum_{i=1}^3\int_{s_0}^\sbar|R_i[\del^I v](\lambda t/s,\lambda x/s)| d\lambda
 + \int_{s_0}^\sbar \lambda^{3/2}\big|[H^{\alpha\beta}u\del_\alpha\del_\beta,\del^I]v \big|d\lambda
\\
\lesssim & C_1\vep (s/t)^{3/2}\int_{s_0}^\sbar \lambda^{-3/2+3\delta} \, d\lambda
+ (C_1\vep)^2(s/t)^{3/2}\int_{s_0}^\sbar \lambda^{-3/2+4\delta} d\lambda
\\
&
+ C_1\vep(s/t)^{-3/2}\sum_{|I_2|\leq |I|-1,\beta\atop N-7\leq |I_2|\leq N-5}\int_{s_0}^\sbar|\newperp  \del_\beta \del^{I_2}v(\lambda t/s,\lambda x/s)|d\lambda
\\
\lesssim & C_1\vep(s/t)^{-3/2}\sum_{|I_2|\leq |I|-1,\beta\atop N-7\leq |I_2|\leq N-5}\int_{s_0}^\sbar|\newperp  \del_\beta \del^{I_2}v(\lambda t/s,\lambda x/s)|d\lambda
+ C_1\vep(s/t)^{3/2}s_0^{-1/2+4\delta}.
\endaligned
$$

\noindent {\bf Case I: $3/5<r/t<1$}. In this case, $s_0 = \sqrt{\frac{t+r}{t-r}}\geq t/s$ and we have
$$
F(\sbar) \lesssim C_1\vep (s/t)^{2-4\delta} + C_1\vep(s/t)^{-3/2}\sum_{|I_2|\leq |I|-1,\beta\atop N-7\leq |I_2|\leq N-5}\int_{s_0}^\sbar|\newperp  \del_\beta \del^{I_2}v(\lambda t/s,\lambda x/s)|d\lambda.
$$
We define
$$
V_{t,x}(\lambda) := (\lambda t/s)^{3/2}\sum_{|I_2|\leq |I|-1,\beta\atop N-7\leq |I_2|\leq N-5}|\newperp  \del_\beta \del^{I_2}v(\lambda t/s,\lambda x/s)|
$$
and find
\bel{pr1 prop refine3 W0dg}
F(\sbar)\lesssim C_1\vep (s/t)^{2-4\delta} + C_1\vep \int_{s_0}^\sbar \lambda^{-3/2}V_{t,x}(\lambda) \, d\lambda, \quad s_0\leq \sbar\leq s.
\ee

Recalling the sup-norm estimate for the Klein-Gordon component \eqref{Linfty KG ineq a} in the case $1>r/t>3/5$, we obtain 
$$
|\newperp  \del^I v(t,x)|\leq Cs^{-1/2}t^{-1}\Big(F(s)+\int_{s_0}^sF(\sbar)|h_{t,x}'(\sbar)|e^{\int_{\sbar}^s|h_{t,x}'(\theta)|d\theta} \, d\sbar\Big).
$$
We replace $(t,x)$ by $(\lambda t/s,\lambda x/s)$ with $s_0\leq \lambda \leq s$, we see that $(\lambda t/s,\lambda x/s)$ is again contained in $\Kcal_{[2,s_1]}$. Then \eqref{Linfty KG ineq a} still holds, and so 
$$
|\newperp  \del^I v(\lambda t/s,\lambda x/s)|
\leq C(s/t)\lambda^{-3/2}\Big(F(\lambda)+\int_{s_0}^\lambda F(\sbar)|h_{t,x}'(\sbar)|e^{\int_{\sbar}^\lambda |h_{t,x}'(\theta)|d\theta} \, d\sbar\Big), 
$$
which implies 
$$
(\lambda t/s)^{3/2}|\newperp  \del^I v(\lambda t/s,\lambda x/s)|
\lesssim C_1\vep (s/t)^{-1/2}\Big(F(\lambda)+\int_{s_0}^\lambda F(\sbar)|h_{t,x}'(\sbar)|e^{\int_{\sbar}^\lambda |h_{t,x}'(\theta)|d\theta} \, d\sbar\Big).
$$
Recall that \eqref{ineq lem h_{t,x}'} holds for $1>r/t>3/5$ and that $F$ is increasing, then 
$$
\int_{s_0}^\lambda F(\sbar)|h_{t,x}'(\sbar)|e^{\int_{\sbar}^\lambda |h_{t,x}'(\theta)|d\theta} \, d\sbar
\leq F(\lambda)\int_{s_0}^\lambda  |h_{t,x}'(\sbar)|e^{\int_{\sbar}^\lambda |h_{t,x}'(\theta)|d\theta} \, d\sbar
\lesssim C_1\vep F(\lambda).
$$
So we see that
$$
(\lambda t/s)^{3/2}|\newperp  \del^I v(\lambda t/s,\lambda x/s)|\lesssim C_1\vep(s/t)^{-1/2} F(\lambda)
$$
and, in combination with \eqref{pr1 prop refine3 W0dg},
$$
(\lambda t/s)^{3/2}|\newperp  \del^I v(\lambda t/s,\lambda x/s)|\lesssim C_1\vep (s/t)^{3/2-4\delta} + C_1\vep (s/t)^{-1/2}\int_{s_0}^\lambda \sbar^{-3/2}V_{t,x}(\sbar)d\sbar, 
$$
which implies (by taking sum over $N-6\leq |I|\leq N-4 $):
\bel{pr2 prop refine3 W0dg}
V_{t,x}(\lambda) \lesssim C_1\vep (s/t)^{3/2-4\delta} + C_1\vep (s/t)^{-1/2}\int_{s_0}^\lambda \sbar^{-3/2}V_{t,x}(\sbar)d\sbar.
\ee
Then, by Gronwall lemma, we see that
$$
\aligned
\int_{s_0}^\lambda \sbar^{-3/2}V_{t,x}(\sbar)d\sbar
\lesssim \, & C_1\vep (s/t)^{3/2-4\delta}\int_{s_0}^{\lambda}\sbar^{-3/2}e^{CC_1\vep(s/t)^{-1/2}\int_{s_0}^\sbar\theta^{-3/2}d\theta} \, d\sbar
\\
\lesssim \, & C_1\vep (s/t)^{3/2-4\delta}\int_{s_0}^{\lambda}\sbar^{-3/2}e^{CC_1\vep(s/t)^{-1/2}s_0^{-1/2}} \, d\sbar
\\
\lesssim \, & C_1\vep (s/t)^{3/2-4\delta}s_0^{-1/2}e^{CC_1\vep(s/t)^{-1/2}s_0^{-1/2}}
\endaligned
$$
Here we recall that $s_0 = \sqrt{\frac{t+r}{t-r}}\geq t/s$, then
\bel{pr3 prop refine3 W0dg}
V_{t,x}(\lambda)\lesssim C_1\vep (s/t)^{3/2-4\delta}.
\ee

Now we substitute \eqref{pr3 prop refine3 W0dg} into \eqref{pr1 prop refine3 W0dg}, and obtain 
\be
F(\sbar) \lesssim C_1\vep (s/t)^{2-4\delta},\quad s_0\leq \sbar \leq s.
\ee
Then we apply the sup-norm estimate \eqref{Linfty KG ineq a} in the case $1>r/t>3/5$ and considering \eqref{ineq lem h_{t,x}'},
\be
|\newperp \del^I v(t,x)|\lesssim C_1\vep (s/t)^{3-4\delta}s^{-3/2},\qquad |\del^I v(t,x)|\lesssim C_1\vep (s/t)^{2-4\delta}s^{-3/2}.
\ee

\

\noindent{\bf Case II: $0\leq r/t\leq 3/5$}.
In this case, $4/5\leq s/t\leq 1$ and $s_0 = 2$, so the discussion is simpler. We just remark that as in the former case,
$$
\aligned
F(\sbar)\lesssim & C_1\vep (s/t)^{3/2}s_0^{-1/2+6\delta}
+ C_1\vep(s/t)^{-3/2}\sum_{|I_2|\leq |I|-1,\beta\atop N-7\leq |I_2|\leq N-5}\int_{s_0}^\sbar|\newperp  \del_\beta \del^{I_2}v(\lambda t/s,\lambda x/s)|d\lambda
\\
\lesssim &C_1\vep + C_1\vep \int_2^\sbar \lambda^{-3/2} V_{t,x}(\lambda) \, d\lambda.
\endaligned
$$
Then by the sup-norm estimate \eqref{Linfty KG ineq a} (with $0\leq r/t\leq 3/5$),
$$
\aligned
|\newperp \del^I v(t,x)|\lesssim
 \, & C_0\vep t^{-3/2}\Big(1+\int_2^s|h_{t,x}'(\sbar)e^{C\int_{\sbar}^s|h_{t,x}'(\theta)d\theta|}| \, d\sbar\Big)
\\
&+ t^{-3/2}\Big(F(s)+\int_2^sF(\sbar)|h_{t,x}'(\sbar)|e^{C\int_\sbar^s\theta|h_{t,x}'(\theta)|d\theta} \, d\sbar\Big).
\endaligned
$$
Then similar to the former case, we get (recall \eqref{ineq lem h_{t,x}'})
$$
|(\lambda t/s)^{3/2}\newperp \del^I v(t,x)|\lesssim
(C_0+C_1)\vep 
+ C_1\vep \int_2^\lambda \sbar^{-3/2}V_{t,x}(\sbar)d\sbar
\lesssim C_1\vep + C_1\vep \int_2^\lambda\sbar^{-3/2}V_{t,x}(\sbar)d\sbar,
$$
provided by $C_1\geq C_0$, which implies
$$
V_{t,x}(\lambda)\lesssim C_1\vep + C_1\vep \int_2^\lambda \sbar^{-3/2}V_{t,x}(\sbar)d\sbar
$$
Then, Gronwall lemma implies
$V_{t,x}(\lambda)\lesssim C_1\vep$
and, therefore,
$$
|\newperp \del^I v(t,x)|\lesssim C_1\vep t^{-3/2}\lesssim C_1 (s/t)^{3/2-4\delta}t^{-3/2}.
$$
And again, as in the former case, we see that
$
|\del^I v(t,x)|\lesssim C_1 (s/t)^{1/2-4\delta}t^{-3/2}.
$
 
When $|I|-1<N-7$, we see that in this case
$$
\big|[H^{\alpha\beta}u\del_\alpha\del_\beta,\del^I]v \big| \lesssim 
(C_1\vep)^2 (s/t)^{3/2}s^{-3+4\delta}.
$$
A direct application of the sup-norm estimate \eqref{Linfty KG ineq a} combined with \eqref{lem refine2' R} will give the estimate on $\del^I v$ and $\del_{\alpha}\del^I v$.
Finally, combining these two cases, we see that the desired estimates are established.
\end{proof}


\subsection{Third improvement on the wave and Klein-Gordon components}

We now establish \eqref{decay refined kde}, by combining the sup-norm estimate for the Klein-Gordon equation \eqref{Linfty KG ineq a} and the sup-norm estimate for the wave equation \eqref{Linfini wave ineq}, together with an additional bootstrap argument.

\begin{proposition}[Third improvement on the wave and Klein-Gordon components]
\label{prop refine4 kde} 
There exist constants $C,\vep_2>0$ (depending only on $N \geq 8$ and the structure of the model system \eqref{eq main}) such that if the bootstrap assumption \eqref{ineq energy assumption} holds for $\vep\leq \vep_2$ and $C_1\vep\leq 1$, then the following estimates also hold for all $s\in[2,s_1]$ and $|I|+|J|\leq N-4$, $|J| = k$:
\begin{subequations}\label{ineq prop refine4 kde}
\bel{ineq prop refine4 kde a}
\sup_{\Hcal_s}\big(t|L^J u|\big)\lesssim C_1\vep s^{k\delta},
\ee
\bel{ineq prop refine4 kde b}
\sup_{\Hcal_s}\big((s/t)^{-3+7\delta}s^{3/2}|\newperp \del^IL^J v|\big) + \sup_{\Hcal_s}\big((s/t)^{-2+7\delta}s^{3/2}|\del^IL^J v| \big)\lesssim C_1\vep s^{k\delta},
\ee
\bel{ineq prop refine4 kde c}
\sup_{\Hcal_s}\big((s/t)^{-1+7\delta}s^{3/2}|\del_\alpha \del^IL^J v|\big)\lesssim C_1\vep s^{k\delta}.
\ee
\end{subequations}
\end{proposition}

Furthermore, we see that by the commutator estimates in Proposition \ref{lem commutator esti I}, the following refined decay estimates are a direct consequence of \eqref{ineq prop refine4 kde c}:
\bel{ineq prop refine4 kde c 1}
\sup_{\Hcal_s}\big((s/t)^{1/2+7\delta}t^{3/2}|\del^IL^J \del_\alpha v|\big)\lesssim 
C_1 \vep s^{k\delta},\qquad |I|+|J|\leq N-4, |J|=k,
\ee
\bel{ineq prop refine4 kde c 2}
\sup_{\Hcal_s}\big((s/t)^{1/2+7\delta}t^{3/2}|\del^IL^J \del_\alpha\del_\beta v|\big)\lesssim
 C_1\vep s^{k\delta},\qquad |I|+|J|\leq N-5,|J|=k,
\ee
\bel{ineq prop refine4 kde c 3}
\sup_{\Hcal_s}\big((s/t)^{-3/2+7\delta}t^{3/2}|\del^IL^J \del_\alpha v|\big)\lesssim 
C_1 \vep s^{k\delta},\qquad |I|+|J|\leq N-5, |J|=k,
\ee
\bel{ineq prop refine4 kde c 4}
\sup_{\Hcal_s}\big((s/t)^{-3/2+7\delta}t^{3/2}|\del^IL^J \del_\alpha\del_\beta v|\big)\lesssim
C_1  \vep s^{k\delta},\qquad |I|+|J|\leq N-6,|J|=k.
\ee

\begin{proof} We proceed by induction on $|J|$ and introduce the notation
$$
\aligned
&V_{k,0}(\lambda) : = \sup_{2\leq s \leq \lambda,|J|\leq k\atop |I|+|J|\leq N-4}
\sup_{\Hcal_s}\big((s/t)^{-2+7\delta}s^{3/2}|\del^IL^Jv|\big),
\\
&V_{k,1}(\lambda) : = \sup_{2\leq s \leq \lambda,|J|\leq k\atop |I|+|J|\leq N-4}
\sup_{\Hcal_s}\big((s/t)^{-3+7\delta}s^{3/2}|\newperp \del^IL^Jv|\big),
\endaligned
$$
and, with $|J| \leq k$, 
$
U_k(\lambda) : =\sup_{2\leq s \leq \lambda\atop |J|\leq k}\sup_{\Hcal_s}\big(t|L^J u| \big).
$
To begin with, we observe that by \eqref{ineq prop refine3 W0dg} and \eqref{ineq lem refine2 KG-W-0 b}, there exists a positive constant $C$ determined by the structure of the system \eqref{eq main} such that
$$
V_{0,0}(\lambda)\lesssim C_1\vep, \qquad V_{0,1}(\lambda)\lesssim C_1\vep, \qquad U_0(\lambda)\lesssim C_1\vep,
$$
That is, \eqref{ineq prop refine4 kde} is proved in the case where $k=0$.

Then we suppose that for all $0\leq j\leq k-1\leq N-5$, there exists a (sufficient large) constant $C_{k-1}$ depending only on the structure of the system \eqref{eq main} and a positive constant $\vep_{k-1}'$ such that for all $\vep \leq \vep_{k-1}'$,
\bel{pr0 porp refine4 kdg}
V_{j,0}(s)\leq C_{k-1}C_1\vep s^{j\delta},\qquad V_{j,1}(s)\leq C_{k-1}C_1\vep s^{j\delta},\qquad U_j\leq C_{k-1}C_1\vep s^{j\delta}
\ee
hold on $[2,s_1]$ with $C_{k-1}$ depending only on $k$ and the structure of \eqref{eq main}. Then we will prove that there exists a pair of positive constant $(C_k,\vep'_k)$ depending only on $N$ and the structure of the model system \eqref{eq main} such that if \eqref{ineq energy assumption} holds with $\vep\leq \vep_k'$ and $C_1\vep\leq 1$, then
\bel{pr0.5 porp refine4 kdg}
V_{k,0}(s)\leq  C_kC_1\vep s^{k\delta},\qquad V_{k,1}(s)\leq  C_kC_1\vep s^{k\delta},\qquad  U_k \leq C_kC_1\vep s^{k\delta}.
\ee

We rely on a bootstrap argument. First, we observe that on the initial hyperboloid $\Hcal_2$, there exists a positive constant $C_0>0$ such that
\be
\aligned
\max_{|I|+|J|\leq N-4 \atop |J|\leq k}\sup_{\Hcal_2}\big((2/t)^{-1/2+7\delta}t^{3/2}|\del^IL^J v|\big)
& \leq C_{0,k}C_1\vep,
\\
\max_{|I|+|J|\leq N-4 \atop |J|\leq k}\sup_{\Hcal_2}\big((2/t)^{-3/2+7\delta}t^{3/2}|\newperp \del^IL^J v|\big)
& \leq C_{0,k}C_1\vep,
\\
\max_{|J|\leq k}\sup_{\Hcal_2}\big(t|L^J u|\big)
& \leq C_{0,k}C_1\vep.
\endaligned
\ee
We choose $C_k>C_{0,k}$ and set
$$
\aligned
s_{2,k} := \sup\Big\{s\in[2,s_1]\,  \ & \,V_{k,0}(s)\leq C_{k}C_1\vep s^{k\delta},
\\
&
V_{k,1}(s)\leq C_{k}C_1\vep s^{k\delta}, \qquad U_k\leq C_{k}C_1\vep s^{k\delta} \Big\}.
\endaligned
$$
By continuity, we have $s_{2,k}>2$. We will prove that for all sufficiently large constant $C_k\geq \max\{C_{0,k},C_{k-1},1\}$ the following bounds hold on $[2,s_{2,k}]$: 
\bel{pr2 prop refine4 kdg}
V_{k,0}(s)\leq \frac{1}{2}C_kC_1\vep s^{k\delta},\qquad
V_{k,1}(s)\leq \frac{1}{2}C_kC_1\vep s^{k\delta},\qquad
U_k\leq \frac{1}{2}C_kC_1\vep s^{k\delta}
\ee
for sufficiently small $\vep$. Once this is proven, we conclude that $s_{2,k} = s_1$. Namely, proceeding by contradiction, we see that
in the opposite case at $s_{2,k}<s_1$, at least one of the following conditions must hold:
$$
V_{k,0}(s)= C_{k}C_1\vep s^{k\delta},\qquad
V_{k,1}(s)= C_{k}C_1\vep s^{k\delta},\qquad
U_k= C_{k}C_1\vep s^{k\delta},
$$
which contradicts the improved estimates \eqref{pr2 prop refine4 kdg}.

It remains to establish \eqref{pr2 prop refine4 kdg} and we derive first the following estimate for $|I|+|J|\leq N-4$, $|J| = j \leq k$
(again provided $2 \leq s \leq s_{2,k}$) 
\begin{subequations}\label{pr3 prop refine4 kdg}
\bel{pr3 prop refine4 kdg a}
\big|\del^IL^J v\big|\lesssim C_{k}C_1\vep (s/t)^{2-7\delta}s^{-3/2+j\delta},
\ee
\bel{pr3 prop refine4 kdg b}
\big|\del_\alpha\del^IL^J v\big|\lesssim C_{k}C_1\vep (s/t)^{1-7\delta}s^{-3/2+j\delta},
\ee
\bel{pr3 prop refine4 kdg c}
\big|L^J u\big|\lesssim C_{k}C_1\vep t^{-1}s^{j\delta}.
\ee
\end{subequations}
The derivation of
\eqref{pr3 prop refine4 kdg a} is direct from the decay assumption \eqref{pr0 porp refine4 kdg} 
and the induction assumption \eqref{pr0 porp refine4 kdg}, while \eqref{pr3 prop refine4 kdg b} follows directly from \eqref{ineq Ur-partial}, the decay assumption \eqref{pr0.5 porp refine4 kdg} or the induction assumption \eqref{pr0 porp refine4 kdg}:
$$
\aligned
\big|\del_\alpha\del^IL^J v\big|\lesssim
 \, & (s/t)^{-2}\big|\newperp \del^IL^J v\big| + (s/t)^{-2}\sum_a \big|\delu_a\del^IL^Jv\big|
\\
\lesssim &C_kC_1\vep (s/t)^{-2}(s/t)^{3-7\delta}s^{-3/2+j\delta} + C_1\vep (s/t)^{-2}t^{-5/2}s^{1/2+(j+3)\delta}
\\
\lesssim &C_kC_1\vep (s/t)^{1-7\delta}s^{-3/2+j\delta} + C_1\vep (s/t)^{1/2}s^{-2+(j+3)\delta}
\\
\lesssim &C_kC_1\vep (s/t)^{1-7\delta}s^{-3/2+j\delta},
\endaligned
$$
where the first equation in \eqref{decay basic 2ge2} was used. 
On the other hand,
\eqref{pr3 prop refine4 kdg c} is also direct from \eqref{pr0 porp refine4 kdg} and \eqref{pr0.5 porp refine4 kdg}.

Then we need the following two estimates for $|I|+|J|\leq N-4$, $|J| = k$:
\bel{pr4 prop refine4 kdg}
\aligned
\big|\del^IL^J\big(\del_\alpha v\del_\beta v\big)\big| + \big|\del^IL^J\big(v^2\big)\big|
\lesssim (C_kC_1\vep)^2t^{-2-(1/2 - 7\delta + k\delta/2)}(t-r)^{-1+(1/2 - 7\delta + k\delta/2)},
\endaligned
\ee
\bel{pr5 porp refine4 kdg}
\aligned
\big|[H^{\alpha\beta}u\del_\alpha\del_\beta,\del^IL^J]v\big|
\lesssim &(C_kC_1\vep)^2 (s/t)^{2-7\delta}s^{-5/2+k\delta}.
\endaligned
\ee
The estimate \eqref{pr4 prop refine4 kdg} follows directly from \eqref{pr3 prop refine4 kdg}. We see that
$$
\aligned
\big|\del^IL^J(v^2)\big|\leq \, & \sum_{I_1+I_2=I\atop J_1+J_2=J}\big|\del^{I_1}\del^{J_1}v\del^{I_2}L^{J_2}v\big|
\\
\lesssim \, & C_kC_1\vep (s/t)^{2-7\delta}s^{-3/2+|J_1|\delta} C_kC_1\vep(s/t)^{2-7\delta}s^{-3/2+|J_2|\delta}
\\
\simeq& (C_kC_1\vep)^2(s-r)^{-1 +(1/2-7\delta+k\delta/2)} t^{-2 - (1/2-7\delta -k\delta/2)} 
\endaligned
$$
and
$$
\aligned
\big|\del^IL^J\big(\del_\alpha v\del_\beta v\big)\big|
\leq \, & \sum_{I_1+I_2=I\atop J_1+J_2=J}\big|\del^{I_1}L^{J_1}\del_\alpha v \del^{I_2}L^{J_2}\del_\beta v\big|
\lesssim (C_kC_1\vep)^2(s/t)^{2-14\delta}s^{-3+k\delta}
\\
\simeq &(C_kC_1\vep)^2t^{-2-(1/2-7\delta-k\delta/2)}(t-r)^{-1+(1/2-7\delta + k\delta/2)}.
\endaligned
$$

The estimate of \eqref{pr5 porp refine4 kdg} is also direct by substituting \eqref{pr3 prop refine4 kdg}. We recall \eqref{ineq lem refine2 commu} an write 
$$
\aligned
&\big|[H^{\alpha\beta}u\del_\alpha\del_\beta,\del^IL^J]v\big|
\\
&\lesssim C_1\vep(s/t)^{-2}t^{-1}  C_k\vep (s/t)^{3-7\delta}s^{-3/2+k\delta}
 + (C_k\vep)^2 \sum_{|J_1|+|J_2|\leq |J|}t^{-1}s^{|J_1|\delta} (s/t)^{1-7\delta}s^{-3/2+|J_2|\delta}
\\
&\quad +C_1\vep t^{-3/2}(s/t)^{-3}  C_k\vep (s/t)^{3-7\delta}s^{-3/2+k\delta} + (C_1\vep)^2 (s/t)^{3/2}s^{-3+(k+4)\delta}
\\
&\lesssim (C_kC_1\vep)^2(s/t)^{2-7\delta}s^{-5/2+k\delta},
\endaligned
$$
where we have assumed that $C_k\geq C_1$.

Now we substitute \eqref{pr4 prop refine4 kdg} into \eqref{Linfini wave ineq} and find that (similar to the proof of Proposition \ref{lem refine1 W})  
\bel{pr6 porp refine4 kdg}
\big|\del^IL^J u\big|\lesssim C_{0,k}C_1\vep t^{-3/2} + (C_kC_1\vep)^2(s/t)^{k\delta} t^{-1}s^{k\delta}, 
\ee
which is
\bel{pr6 porp refine4 kdg'}
U_k(s) \lesssim C_{0,k}C_1\vep + (C_kC_1\vep)^2 s^{k\delta}.
\ee
On the other hand, the estimate on $|\del^IL^J v|$ and $\big|\del_\alpha\del^IL^J v\big|$ is a bit more difficult. We see that
$$
\aligned
F(\sbar)\leq \int_{s_0}^\sbar \sum_{i=1}^3 R_i[\del^IL^Jv](\lambda t/s,\lambda x/s) \, d\lambda
+ \int_{s_0}^\sbar \lambda^{3/2}\big|[H^{\alpha\beta}u\del_\alpha\del_\beta,\del^IL^J]v\big|(\lambda t/s, \lambda x/s) \, d\lambda.
\endaligned
$$

By the sup-norm estimate \eqref{Linfty KG ineq a}, in the region $\Kcal \cap \{3/5<r/t<1\}$, recall that $s_0\geq C t/s$, we can calculate each term in the right-hand side of the above inequality with that aid of \eqref{pr5 porp refine4 kdg} and \eqref{ineq refine2 R} and find that
$$
\aligned
\big|F(s)\big|\lesssim
 \, & C_1\vep (s/t)^{2-7\delta-k\delta} + (C_kC_1\vep)^2(s/t)^{2-7\delta}s^{k\delta}
\\
\lesssim \, & C_1\vep(s/t)^{2-7\delta}s^{k\delta} + (C_kC_1\vep)^2(s/t)^{2-7\delta}s^{k\delta}
\simeq
\big(C_1\vep + (C_kC_1\vep)^2\big)(s/t)^{2-7\delta}s^{k\delta}.
\endaligned
$$
Then, we apply the sup-norm estimate \eqref{Linfty KG ineq a} with \eqref{ineq lem h_{t,x}'} and by the same procedure in the proof of Proposition \eqref{lem refine2 KG-W}, we conclude that when $3/5<r/t<1$,
\begin{subequations}\label{pr7 porp refine4 kdg}
\bel{pr7 porp refine4 kdg a}
|\newperp \del^IL^J v(t,x)|\lesssim 
\big(C_1\vep + (C_kC_1\vep)^2\big)(s/t)^{3-7\delta}s^{-3/2 + k\delta},
\ee
\bel{pr7 porp refine4 kdg b}
|\del^IL^J v(t,x)|\lesssim 
\big(C_1\vep + (C_kC_1\vep)^2\big)(s/t)^{2-7\delta}s^{-3/2 + k\delta}.
\ee
\end{subequations}

When $0\leq r/t\leq 3/5$, we see that $4/5\leq s/t\leq 1$, then
$$
F(s)\lesssim 
\big(C_1\vep+(C_kC_1\vep)^2\big)t^{-3/2}s^{k\delta}.
$$
Then, also by the sup-norm estimate \eqref{Linfty KG ineq a} and \eqref{ineq lem h_{t,x}'}, we find that
\begin{subequations}\label{pr8 porp refine4 kdg}
\bel{pr8 porp refine4 kdg a}
|\newperp \del^IL^Jv|\lesssim 
(C_{0,k}+1) C_1\vep (s/t)^{3-7\delta}s^{-3/2 + k\delta} + (C_kC_1\vep)^2 (s/t)^{3-7\delta}s^{-3/2 + k\delta},
\ee
\bel{pr8 porp refine4 kdg b}
|\del^IL^Jv|\lesssim 
(C_{0,k}+1)C_1\vep (s/t)^{2-7\delta}s^{-3/2+ k\delta} + (C_kC_1\vep)^2 (s/t)^{2-7\delta}s^{-3/2+k\delta},
\ee
where we recall that $C_1>C_0$
\end{subequations}


Then we conclude that there exists a positive constant $\Cbar$ determined only by the structure of the system \eqref{eq main} such that
\begin{subequations}\label{pr8 porp refine4 kdg'}
\bel{pr8 porp refine4 kdg a'}
V_{k,1}(s)\leq \Cbar (C_{0,k}+1)C_1\vep s^{k\delta} + \Cbar(C_kC_1\vep)^2s^{k\delta},
\ee
\bel{pr8 porp refine4 kdg b'}
V_{k,0}(s)\leq \Cbar(C_{0,k}+1)C_1\vep s^{k\delta} + \Cbar(C_kC_1\vep)^2s^{k\delta}.
\ee
\end{subequations}
Now we consider together \eqref{pr6 porp refine4 kdg'} and \eqref{pr8 porp refine4 kdg'} and see that if $C_k>2C(C_{0,k}+1)$, then we can take 
$
\vep_k' := \frac{C_k-2\Cbar(C_{0,k}+1)}{2\Cbar C_k^2C_1}.
$
Then we find that
$$
V_{0,k}(s)\leq \frac{1}{2}C_kC_1 \vep s^{k\delta}, \qquad V_{1,k}(s)\leq \frac{1}{2}C_kC_1 \vep s^{k\delta},\qquad  U_k(s)\leq \frac{1}{2}C_kC_1 \vep s^{k\delta},
$$
for all $\vep \leq \vep_k'$. This conclude that $s_{2,k} = s_2$ so the case $|J|=k$ is proven. Then by induction we see that for all $k\leq N-4$, \eqref{pr0 porp refine4 kdg} is established.  Then taking $\vep_2 := \min_{k\leq N-4}\{\vep_k'\}$ and $C_\infty: = \max_{k\leq N-4}{C_k}$, we see that \eqref{ineq prop refine4 kde a} and \eqref{ineq prop refine4 kde b} are established for all $k\leq N-4$ and, more precisely, 
\begin{subequations}
\label{ineq prop refine4 kde M}
\bel{ineq prop refine4 kde a M}
\sup_{\Hcal_s}\big(t|L^J u|\big)\leq C_\infty C_1\vep s^{k\delta},
\ee
\bel{ineq prop refine4 kde b M}
\sup_{\Hcal_s}\big((s/t)^{-3+7\delta}s^{3/2}|\newperp \del^IL^J v|\big) + \sup_{\Hcal_s}\big((s/t)^{-2+7\delta}s^{3/2}|\del^IL^J v| \big)\leq C_\infty C_1\vep s^{k\delta}.
\ee
\bel{ineq prop refine4 kde c M}
\sup_{\Hcal_s}\big((s/t)^{-1+7\delta}s^{3/2}|\del_\alpha \del^IL^J v|\big)\leq C_\infty C_1\vep s^{k\delta}.
\ee
\end{subequations} 
From its definition, we see that $C_\infty$ is determined only from the structure of the system and therefore, we have proven \eqref{ineq prop refine4 kde}. 
\end{proof}


\section{Refined energy estimate and completion of the bootstrap argument}

\subsection{Overview}

In this section, we derive the improved energy estimates \eqref{ineq energy assumption'} which concludes the main result. The improved estimates are classified in two categories. The first refers to the energy estimates of order higher than or equal to $N-3$, the second refers to those of order lower that or equal to $N-4$.

First, we apply $\del^IL^J$ (with $|I|+|J|\leq N$) to our system of equations
\be
-\Box \del^IL^J u = P^{\alpha\beta}\del^IL^J\big( \del_\alpha v\del_\beta v\big) + R\del^IL^J(v^2),
\ee
\be
-\Box \del^IL^J v + c^2\del^IL^J v + H^{\alpha\beta}u\del_\alpha\del_\beta \del^IL^Jv = [H^{\alpha\beta}u\del_\alpha\del_\beta,\del^IL^J]v.
\ee
To be able to apply the energy estimate (Proposition \ref{prop energy}), we need first to check
 \eqref{ineq coersive 1} and \eqref{ineq coersive 1'}.

\begin{lemma}
\label{lem energy coef}
There exists a positive constant $\vep_0$ such that if the energy assumption \eqref{ineq energy assumption} is valid with $C_1\vep\leq 1$ and $\vep\leq \vep_0$, then the following estimates hold:
\bel{ineq lem energy coef 1}
\frac{1}{2}E_{m,c}\leq E_{g,c}\leq 2E_{m,c},
\ee
\bel{ineq lem energy coef 2}
\aligned
& \int_{\Hcal_s}(s/t)|\del_\alpha h^{\alpha\beta}\del_t\del^IL^Jv \del_\beta\del^IL^J v|dx
 + \int_{\Hcal_s}(s/t)|\del_th^{\alpha\beta}\del_\alpha\del^IL^Jv\del_\beta\del^IL^J v|dx
\\
& \lesssim M(s) E(s,\del^IL^J v)^{1/2}
\endaligned
\ee
with
$$
M(s)\leq\left\{
\aligned
& C_1\vep s^{-1/2+k\delta},\quad N-3\leq |I|+|J|\leq N,
\\
& C_1\vep s^{-1+k\delta},\quad |I|+|J|\leq N-4.
\endaligned
\right.
$$
\end{lemma}

\begin{proof}
The proof of \eqref{ineq lem energy coef 1} follows directly from \eqref{ineq lem refine2 KG-W-0 b}. We remark that
$$
|h^{\alpha\beta}| = |H^{\alpha\beta}u|\lesssim C_1\vep t^{-1}\lesssim C_1\vep (s/t)^2
$$
where we have observed that $t^{1/2}\leq s \leq t$  in $\Kcal$. We get
$$
\int_{\Hcal_s}|h^{\alpha\beta}\del_t\del^IL^Jv\del_\beta\del^IL^Jv|dx
\lesssim C_1\vep\int_{\Hcal_s}|(s/t)^2\del_\alpha\del^IL^Jv\del_\beta\del^IL^Jv|dx
\lesssim C_1\vep E_{g,c}(s,\del^IL^Jv),
$$
$$
\int_{\Hcal_s}|h^{\alpha\beta}\del_\alpha\del^IL^J v\del_\beta\del^IL^J v|\lesssim C_1\vep\int_{\Hcal_s}|(s/t)^2\del_\alpha\del^IL^Jv\del_\beta\del^IL^Jv|dx
\lesssim C_1\vep E_{g,c}(s,\del^IL^Jv),
$$
where we have used
$
\int_{\Hcal_s}|(s/t)\del_\alpha\del^IL^Jv|^2dx\leq E_{g,c}(s,\del^IL^J v).
$

So  for some $C'>0$ we have
$$
\big|E_{g,c}(s,\del^IL^J v) - E_{m,c}(s,\del^IL^Jv)\big|\leq  C' C_1\vep E_{g,c}(s,\del^IL^Jv),
$$
and we choose $\vep_0\leq \frac{1}{2C' C_1}$. Then, for $\vep\leq \vep_0$, it holds
$$
\big|E_{g,c}(s,\del^IL^J v) - E_{m,c}(s,\del^IL^Jv)\big|\lesssim 
C_1\vep E_{g,c}(s,\del^IL^Jv) \leq \frac{1}{2}E_{g,c}(s,\del^IL^Jv),
$$
which yields \eqref{ineq lem energy coef 1}.

To derive \eqref{ineq lem energy coef 2}, we just need to observe that
$$
\aligned
&\int_{\Hcal_s}\big|\del_{\gamma}h^{\alpha\beta}\del_\alpha\del^IL^Jv\big|^2dx
\\
&\lesssim C_1\vep \int_{\Hcal_s}t^{-1}s^{-2}(t/s)^2 \,  \big|(s/t)\del_\alpha\del^IL^J v \big|^2dx
\simeq C_1\vep \int_{\Hcal_s}ts^{-4}\big|(s/t)\del_\alpha\del^IL^J v \big|^2dx
\\
& \lesssim C_1\vep s^{-2}E_{g,c}(s,\del^IL^Jv),
\endaligned
$$
and we use the first estimate in \eqref{decay basic 3ge1RR}: 
$$
\int_{\Hcal_s}\big|\del_{\gamma}h^{\alpha\beta}\del_\alpha\del^IL^Jv\big|^2dx\leq
\left\{
\aligned
& C(C_1\vep)^2 s^{-1+2k\delta},\quad N-3\leq |I|+|J|\leq N,
\\
& C(C_1\vep)^2 s^{-2+2k\delta},\quad |I|+|J|\leq N-4.
\endaligned
\right.
$$
So we see that
$$
\int_{\Hcal_s}(s/t)|\del_\alpha h^{\alpha\beta}\del_t\del^IL^Jv \del_\beta\del^IL^J v|dx\leq \|\del_\alpha h^{\alpha\beta}\del_t\del^IL^Jv\|_{L_f^2(\Hcal_s)}\|(s/t) \del_\beta\del^IL^J v\|_{L_f^2(\Hcal_s)}, 
$$
which is bounded by the right-hand side of \eqref{ineq lem energy coef 2}. The other term in the left-hand side is bounded in the same manner and we thus omit the details.
\end{proof}


\subsection{Lower-order $L^2$ estimates}

We remark that in lower order case where $|I|+|J|\leq N-4$, we have $M(s)\lesssim C_1\vep s^{-1+k\delta}$ and we need again the estimate on the source term $\del^IL^J\big(P^{\alpha\beta}\del_\alpha v\del_\beta v + Rv^2\big)$ and $[H^{\alpha\beta}u\del_\alpha\del_\beta,\del^IL^J]v$.

\begin{lemma}
\label{lem energy lower}
Under the assumption of \eqref{ineq energy assumption}, the following estimates hold for $|I|+|J|\leq N-4$ with $|J|=k$:
\bel{ineq lem energy lower 1}
\big\|\del^IL^J\big(P^{\alpha\beta}\del_\alpha v\del_\beta v\big)\big\|_{L_f^2(\Hcal_s)}
+ \big\|R\del^IL^Jv^2\big\|_{L_f^2(\Hcal_s)}\lesssim 
(C_1\vep)^2 s^{-3/2+k\delta},
\ee
\bel{ineq lem energy lower 2}
\|[H^{\alpha\beta}u\del_\alpha\del_\beta,\del^IL^J]v\|_{L_f^2(\Hcal_s)}\lesssim
 (C_1\vep)^2s^{-1+k\delta}. 
\ee 
\end{lemma}

\begin{proof}
The estimates of these terms relies on the basic $L^2$ and refined sup-norm estimates.
We remark that
$$
\aligned
\big\|\del^IL^J\big(\del_\alpha v\del_\beta v\big)\big\|_{L_f^2(\Hcal_s)}
\leq & \sum_{I_1+I_2=I\atop J_1+J_2=J}\big\|\del^{I_1}L^{J_1}\del_\alpha v\del^{I_2}L^{J_2}\del_\beta v\big\|_{L_f^2(\Hcal_s)}
\\
\leq & \sum_{1\leq |I_1|+J_1|\leq N-4\atop |I_2|+|J_2|\leq N-5}
\big\|\del^{I_1}L^{J_1}\del_\alpha v\big\|_{L^\infty(\Hcal_s)}\big\|\del^{I_2}L^{J_2}\del_\beta v\big\|_{L_f^2(\Hcal_s)}
\\
&+
\big\|(t/s)\del_\alpha v\big\|_{L^\infty(\Hcal_s)}\big\|(s/t)\del^{I}L^{J}\del_\beta v\big\|_{L_f^2(\Hcal_s)}
=:T_1+T_2.
\endaligned
$$
For $T_2$, we apply \eqref{ineq prop refine3 W0dg b} with $|I| = 1$ and \eqref{L2 basic 3ge2}
and we conclude that
$$
T_2\lesssim C_1\vep (s/t)^{-1/2-7\delta}t^{-3/2} C_1\vep s^{|J|\delta}\vep 
\lesssim (C_1\vep)^2s^{-3/2+k\delta}.
$$
For $T_1$, we apply \eqref{ineq prop refine4 kde c 1} and \eqref{L2 basic 4ge2}
and we conclude that
$$
T_1\lesssim 
C_1\vep (s/t)^{-1/2-7\delta}t^{-3/2}s^{|J_1|\delta} C_1\vep s^{|J_2|\delta}\vep \lesssim
 (C_1\vep)^2 s^{-3/2+k\delta}.
$$

The estimate on the term $\del^IL^J\big(v^2\big)$ is similar by apply \eqref{ineq prop refine4 kde b} \eqref{L2 basic 3ge2}
and we omit the details.

To see the estimate on $[H^{\alpha\beta}u\del_\alpha\del_\beta,\del^IL^J]v$ is quite similar, we just need to remark that it is a linear combination of the following terms:
$$
L^{J'_1}u\del^IL^{J'_2}\del_\alpha\del_\beta v,\quad
\del^{I_1}L^{J_1}u\del^{I_2}L^{J_2}\del_\alpha\del_\beta v,\quad
u\del_\alpha\del_\beta\del^IL^{J''_2}v
$$
where $I_1+I_2=I, J_1+J_2=J, J_1'+J_2'=J$ with $|J_1'| \geq 1$, $|I_1|\geq 1$ and $|J''_2|\leq |J|-1$. For the last term we apply \eqref{ineq lem refine2 KG-W-0 b} and \eqref{L2 basic 3ge1}:
$$
\aligned
\big\|u\del_\alpha\del_\beta\del^IL^{J''_2}v\big\|_{L_f^2(\Hcal_s)}
\leq \, & \big\|(t/s)u\big\|_{L^\infty(\Hcal_s)}\big\|(s/t)\del_\alpha\del_\beta\del^IL^{J''_2}v\big\|_{L_f^2(\Hcal_s)}
\\
\lesssim \, & C_1\vep s^{-1}  C_1\vep s^{k\delta}\lesssim
 (C_1\vep)^2 s^{-1+k\delta}.
\endaligned
$$
For the first term, we see that $|J_1'|\leq N-4$, then we apply \eqref{ineq prop refine4 kde a} and \eqref{L2 basic 3ge1}:
$$
\aligned
\big\|L^{J'_1}u\del^IL^{J'_2}v\big\|_{L_f^2(\Hcal_s)}
\leq \, & \big\|(t/s)L^{J'_1}u\big\|_{L^\infty(\Hcal_s)} \big\|(s/t)\del^IL^{J'_2}v\big\|_{L_f^2(\Hcal_s)}
\\
\lesssim \, & C_1\vep s^{-1+|J'_1|\delta} C_1\vep s^{|J_2'|\delta}\lesssim
 (C_1\vep)^2s^{-1+k\delta}.
\endaligned
$$
For the second term, we see that when $|I_1|=1$ and $J_1=0$, 
$$
\aligned
\|\del_{\gamma}u\del^{I_2}L^J\del_\alpha\del_\beta v\|_{L_f^2(\Hcal_s)}
\leq& \|(t/s) \del_\gamma u\|_{L^\infty(\Hcal_s)}\, \|(s/t)\del^IL^J\del_\alpha\del_\beta v\|_{L_f^2(\Hcal_s)}
\\
\lesssim & C_1\vep \|(t/s)t^{-1/2}s^{-1}\|_{L^\infty(\Hcal_s)}\, C_1\vep s^{k\delta} 
\simeq (C_1\vep)^2s^{-1+k\delta}.
\endaligned
$$
When $|I_1|+|J_1|\geq 2$, we see that $|I_2|+|J_2|\leq N-6$. Then by the first inequality in \eqref{decay basic 1ge1R}
and \eqref{L2 basic 4ge2}, we find  
$$
\aligned
\big\|\del^{I_1}L^{J_1}u\del^{I_2}L^{J_2}\del_\alpha\del_\beta v\big\|_{L_f^2(\Hcal_s)}
\leq &\big\|\del^{I_1}L^{J_1}u\|_{L^\infty(\Hcal_s)}\big\|\del^{I_2}L^{J_2}\del_\alpha\del_\beta v\big\|_{L_f^2(\Hcal_s)}
\\
\lesssim & C_1\vep s^{-3/2+(|J_1|+2)\delta}  C_1\vep s^{|J_2|\delta}\lesssim
(C_1\vep)^2s^{-3/2+(k+2)\delta}
\endaligned
$$
and we conclude with \eqref{ineq lem energy lower 2}.
\end{proof}


\subsection{Higher-order $L^2$ estimates}

When $N-3\leq |I|+|J|\leq N$, the energy estimate is more complicated. For the source terms we have the following estimates.

\begin{lemma}
\label{lem energy higher}
Under the energy assumption \eqref{ineq energy assumption} the following estimates hold for $N-4\leq |I|+|J|\leq N$ and $|J| = k$:
\bel{ineq lem energy higher 1}
\big\|\del^IL^J\big(P^{\alpha\beta}\del_\alpha v\del_\beta v\big)\big\|_{L_f^2(\Hcal_s)}
+ \big\|\del^IL^J\big(Rv^2\big)\big\|_{L_f^2(\Hcal_s)}
\lesssim
 (C_1\vep)^2 s^{-1+k\delta},
\ee
\bel{ineq lem energy higher 2}
\big\|[H^{\alpha\beta}u\del_\alpha\del_\beta,\del^IL^J]v\big\|_{L_f^2(\Hcal_s)}\lesssim
 (C_1\vep)^2s^{-1/2+k\delta}.
\ee
\end{lemma}

\begin{proof}
The proof relies on the refined decay estimate \eqref{ineq prop refine4 kde} and the basic $L^2$ estimates. We begin with \eqref{ineq lem energy higher 1}. We remark that $\del^IL^J\big(\del_\alpha v\del_\beta v\big)$ is a linear combination of the following terms
$$
\del^{I_1}L^{J_1}\del_\alpha v \del^{I_2}L^{J_2}\del_\beta v
$$
with $I_1+I_2=I$, $J_1+J_2=J$. We see that when $|I_1|+|J_1|=0$, we apply \eqref{ineq prop refine3 W0dg b} on $\del_\alpha v$ ( with $1\leq N-4$) and \eqref{L2 basic 1ge2}
$$
\aligned
\big\|\del^{I_1}L^{J_1}\del_\alpha v \del^{I_2}L^{J_2}\del_\beta v\big\|_{L_f^2(\Hcal_s)}
&= \big\|(t/s)\del_\alpha v (s/t)\del^{I}L^{J}\del_\beta v\big\|_{L_f^2(\Hcal_s)}
\\
&\lesssim C_1\vep \|(t/s) (s/t)^{1/2-7\delta}t^{-3/2}\|_{L^{\infty(\Hcal_s)}}   CC_1\vep s^{1/2+k\delta} 
\\
& \lesssim (C_1\vep)^2s^{-1+k\delta}.
\endaligned
$$
When $1\leq |I_1|+|J_1|\leq N-4$, we see that $4\leq |I_2|+|J_2|\leq N-1$. Then we apply \eqref{ineq prop refine4 kde c} and the third inequality in \eqref{L2 basic 2ge2}:
$$
\aligned
\big\|\del^{I_1}L^{J_1}\del_\alpha v \del^{I_2}L^{J_2}\del_\beta v\big\|_{L_f^2(\Hcal_s)}
\leq \, &  \big\|\del^{I_1}L^{J_1}\del_\alpha v\big\|_{L^\infty(\Hcal_s)}\big\|\del^{I_2}L^{J_2}\del_\beta v\big\|_{L_f^2(\Hcal_s)}
\\
\lesssim \, & C_kC_1\vep \|(s/t)^{-1/2-7\delta}t^{-3/2 + |J_1|\delta}\|_{L^\infty(\Hcal_s)} C_1\vep s^{1/2+|J_2|\delta}
\\
\lesssim & C_k(C_1\vep)^2s^{-1+k\delta}.
\endaligned
$$
When $N-3\leq |I_1|+|J_1|\leq N-1$, we see that $1\leq |I_2|+|J_2|\leq 3\leq N-4$.
Then we apply the third inequality in \eqref{L2 basic 2ge2}
and \eqref{ineq prop refine4 kde c}. Similar to the former case,
$$
\big\|\del^{I_1}L^{J_1}\del_\alpha v \del^{I_2}L^{J_2}\del_\beta v\big\|_{\Hcal_s}\lesssim  (C_1\vep)^2s^{-1+k\delta}.
$$
When $|I_1|+|J_1| = N$ and $|I_2|+|J_2| = 0$, the estimate is derived similarly as in the first case by exchanging the role of $\del_\alpha v$ and $\del_\beta v$. The we conclude that
$$
\big\|\del^{I_1}L^{J_1}\del_\alpha v \del^{I_2}L^{J_2}\del_\beta v\big\|_{L_f^2(\Hcal_s)}\lesssim
 (C_1\vep)^2s^{-1+k\delta}.
$$
The estimate on $\del^IL^J\big(v^2\big)$ is quite similar by applying \eqref{L2 basic 1ge2} and \eqref{ineq prop refine4 kde b},  we omit the detail.

The estimate on $[H^{\alpha\beta}u\del_\alpha\del_\beta,\del^IL^J]v$ is as follows: we observe that this term is a linear combination of the following terms
$$
L^{J_1'}u\del^IL^{J_2'}\del_\alpha\del_\beta v,\quad
\del^{I_1}L^{J_1}u\del^{I_2}L^{J_2}\del_\alpha\del_\beta v,\quad
u\del_\alpha\del_\beta\del^{I_2}L^{J''_2}v
$$
where $I_1+I_2=I$, $J_1+J_2=J$, $J_1'+J_2' = J$ with $|J_1'|\geq 1$ $|I_1|\geq 1$ and $|J_2''|\leq |J|-1$.
The last term is bounded by applying \eqref{ineq lem refine2 KG-W-0} and \eqref{L2 basic 1ge2}:
$$
\aligned
\big\|u\del_\alpha\del_\beta\del^{I_2}L^{J''_2}v\big\|_{L_f^2(\Hcal_s)}
\leq \, & \|(t/s)u\|_{L^\infty(\Hcal_s)}\big\|(s/t)\del_\alpha\del_\beta\del^{I_2}L^{J''_2}v\big\|_{L_f^2(\Hcal_s)}
\\
\lesssim \, & C_1\vep s^{-1}  C_1\vep s^{1/2+|J''_2|\delta}
\lesssim  (C_1\vep)^2s^{-1/2+k\delta}.
\endaligned
$$

For the first term, we make the following observation. When $1\leq |J_1'|\leq N-4$, we have
$4\leq |I|+|J_2|\leq N-1$. Then we apply \eqref{ineq prop refine4 kde a} and \eqref{L2 basic 2.5ge2}:
$$
\aligned
\big\|L^{J_1'}u\del^IL^{J_2'}\del_\alpha\del_\beta v\big\|_{L^(\Hcal_s)}
\leq & \big\|(t/s)L^{J_1'}u\big\|_{L^\infty(\Hcal_s)}\big\|(s/t)\del^IL^{J_2'}\del_\alpha\del_\beta v\big\|_{L^(\Hcal_s)}
\\
\lesssim \, & C_1\vep s^{-1}s^{|J_1'|\delta}  C_1\vep s^{1/2+|J_2'|\delta}
\lesssim  (C_1\vep)^2s^{-1/2+k\delta}.
\endaligned
$$
When $N-3\leq |J_1'|\leq N$, we see that $|I|+|J_2|\leq 3\leq N-5$.
Then we apply the Hardy inequality in the form \eqref{L2 basic 1ge3 a} as well as \eqref{ineq prop refine4 kde c 2}. So we see that
$$
\aligned
\big\|L^{J_1'}u\del^IL^{J_2'}\del_\alpha\del_\beta v\big\|_{L_f^2(\Hcal_s)}
\leq \, & \big\|s^{-1}L^{J_1'}u\big\|_{L_f^2(\Hcal_s)}\big\|s\del^IL^{J_2'}\del_\alpha\del_\beta v\big\|_{L^\infty(\Hcal_s)}
\\
\lesssim \, & C_1\vep s^{|J_1'|\delta}  C_1\vep s^{-1/2+|J_2'|\delta}
\lesssim  (C_1\vep)^2 s^{-1/2+k\delta}.
\endaligned
$$
The second term is easier, since the factor $\del^{I_1}L^{J_1}u$ has better decay when $|I_1|\geq 1$. Then we see that when $|I_1|=1$ and $|J_2|=0$,
$$
\aligned
\|\del^{I_1}u\del^{I_2}L^J\del_\alpha\del_\beta v\|_{L_f^2(\Hcal_s)}
\leq \, & \|(t/s)\del^{I_1}u \, (s/t)\del^{I_2}L^J\del_\alpha\del_\beta v\big\|_{L_f^2(\Hcal_s)}
\\
\lesssim \, & (C_1\vep)\|t^{1/2}s^{-2}\|_{L^\infty(\Hcal_s)} C_1\vep s^{1/2+k\delta} 
\simeq  (C_1\vep)^2s^{-1/2+k\delta} 
\endaligned
$$
when $2\leq |I_1|+|J_1|\leq N-2$, $|I_2|+|J_2|\leq N-2$. Then we apply the third inequality in \eqref{L2 basic 2ge2}
and we see that
$$
\aligned
\|\del^{I_1}L^{J_1}u\del^{I_2}L^{J_2}\del_\alpha\del_\beta v\|_{L_f^2(\Hcal_s)}
\leq \, & \|\del^{I_1}L^{J_1}u\|_{L^\infty(\Hcal_s)}\big\|\del^{I_2}L^{J_2}\del_\alpha\del_\beta v\big\|_{L_f^2(\Hcal_s)}
\\
\lesssim \, & (C_1\vep)^2s^{-1+(k+2)\delta}\leq C(C_1\vep)^2s^{-1/2+k\delta}.
\endaligned
$$
When $N-1\leq |I_1|+|J_1|\leq N$, $|I_2|+|J_2|\leq 1\leq N-7$ then we apply \eqref{L2 basic 1ge2} and \eqref{ineq prop refine4 kde c 4}. Then, we obtain
$$
\aligned
\|\del^{I_1}L^{J_1}u\del^{I_2}L^{J_2}\del_\alpha\del_\beta v\|_{L_f^2(\Hcal_s)}
\leq \, & \|(s/t)\del^{I_1}L^{J_1}u\|_{L_f^2(\Hcal_s)}\big\|(t/s)\del^{I_2}L^{J_2}\del_\alpha\del_\beta v\|_{L^\infty(\Hcal_s)}
\\
\lesssim \, & C_1\vep s^{|J_1|\delta}  C C_1\vep s^{-3/2+|J_2|\delta}
\lesssim (C_1\vep)^2s^{-1/2+k\delta},
\endaligned
$$
which completes the argument.
\end{proof}


\subsection{Proof of Proposition \ref{prop bootstrap}}

Our aim is to establish the improved energy estimate \eqref{ineq energy assumption'} and to conclude the proof of Theorem \ref{thm main}, that is, we now establish Proposition \ref{prop bootstrap}.
The strategy is to apply the energy estimate \ref{prop energy} with \eqref{ineq lem energy coef 1}, \eqref{ineq lem energy coef 2}, \eqref{ineq lem energy lower 1}, \eqref{ineq lem energy lower 2}, \eqref{ineq lem energy higher 1}, and \eqref{ineq lem energy higher 2}.
 
We need to specify the constants and we denote by $\Cbar$ a sufficiently large constant determined only by the structure of the system such that all of the above estimates 
hold true.  
We derive the wave equation of \eqref{eq main} by $\del^IL^J$:
$$
-\Box \del^IL^J u = \del^IL^J\big(P^{\alpha\beta}\del_\alpha v\del_\beta v\big) + \del^IL^J\big(v^2\big).
$$
Recall the energy estimate \eqref{ineq energy wave}
$$
E_m(s,\del^IL^J u)^{1/2}\leq E_m(2,\del^IL^J u)^{1/2} + \int_2^s \| \Box u \|_{L_f^2(\Hcal_\sbar)} \, d\sbar
$$
with
$
\| \Box u \|_{L_f^2(\Hcal_\sbar)}\leq \big\|\del^IL^J\big(P^{\alpha\beta}\del_\alpha v\del_\beta v\big)\big\|_{L_f^2(\Hcal_s)}
+ \big\|\del^IL^J\big(Rv^2\big)\big\|_{L_f^2(\Hcal_s)}.
$
Then by \eqref{ineq lem energy lower 1}, when $|I|+|J|\leq N-4$, we have 
$
\| \Box u \|_{L_f^2(\Hcal_\sbar)}\leq \Cbar (C_1\vep)^2 s^{-3/2+k\delta}, 
$
and we conclude that
\bel{pr1 prop bootstrap}
E_m(s,\del^IL^J u)^{1/2}\leq \Cbar C_0\vep + \Cbar (C_1\vep)^2.
\ee
When $N-3\leq|I|+|J|\leq N$ and $|J| = k$, by \eqref{ineq lem energy higher 1}
\bel{pr2 prop bootstrap}
\aligned
E_m(s,\del^IL^J u)^{1/2}\leq \, &  \Cbar C_0\vep +  \Cbar (C_1\vep)^2\int_2^s \sbar^{-1+k\delta} \, d\sbar
\\
\leq \, & \Cbar C_0\vep + \Cbar (C_1\vep)^2s^{k\delta}.
\endaligned
\ee

For the energy estimates on $v$, we apply $\del^IL^J$ to the Klein-Gordon equation in \eqref{eq main} and obtain
$$
-\Box \del^IL^J v + H^{\alpha\beta}u\del_\alpha\del_\beta\del^IL^J v + c^2\del^IL^J v = [H^{\alpha\beta}u\del_\alpha\del_\beta,\del^IL^J]v.
$$
Then by \eqref{ineq energy KG}, and \eqref{ineq lem energy coef 1} (with $\kappa =1/2$), we find 
$$
E_{m,c}(s,\del^IL^J v)^{1/2}
\leq \kappa^2 E_{m,c}(2,\del^IL^J v)^{1/2} + \kappa^2 \int_2^s \| f \|_{L_f^2(\Hcal_\sbar)} \, d\sbar
+ \kappa^2\int_2^s\|M(\sbar)\|_{L_f^2(\Hcal_s)} \, d\sbar.
$$
When $|I|+|J|\leq N-4$, we rely \eqref{ineq lem energy lower 2} and \eqref{ineq lem energy coef 2} and observe that
\be
\label{pr3 prop bootstrap}
\aligned
E_{m,c}(s,\del^IL^J v)^{1/2} \leq \, &  \Cbar C_0\vep + \Cbar (C_1\vep)^2\int_2^s \sbar^{-1+k\delta} \, d\sbar
\\
\leq \, & \Cbar C_0\vep + \Cbar (C_1\vep)^2s^{k\delta}.
\endaligned
\ee
When $N-3\leq |I|+|J|\leq N$, we apply \eqref{ineq lem energy higher 2} and \eqref{ineq lem energy coef 2} and observe that
\be
\label{pr4 prop bootstrap}
\aligned
E_{m,c}(s,\del^IL^J v)^{1/2} \leq \, & \Cbar C_0\vep + \Cbar (C_1\vep)^2\int_2^s \sbar^{-1/2+k\delta} \, d\sbar
\\
\leq \, & \Cbar C_0\vep +\Cbar (C_1\vep)^2s^{1/2+k\delta}.
\endaligned
\ee
Finally, by choosing $C_1\geq 4 \Cbar C_0$ and $\vep\leq (4 \Cbar C_1)^{-1}$, \eqref{pr1 prop bootstrap}--\eqref{pr4 prop bootstrap} lead to \eqref{ineq energy assumption'}.
  

 \section*{Acknowledgements}

The first author (PGLF) gratefully acknowledges financial support from the {\sl Simons Center for Geometry and Physics} (SCGP), Stony Brook University, during the one-month Program ``Mathematical Method in General Relativity', organized by M. Anderson, S. Klainerman, P.G. LeFloch, and J. Speck in January 2015. This paper was completed when PGLF enjoyed the hospitality of the {\sl Courant Institute of Mathematical Sciences} (CIMS), New York University. The authors were partially supported by the Agence Nationale de la Recherche through grant ANR SIMI-1-003-01. 


\end{document}